%% file: neurips_2023.tex
\documentclass{article}

\input{heading}

\input{def}

\usepackage[numbers]{natbib}

\usepackage[final]{neurips_2023}

\usepackage[utf8]{inputenc} 
\usepackage[T1]{fontenc}    
\usepackage{hyperref}       
\usepackage{url}            
\usepackage{booktabs}       
\usepackage{amsfonts}       
\usepackage{nicefrac}       
\usepackage{microtype}      
\usepackage{xcolor}         

\title{First Order Methods with Markovian Noise: from Acceleration to Variational Inequalities}

\author{%
  Aleksandr Beznosikov \\
  Innopolis University, Skoltech, MIPT, Yandex
  \And
  Sergey Samsonov\\
  HSE University
  \And
  Marina Sheshukova\\
  HSE University
  \And
  Alexander Gasnikov\\
  MIPT, Skoltech, IITP RAS
  \And
  Alexey Naumov\\
  HSE University
  \And
  Eric Moulines\\
  Ecole polytechnique
}

\begin{document}

\maketitle

\begin{abstract}
This paper delves into stochastic optimization problems that involve Markovian noise.
We present a unified approach for the theoretical analysis of first-order gradient methods for stochastic optimization and variational inequalities. Our approach covers scenarios for both non-convex and strongly convex minimization problems. To achieve an optimal (linear) dependence on the mixing time of the underlying noise sequence, we use the randomized batching scheme, which is based on the multilevel Monte Carlo method. Moreover, our technique allows us to eliminate the limiting assumptions of previous research on Markov noise, such as the need for a bounded domain and uniformly bounded stochastic gradients. Our extension to variational inequalities under Markovian noise is original. Additionally, we provide lower bounds that match the oracle complexity of our method in the case of strongly convex optimization problems.
\end{abstract}
\section{Introduction}
\vspace{-3mm}
Stochastic gradient methods are an essential ingredient for solving various optimization problems, with a wide range of applications in various fields such as machine learning \cite{goodfellow2014generative,GoodBengCour16}, empirical risk minimization problems \cite{van2000asymptotic}, and reinforcement learning \cite{Sutton1998,schulman2015trust,deeprl}. Various stochastic gradient descent methods (SGD) and their accelerated versions \cite{nesterov_accelerated,ghadimi2013stochastic} have been extensively studied under different statistical frameworks \cite{dieuleveut2017harder, vaswani2019fast}. The standard assumption for stochastic optimization algorithms is to consider independent and identically distributed noise variables. However, the growing usage of stochastic optimization methods in reinforcement learning \cite{bhandari2018finite, srikant2019finite, durmus2021stability} and distributed optimization \cite{lopes_sayed07, dimakis07, mao20} has led to increased interest in problems with Markovian noise. Despite this, existing theoretical works that consider Markov noise have significant limitations, and their analysis often results in suboptimal finite-time error bounds.
\par 
Our research aims to fill the gap in the existing literature on the first-order Markovian setting. By focusing on uniformly geometrically ergodic Markov chains, we obtain finite-time complexity bounds for achieving $\varepsilon$-accurate solutions that scale linearly with the mixing time of the underlying Markov chain. Our approach is based on careful applications of randomized batch size schemes and provides a unified view on both non-convex and strongly convex minimization problems, as well as variational inequalities.

\vspace{-1mm}

\textbf{Our contributions.} Our main contributions are the following:
\vspace{-1mm}

$\diamond$ \textbf{Accelerated SGD.} We provide the first analysis of SGD, including the Nesterov accelerated SGD method, with Markov noise without the assumption of bounded domain and uniformly bounded stochastic gradient estimates. Our results are summarised in \Cref{tab:comparison0} and \Cref{sec:acc-method} and cover both strongly convex and non-convex scenarios. Our findings for non-convex minimization problems complement the results obtained in \cite{dorfman2022adapting}.
\vspace{-1mm}

$\diamond$ \textbf{Lower bounds.} In \Cref{sec:lower_bounds_main} we give the lower bounds showing that the presence of mixing time in the upper complexity bounds is not an artefact of the proof. This is consistent with the results reported in \cite{nagaraj2020least}. 

$\diamond$ \textbf{Extensions.} In \Cref{sec:vi} we provide, as far as we know, the first analysis for variational inequalities with general stochastic Markov oracle, arbitrary optimization set, and arbitrary composite term. Our finite-time performance analysis provides complexity bounds in terms of oracle calls that scale linearly with the mixing time of the underlying chain, which is an improvement over the bounds obtained in \cite{wang2022stability} for the Markov setting.
\vspace{-4mm}

\paragraph{Related works.} Next, we briefly summarize the related works.
\vspace{-1mm}

$\diamond$ \textbf{Stochastic gradient methods.} 
Numerous research papers have reported significant improvements achieved by accelerated methods for stochastic optimization with stochastic gradient oracles involving independent and identically distributed (i.i.d.) noise. These methods have been extensively studied in theory \cite{hu2009accelerated,cotter2011better,devolder2011stochastic,lan2012optimal,lin2014smoothing,dvurechensky2016stochastic,gasnikov2018universal,vaswani2019fast,taylor2019stochastic,aybat2019universally,gorbunov2021near,woodworth2021even} and have shown practical success \cite{kingma2014adam, pmlr-v28-sutskever13}. The finite-time analysis of first-order methods in i.i.d. noise settings has been extensively studied by many authors, as discussed in \cite{lan20} and references therein. In \Cref{tab:comparison0} we include only some important results because i.i.d. setting is not in the interest of this paper.
\par 
While the literature on i.i.d. noise is extensive, existing research on the first-order Markovian setting is relatively sparse. In this study, we focus on Markov chains that are uniformly geometrically ergodic, and we refer the reader to \Cref{sec:main_res} for detailed definitions. We note that the complexity bounds which scale linearly with the mixing time of the underlying general Markov chain are currently available only for general convex and non-convex minimization problems. Namely, \cite{duchi2012ergodic} has investigated a version of the ergodic mirror descent algorithm that yields optimal convergence rates for Lipschitz, general convex and non-convex problems.
Recently, \cite{dorfman2022adapting} proposed a random batch size algorithm that adapts to the mixing time of the underlying chain for non-convex optimization with a compact domain. In particular, \cite[Theorem~4.3]{dorfman2022adapting} yields optimal complexity rates in terms of the number of oracle calls required for non-convex problems, which is consistent with the results obtained in \cite{duchi2012ergodic}. Unlike previous studies, this method is insensitive to the mixing time of the noise sequence.
\par 
For the general case of Markovian noise the finite-time analysis of non-accelerated SGD-type algorithms was carried out in \cite{sun2018markov} and \cite{doan23}. However, \cite{sun2018markov} heavily relies on the bounded domain assumption and uniformly bounded stochastic gradient oracles, while its bound in \cite[Theorem~5]{sun2018markov} has a suboptimal dependence on the mixing time of the underlying chain, see \Cref{tab:comparison0}. Additionally, \cite{sun2018markov} does not cover the strongly convex setting. On the other hand, \cite{doan23} covers both non-convex and strongly convex settings, but the bounds of \cite[Theorem~1]{doan23} has terms that are \emph{exponential} in the mixing time, and a careful examination reveals suboptimal dependence on the initial condition for strongly convex problems when SGD is applied.
\par 
In the study of Nesterov-accelerated SGD with Markovian noise, the authors of \cite{doan2020convergence} considered the use of a batch size of 1 and achieved a rate of forgetting the initial condition that matches that of the i.i.d. noise setting. However, their result is suboptimal in terms of the variance terms in both non-convex and strongly convex settings, as detailed in \Cref{tab:comparison0}. We emphasize that the case of unbounded gradient oracles with Markov noise is not treated in contrast to the i.i.d. setup
\cite{vaswani2019fast,liu2018accelerating}.
\par 
The above papers deal with general Markovian noise optimization. But there are also results that deal with Markovian stochasticity with a finite state space. Here we can highlight the work \cite{even2023stochastic}, where the author gives quite extensive results and achieves linear scaling by mixing time in the non-convex as well as strongly convex cases.
Recently, numerous papers have appeared dealing with the special scenario of distributed optimization \cite {pmlr-v162-sun22b}. 
\cite{wang2022stability} investigates the generalization and stability of Markov SGD with special attention to the excess variance guarantees. We note that first, these algorithms only need to deal with a very special case of Markov gradients, and second, the corresponding dependence on the mixing time of the Markov chain is again quadratic. At the same time, there exist particular results, e.g. \cite{nagaraj2020least}, which provide a lower bound for the particular finite sum problems in the Markovian setting.
\par 
$\diamond$ \textbf{Variational inequalities.} Variational inequalities \cite{facchinei2007finite} have been an active area of research in applied mathematics for more than half a century \cite{NeumannGameTheory1944, HarkerVIsurvey1990,scutari2010vi}. VI cover important special cases, e.g., minimization over a convex domain, saddle point or min-max and fixed point problems. computational game theory \cite{facchinei2007finite}, robust \citep{BenTal2009:book} and nonsmooth \citep{nesterov2005smooth,nemirovski2004prox} optimization, supervised \cite{Thorsten, bach2011optimization} and unsupervised \cite{NIPS2004_64036755, bach2008convex} learning, image denoising \citep{esser2010general,chambolle2011first}. In the last 5 years, variational inequalities and their special cases have attracted much interest in the machine learning community due to new connections to reinforcement learning \cite{Omidshafiei2017:rl,Jin2020:mdp}, adversarial training \cite{Madry2017:adv}, and GANs \cite{daskalakis2017training,gidel2018variational,mertikopoulos2018optimistic,chavdarova2019reducing,pmlr-v89-liang19b,peng2020training}.

Variational inequalities (VI) and saddle point problems have their own well-established theory and methods. Unlike minimization problems, solving variational inequalities doesn't rely on (accelerated) gradient descent. Instead, the extragradient method \cite{korpelevich1976extragradient}, various modified versions \cite{nemirovski2004prox, hsieh2019convergence}, or similar techniques \cite{doi:10.1137/S0363012998338806} are recommended as the basic and theoretically optimal methods.
While deterministic methods have long been used for solving variational inequalities, stochastic methods have gained importance only in the last 15 years, following pioneering works by \cite{4610024, juditsky2011solving}.
We summarise the  results on methods for stochastic variational inequalities with the Lipschitz operator and smooth stochastic saddle point problems in \Cref{tab:comparison1}. 
The number of papers dealing with stochastic VIs and saddle point problems is small compared to those dealing with stochastic optimization, we include in  \Cref{tab:comparison1} papers with the i.i.d. noise (which we do not do for stochastic optimization).  The only competing work dealing with Markovian noise in saddle point problems consider the finite sum problem and thus the finite Markov chain \cite{wang2022stability}, therefore we do not include it in \Cref{tab:comparison1}. Moreover, the results from \cite{wang2022stability}
has much worse oracle complexity guarantees $\mathcal{O}(\taumix^2/\varepsilon^2)$ in terms of $\taumix$.
There are more papers dealing with stochastic finite-sum variational inequalities or saddle point problems, but in the i.i.d. setting \cite{chavdarova2019reducing, palaniappan2016stochastic, NEURIPS2020_0cc6928e, alacaoglu2022stochastic, beznosikov2023stochastic}. We also do not consider in  \Cref{tab:comparison1} because of the difference in the stochastic oracle structure.
It is important to note that, unlike most previous works, we consider the most general formulation of VI itself for an arbitrary optimization set and composite term.
\vspace{-4mm}
\paragraph{Notations and definitions.} Let $(\Zset,\metricz)$ be a complete separable metric space endowed with its Borel $\sigma$-field $\Zsigma$. Let $(\Zset^{\nset},\Zsigma^{\otimes \nset})$ be the corresponding canonical process. Consider the Markov kernel $\MKQ$ defined on $\Zset \times \Zsigma$, and denote by $\PP_{\xi}$ and $\E_{\xi}$ the corresponding probability distribution and the expected value with initial distribution $\xi$. Without loss of generality, we assume that $(Z_k)_{k \in \nset}$ is the corresponding canonical process. By construction, for any $A \in \Zsigma$, it holds that $\PP_{\xi}(Z_k \in A | Z_{k-1}) = \MKQ(Z_{k-1},A)$, $\PP_{\xi}$-a.s. If $\xi = \delta_{z}$, $z \in \Zset$, we write $\PP_{z}$ and $\E_{z}$ instead of $\PP_{\delta_{z}}$ and $\E_{\delta_{z}}$, respectively. For $x^{1},\ldots,x^{k}$ being the iterates of any stochastic first-order method, we denote $\mathcal{F}_{k} = \sigma(x^{j}, j \leq k)$ and write $\E_{k}$ as an alias for $\E[\cdot | \mathcal{F}_{k}]$. We also write $\nsets := \nset \setminus \{0\}$. For the sequences $(a_n)_{n \in \nset}$ and $(b_n)_{n \in \nset}$ we write $a_n \lesssim b_n$ if there exists a constant $c$ such that that $a_n \leq c b_{n}$ for all $n \in \nset$.

\vspace{-5mm}
\begin{table}[h!]
\renewcommand{\arraystretch}{1.5}
    \centering
    \scriptsize
\captionof{table}{This table summarizes our results on  first-order method with Markovian noise. The columns of the table indicate whether the authors consider optimization over bounded domain, potentially unbounded gradients, and whether or not they assume additional restrictions on the Markovian noise (finite state space or reversibility). For ease of comparison we provide the respective results on SGD and ASGD (accelerated SGD) in the i.i.d. setting.}
    \label{tab:comparison0}   
    \scriptsize
\resizebox{\linewidth}{!}{
  \begin{threeparttable}
    \begin{tabular}{|c|c|c|c|c|c|c|c|}
    \cline{3-4}
    \multicolumn{2}{c|}{} & \multicolumn{2}{c|}{\textbf{Unbounded}} & \multicolumn{4}{c}{} \\
    \cline{2-8}
    \multicolumn{1}{c|}{} & \textbf{Method} & \textbf{Domain}  & 
    \begin{tabular}{@{}c@{}}
    \textbf{Gradient}
    \\[-2mm]
    \textbf{noise}
    \end{tabular}
    & 
    \begin{tabular}{@{}c@{}}
    \textbf{General}
    \\[-2mm]
    \textbf{MC}
    \end{tabular} 
    &\textbf{Acceleration} 
    &
    \begin{tabular}{@{}c@{}}
    \textbf{Oracle complexity}
    \\[-2mm]
    \textbf{(Smooth and non-convex)}
    \end{tabular}
    & 
    \begin{tabular}{@{}c@{}}
    \textbf{Oracle complexity}
    \\[-2mm]
    \textbf{(Smooth and strongly convex)}
    \end{tabular} 
    \\
    \hline
    \multirow{2}{*}{\rotatebox[origin=c]{90}{\textbf{i.i.d. }}} & \texttt{SGD} \cite{10.1214/aoms/1177729586, bachmoulines2011, ghadimi2016mini}  & \cmark  & \xmark  & N/A & \xmark  & $\mathcal{\tilde O} \left(L (f(x^0) - f(x^*)) \left[\frac{1}{\varepsilon^2}+ \frac{\cmax^2}{ \varepsilon^4}\right]\right)$ & $\mathcal{\tilde O}\left(\frac{L}{\mu}\log\frac{\| x^0 - x^*\|^2}{\varepsilon} + \frac{\sigma^2}{\mu^2\varepsilon}\right)$
    \\
    \cline{2-8}
    & \texttt{ASGD} \cite{vaswani2019fast,chen2020convergence} \tnote{{\color{blue}(1)}} & \cmark & \cmark & N/A & \cmark & $\mathcal{\tilde O} \left( L (f(x^0) - f(x^*)) \left[\frac{1 + \dmax^2}{\varepsilon^2}+ \frac{\cmax^2}{ \varepsilon^4}\right]\right)$ & $\mathcal{\tilde O}\left(\left(1 + \delta^2\right)\sqrt{\frac{L}{\mu}}\log\frac{\| x^0 - x^*\|^2}{\varepsilon} + \frac{\sigma^2}{\mu^2\varepsilon}\right)$
    \\
    \hline
    \hline
    \multirow{8}{*}{\rotatebox[origin=c]{90}{\textbf{Markovian \quad \quad \quad \quad}}} &
    \texttt{EMD} \cite{duchi2012ergodic}\tnote{{\color{blue}(2)}} & \xmark & \xmark & \cmark & \xmark & $\mathcal{\tilde O}\left( \frac{\taumix G^{2} D^{2}}{\varepsilon^{4}}\right)$ & \xmark 
    \\
    \cline{2-8}
    & \texttt{MC SGD} \cite{sun2018markov}\tnote{{\color{blue}(3)}} & \cmark & \xmark & \xmark & \xmark & $\mathcal{\tilde O}\left( h(G,L) \left(\tfrac{\taumix}{\varepsilon^{2}}\right)^{1/(1-q)}\right)$ & \xmark 
    \\
    \cline{2-8}
    &\texttt{MC SGD} \cite{doan23}\tnote{{\color{blue}(4)}} & \cmark & \cmark & \cmark & \xmark & $\mathcal{\tilde O}\left( \frac{\taumix L^2 (1 + \norm{x^{*}}^2 + \norm{x^{0}-x^*}^2)}{\varepsilon^{4}}\right)$ & $\mathcal{\tilde O} \left(e^{\taumix (L/\mu)^{2}}\left[ h(\tfrac{L}{\mu}) \log\frac{\norm{x^0-x^*}^2}{\varepsilon} + \frac{\taumix^{2} L^{2}(1 + \norm{x^*}^2)}{\mu^2 \varepsilon}\right]\right)$ \\
    \cline{2-8}
    & \texttt{ASGD}\cite{doan2020convergence}\tnote{{\color{blue}(5)}} & \xmark & \xmark & \xmark & \cmark & $\mathcal{\tilde O}\left( \tfrac{1}{\varepsilon^{4}}\left[B^{2} + G^6(L^2 \taumix^2 + 1)\right]\right)$ & $\mathcal{\tilde O} \left( \sqrt{\frac{L}{\mu}}\frac{\| x^0 - x^* \|^2}{\varepsilon^{1/2}} + \frac{\taumix^{2}(G^2 + \mu G D + \mu L D^2)}{\mu^{2} \varepsilon} \right)$ \\
    \cline{2-8}
    &\texttt{MAG} \cite{dorfman2022adapting}\tnote{{\color{blue}(6)}} & \cmark & \xmark & \cmark & \xmark & $\mathcal{\tilde O}\left( \frac{\taumix (G + L + B)^2G^2}{\varepsilon^{4}}\right)$ & \xmark  \\
    \cline{2-8}
    &\texttt{MC SGD} \cite{even2023stochastic} (Sec. 5.1) \tnote{{\color{blue}(7)}} & \cmark &\xmark & \xmark &\xmark & $\mathcal{O}\left( \frac{\tau (L(f(x^0) - f(x^*)) + \sigma^2)}{\varepsilon^2} + \frac{\tau (L(f(x^0) - f(x^*)) + \sigma^2) \sigma^2}{\varepsilon^4}\right)$ & $\mathcal{O} \left( \frac{\tau L}{\mu} \log \frac{(f(x^0) - f(x^*))/\mu + \sigma^2/(\mu L)}{\varepsilon} + \frac{\tau \sigma^2}{\mu^2 \varepsilon}\right)$ \\
    \cline{2-8}
    &\texttt{MC SGD} \cite{even2023stochastic} (Sec. 5.2) \tnote{{\color{blue}(8)}} & \cmark &\xmark & \xmark &\xmark & \xmark & $O \left( \frac{L}{\mu} \log \frac{\|x^0 - x^* \|^2}{\varepsilon} + \frac{L \tau \sigma^2_*}{\mu^3 \varepsilon}\right)$ \\
    \cline{2-8}
    & \cellcolor{bgcolor2}{\texttt{RASGD} (ours)} & \cellcolor{bgcolor2}{\cmark} & \cellcolor{bgcolor2}{\cmark} & \cellcolor{bgcolor2}{\cmark} & \cellcolor{bgcolor2}{\cmark} & \cellcolor{bgcolor2}{$\mathcal{\tilde O} \left(\taumix L (f(x^0) - f(x^*)) \left[\frac{1 + \dmax^2}{\varepsilon^2}+ \frac{\cmax^2}{ \varepsilon^4}\right]\right)$} & \cellcolor{bgcolor2}{$\mathcal{\tilde O} \left( \taumix \left[ (1 + \dmax^2) \sqrt{\frac{L}{\mu}} \log \frac{\norm{x^0 - x^*}^2}{\varepsilon} + \frac{\cmax^2}{\mu^2\varepsilon} \right] \right)$}
    \\
    \hline
    \end{tabular}   
    \begin{tablenotes}
    {\small  
    \item [] {\em notation:} $\mu$ and $L$ are as in \Cref{as:lipsh_grad} and \Cref{as:strong_conv},$G = \sup_{x,z} \norm{\nabla F(x,z)}$. Note that $G \geq L$ and $G^2 \geq \cmax^2$ under \Cref{as:bounded_markov_noise_UGE}. We also set $B = \sup_x |f(x)|$; $x^0$ - starting point, $x^*$ - solution, $\mathcal{D}$ - optimisation domain; $D = \sup_{x \in \mathcal{D}} \|x - x^*\|$, $\cmax$ and $\dmax$ - stochasticity parameters (see \Cref{as:bounded_markov_noise_UGE}); $\cmax_*$ - stochasticity parameter in $x^*$ ; $\taumix$ - mixing time of the chain (see \Cref{as:Markov_noise_UGE}), $\varepsilon$ - accuracy of the solution, measured as $\EE[\|\nabla f(x)\|^2] \lesssim \varepsilon^2$ for non-convex problems and $\EE[\|x - x^*\|^2] \lesssim \varepsilon$ for the strongly convex ones. Functions $h(L/\mu)$ and $h(G,L)$ stands for an implicit dependence of the respective parameters.
    \item[] \tnote{{\color{blue}(1)}}
    gives results with stepsize as a parameter, we choose it the close way as in our \Cref{cor:sample_complexity_accelerated}. \tnote{{\color{blue}(2)}} 
    covers more general noise setting, then just Markovian. \tnote{{\color{blue}(3)}} for general state-space Markov noise the analysis of \cite{sun2018markov} requires reversibility. Parameter $q \in (1/2;1)$ refers to the step size $\sim 1/k^{q}$. \tnote{{\color{blue}(4)}} The fluctuation terms in \cite[Theorem~1,3]{doan23} contain hidden dependence on the initial error and $\norm{x^*}$ in the fluctuation terms, making the result comparison complicated. They also contain hidden factors, which are exponential in $C  = \taumix / \log{4}$ in the notations of our paper. Moreover, the analysis of \cite{doan23} requires that $F(x,z)$ is Lipschitz w.r.t. $x$ for any $z \in \Zset$. \tnote{{\color{blue}(5)}} 
    considers Markovian noise with finite state space and a specifically decreasing step size. Moreover, in the proof of \cite[Theorem~3]{doan2020convergence} (equations $(64)-(66)$) the authors lost the factor $C^{2}$, with $C = \taumix / \log{4}$. The result in the table accounts for this lost factor. \tnote{{\color{blue}(6)}} 
    considers the adaptive tuning of batch size, which is oblivious to $\taumix$. 
    \tnote{{\color{blue}(7)}} considers Markov noise with finite state space and additionally assumes that all stochastic realization $F(\cdot, Z)$ are $L$-smooth.
    \tnote{{\color{blue}(8)}} considers Markovian noise with finite state space, $\sigma_*$ bounds noise only in $x^*$, but additionally assumes that all stochastic realization $F(\cdot, Z)$ are $L$-smooth and $\mu$-strongly convex.
    }
\end{tablenotes}    
    \end{threeparttable}
}
\end{table}

\begin{table}[h!]
\renewcommand{\arraystretch}{1.5}
    \centering
    \scriptsize
\captionof{table}{This table summarizes the findings on methods for solving stochastic (strongly) monotone variational inequalities with a Lipschitz operator and (un)bounded stochasticity. The columns of the table indicate whether the authors consider variational inequalities or only certain saddle point problems, the arbitrariness of the sets, and the use of additional composite terms. The columns on stochasticity provide information on the assumptions made with respect to the stochastic operator, such as bounded variance and the Markovian noise setting. Note that with the exception of our work, all other studies assume the independent noise.} 
    \label{tab:comparison1}   
    \scriptsize
\resizebox{\linewidth}{!}{
  \begin{threeparttable}
    \begin{tabular}{|c|c|c|c|c|c|c|c|}
    \cline{3-7}
    \multicolumn{2}{c|}{} & \multicolumn{3}{c|}{\textbf{Statement}} & \multicolumn{2}{c|}{\textbf{Stochasticity}} & \multicolumn{1}{c}{}
    \\
    \cline{2-8}
    \multicolumn{1}{c|}{} &
    \textbf{Method} & \textbf{VI?} & \textbf{Any set?} & \textbf{Composite?} & \textbf{Unbounded?} & \textbf{Markovian?}  & \textbf{Oracle complexity} 
    \\
    \hline
    \multirow{8}{*}{\rotatebox[origin=c]{90}{\textbf{Strongly monotone \quad \quad \quad }}} &
    \texttt{SPEG} \cite{gidel2018variational, hsieh2019convergence} & \cmark & \cmark & \xmark & \xmark & \xmark & $\mathcal{\tilde O}\left( \frac{L}{\mu} \log \frac{\| x^0 - x^*\|^2}{\varepsilon} + \frac{\sigma^2}{\mu^2 \varepsilon}\right)$ \tnote{{\color{blue}(1)}}
    \\
    \cline{2-8}
    &\texttt{SEG} \cite{kannan2019optimal} & \cmark & \cmark & \xmark & \xmark & \xmark & $\mathcal{\tilde O}\left( \frac{\| x^0 - x^* \|^2 }{ \varepsilon} +  \frac{B^2 + \sigma^2 + (B + \sigma)(1 + L D)}{ \sigma^2\varepsilon} \right)$
    \\
    \cline{2-8}
    &\texttt{SS-SEG} \cite{mishchenko2020revisiting, gorbunov2022stochastic} & \cmark & \cmark & \cmark & \cmark & \xmark & $\mathcal{\tilde O}\left( \frac{L}{\mu} \log \frac{\| x^0 - x^*\|^2}{\varepsilon} + \frac{\sigma^2_*}{\mu^2 \varepsilon} \right)$
    \\
    \cline{2-8}
    &\texttt{SEG} \cite{beznosikov2020distributed} & \xmark & \cmark & \xmark & \xmark & \xmark & $\mathcal{\tilde O}\left( \frac{L}{\mu} \log \frac{\| x^0 - x^*\|^2}{\varepsilon} + \frac{\sigma^2}{\mu^2 \varepsilon}\right)$
    \\
   \cline{2-8}
    &\texttt{DSEG} \cite{hsieh2020explore} & \cmark & \xmark & \xmark & \cmark & \xmark & $\mathcal{O}\left( \left[\frac{L^2 \sigma^2}{\mu^4 \varepsilon} \right]^3 \right)$ \tnote{{\color{blue}(2)}}
    \\
    \cline{2-8}
    &\texttt{UEG} \cite{gorbunov2022stochastic} & \cmark & \xmark & \xmark & \cmark & \xmark & $\mathcal{O}\left( \left(\frac{L + \Delta}{\mu} + \frac{\Delta^2}{\mu^2} \right) \log \frac{\| x^0 - x^*\|^2}{\varepsilon} + \frac{\sigma^2}{\mu^2 \varepsilon}\right)$
    \\
    \cline{2-8}
    &\texttt{SGDA} \cite{beznosikov2023stochastic} & \cmark & \cmark & \cmark & \xmark & \xmark & $\mathcal{O}\left( \frac{L^2}{\mu^2} \log \frac{\| x^0 - x^*\|^2}{\varepsilon} + \frac{\sigma^2}{\mu^2 \varepsilon}\right)$ \tnote{{\color{blue}(3)}}
    \\
    \cline{2-8}
    &\cellcolor{bgcolor2}{\texttt{REG} (ours)} & \cellcolor{bgcolor2}{\cmark} & \cellcolor{bgcolor2}{\cmark} & \cellcolor{bgcolor2}{\cmark} & \cellcolor{bgcolor2}{\cmark} & \cellcolor{bgcolor2}{\cmark} & \cellcolor{bgcolor2}{$\mathcal{\tilde O}\left( \taumix \cdot \left[\left(\frac{L + \Delta}{\mu} + \frac{\Delta^2}{\mu^2} \right) \log \frac{\| x^0 - x^*\|^2}{\varepsilon} + \frac{\sigma^2}{\mu^2 \varepsilon}\right] \right)$}
    \\
    \hline
    \hline
    \multirow{6}{*}{\rotatebox[origin=c]{90}{\textbf{Monotone \quad \quad \quad }}} &
    \texttt{SMP} \cite{juditsky2011solving} & \cmark & \cmark & \xmark & \xmark & \xmark & $\mathcal{O}\left( \frac{L D^2}{\varepsilon} + \frac{\sigma^2\Delta^2}{\varepsilon^2}\right)$
    \\
    \cline{2-8}
    &\texttt{VR-SEG} \cite{doi:10.1137/15M1031953} & \cmark & \cmark & \xmark & \cmark & \xmark & $\mathcal{O}\left( \frac{(\sigma + \Delta)^8 D^4}{\varepsilon^2} + \frac{D^4}{\varepsilon^2}\right)$
    \\
    \cline{2-8}
    &\texttt{IPM} \cite{iusem2019incremental} & \cmark & \cmark & \cmark & \cmark & \xmark & $\mathcal{O}\left( \mathcal{\tilde O}\left( \frac{L^4 D^4}{\varepsilon^2} + \frac{\sigma_*^2 D^4}{\varepsilon^2}\right) \right)$
    \\
    \cline{2-8}
    &\texttt{SS-SEG} \cite{mishchenko2020revisiting} & \cmark & \cmark & \cmark & \cmark & \xmark & $\mathcal{\tilde O}\left( \frac{L^2 D^4}{\varepsilon} + \frac{\sigma_*^4}{L^2 \varepsilon^2}\right)$
    \\
    \cline{2-8}
    &\texttt{SEG} \cite{beznosikov2020distributed} & \xmark & \cmark & \xmark & \xmark & \xmark & $\mathcal{O}\left( \frac{L D^2}{\varepsilon} + \frac{\sigma^2\Delta^2}{\varepsilon^2}\right)$
    \\
    \cline{2-8}
    &\cellcolor{bgcolor2}{\texttt{REG} (ours)} & \cellcolor{bgcolor2}{\cmark} & \cellcolor{bgcolor2}{\cmark} & \cellcolor{bgcolor2}{\cmark} & \cellcolor{bgcolor2}{\cmark} & \cellcolor{bgcolor2}{\cmark} & \cellcolor{bgcolor2}{$\mathcal{\tilde O}\left( \taumix \cdot \left[\frac{L D^2}{\varepsilon} + \frac{\sigma^2 D^2}{\varepsilon^2} + \frac{\Delta^2 D^4}{\varepsilon^2} \right] \right)$}
    \\
    \hline
    \end{tabular}   
    \begin{tablenotes}
    {\small   
    \item [] {\em notation:} $\mu$ = constant of strong monotonicity of operator $F$, $L$ = Lipschitz constant of $F$, $B$ = uniform bound of $F$, $D$ = uniform bound of iterations $x^k$, $x^0$ = starting point, $x^*$ = solution, $\Delta$ and $\sigma$ = stochasticity parameters (see \Cref{as:bounded_markov_noise_UGE_op}, \cite{juditsky2011solving, gidel2018variational, hsieh2019convergence, kannan2019optimal, beznosikov2020distributed, beznosikov2023stochastic} take $\Delta = 0$), $\sigma_*$ = stochasticity parameter in $x^*$ (see \cite{mishchenko2020revisiting}), $\taumix$ = mixing time of the chain (see \Cref{as:Markov_noise_UGE}), $\varepsilon$ = accuracy of the solution.
    \item [] \tnote{{\color{blue}(1)}} give results with stepsize as a parameter, we choose it according to Section 3 from \cite{stich2019unified}.      
    \tnote{{\color{blue}(2)}} consider \Cref{as:bounded_markov_noise_UGE_op}, but do not provide explicit rates if $\Delta \neq 0$ (see also   \cite[Table~1]{gorbunov2022stochastic}). \tnote{{\color{blue}(3)}} consider the cocoercive case, for which in general $\ell = L^2/\mu$.    
    }
\end{tablenotes}    
    \end{threeparttable}
}
\vspace{-5mm}
\end{table}

\vspace{-6mm}
\section{Main results}
\label{sec:main_res}
\vspace{-3mm}
\textbf{Assumptions.} In this paper we study the minimization problem
\begin{equation}
\label{eq:erm}
\textstyle{
\min_{x \in \R^d} f(x) := \E_{Z \sim \pi} [F(x, Z)]
}\,,
\end{equation}
where the access to the function $f$ and its gradient is available only through the (unbiased) noisy oracle $F(x, Z)$ and $\nabla F(x, Z)$, respectively. In the following presentation we impose at least one of the following regularity constraint on the underlying function $f$ itself:
\begin{assumption}
\label{as:lipsh_grad}
The function $f$ is $L$-smooth on $\R^d$ with $L > 0$, i.e., it is differentiable and there is a constant $L > 0$ such that the following inequality holds for all $x,y \in \R^d$:
\begin{equation*}
    \label{L-smooth1}
    \norm{\nabla f(x) - \nabla f(y) } \leq L \norm{x-y }.
\end{equation*}
\end{assumption}

\begin{assumption}
\label{as:strong_conv}
The function $f$ is $\mu$-strongly convex on $\R^d$, i.e., it is continuously differentiable and there is a constant $\mu > 0$ such that the following inequality holds for all $x,y \in \R^d$:
\begin{equation}
 \label{eq:strong_conv}
 \textstyle{
 (\mu/2)\norm{x-y}^2 \leq f(x) - f(y) - \langle \nabla f(y), x-y \rangle
 }\,.
\end{equation}
\end{assumption}
\vspace{-2mm}
Next we specify our assumptions on the sequence of noise variables $\{Z_i\}_{i=0}^{\infty}$. We consider here the general setting of $\{Z_i\}_{i=0}^{\infty}$ being a time-homogeneous Markov chain. Such problems naturally arise in stochastic optimization. In the empirical risk minimization problems it naturally appears in the context of non-random minibatch choice. Indeed, a random choice of a batch number may lose to a non-random one, see \cite{mishchenko2020random, koloskova2023shuffle}. A wide range of problems dealing with Markovian noise is spawned by the reinforcement learning methods. The usual MDP setting falls naturally inside this paradigm, moreover, the analysis of non-tabular RL problems requires to deal with the general state-space Markov noise. Here the potential range of applications include the policy evaluation methods, such as the temporal difference methods \cite{sutton1988learning}, and policy optimization algorithms, such as policy gradient family, e.g. the celebrated REINFORCE algorithm \cite{williams1992simple}.

We denote by $\MKQ$ the Markov kernel corresponding to the sequence $\{Z_i\}_{i=0}^{\infty}$ and impose the following assumption on the mixing properties of $\MKQ$:

\begin{assumption}
\label{as:Markov_noise_UGE}
$\{Z_i\}_{i=0}^{\infty}$ is a stationary Markov chain on $(\Zset,\Zsigma)$ with Markov kernel $\MKQ$ and unique invariant distribution $\pi$. Moreover, $\MKQ$ is uniformly geometrically ergodic with mixing time $\taumix \in \nset$, i.e., for every $k \in \nset$,
\begin{equation}
\label{eq:drift-condition}
 \textstyle{
 \dobru{\MKQ^k} = \sup_{z,z' \in \Zset} (1/2) \norm{\MKQ^k(z, \cdot) - \MKQ^k(z',\cdot)}[\mathsf{TV}] \leq (1/4)^{\lfloor k / \taumix \rfloor}
 }\,.
\end{equation}
\end{assumption}
\vspace{-2mm}
The assumption \Cref{as:Markov_noise_UGE} is classical in the literature on optimization methods with Markovian noise and has been considered in particular in recent works \cite{sun2018markov,dorfman2022adapting,doan2020convergence}. In particular, this assumption covers finite state-space Markov chains with irreducible and aperiodic transition matrix considered in \cite{even2023stochastic}. Yet our definition of the mixing time $\taumix$ is more classical in the probability literature \cite{paulin_spectral}, and is slightly different from the one considered e.g. in \cite{even2023stochastic,mao20}. Next we specify our assumptions on stochastic gradient:

\begin{assumption}
\label{as:bounded_markov_noise_UGE}
For all $x \in \rset^{d}$ it holds that $\E_{\pi}[\nabla F(x, Z)] = \nabla f(x)$. Moreover, for all $z \in \Zset$ and $x \in \rset^{d}$ it holds that
\begin{equation}
\label{eq:growth_condition}
\norm{\nabla F(x, z) - \nabla f(x)}^{2} \leq \cmax^{2} + \dmax^{2} \| \nabla f(x) \|^{2}\,.
\end{equation}
\end{assumption}
\vspace{-2mm}
The assumption \Cref{as:bounded_markov_noise_UGE} resembles the strong growth condition \cite{vaswani2019fast}, which is classical for the overparametrized learning setup \cite{vaswani2019fast, DBLP:journals/corr/abs-1905-09997}. The main difference is that \Cref{as:bounded_markov_noise_UGE} concerns the almost sure bound in \eqref{eq:growth_condition}, which is unavoidable when dealing with uniformly geometrically ergodic Markovian noise \Cref{as:Markov_noise_UGE}. Note that it is possible that the quantity $\delta^2$ in \eqref{eq:growth_condition} is not instance-independent and scales with the ratio $L/\mu$ from \Cref{as:lipsh_grad}-\Cref{as:strong_conv} in the particular problems. With the assumptions \Cref{as:Markov_noise_UGE} and \Cref{as:bounded_markov_noise_UGE} we can prove the result on the mean squred error of the stochastic gradient estimate computed over batch size $n$ under arbitrary initial distribution. This result is summarized below in \Cref{lem:tech_markov}:
\begin{lemma}
\label{lem:tech_markov}
Assume \Cref{as:Markov_noise_UGE} and \Cref{as:bounded_markov_noise_UGE}. Then, for any $n \geq 1$ and $x \in \rset^{d}$, it holds that
\begin{equation}
\label{eq:var_bound_stationary}
\textstyle{
\E_{\pi}[\norm{n^{-1}\sum\nolimits_{i=1}^{n}\nabla F(x, Z_{i}) - \nabla f(x)}^2]
    \leq
    \frac{8 \taumix}{n} \left( \cmax^2 + \dmax^2 \| \nabla f(x) \|^2 \right)
}\,.
\end{equation}
Moreover, for any initial distribution $\xi$ on $(\Zset,\Zsigma)$, that
\begin{equation}
\label{eq:var_bound_any}
\textstyle{
\E_{\xi}[\norm{n^{-1}\sum\nolimits_{i=1}^{n}\nabla F(x, Z_i) - \nabla f(x)}^2] \leq \frac{C_{1} \taumix}{n} \left( \cmax^2 + \dmax^2 \| \nabla f(x) \|^2 \right)
}\,,
\end{equation}
where $C_{1} = 16(1 + \frac{1}{\ln^{2}{4}})$.
\end{lemma}
\vspace{-2mm}
\begin{proof}
We first prove \eqref{eq:var_bound_stationary}.  Note that due to \cite[Proposition~$3.4$]{paulin_spectral} the Markov kernel $\MKQ$ under \Cref{as:Markov_noise_UGE} admits a positive pseudospectral gap $\gamma_{ps} > 0$ such that
$1/{\gamma_{ps}} \leq 2\taumix$. Thus, applying the statement of \cite[Theorem~$3.2$]{paulin_spectral}, we get under \Cref{as:bounded_markov_noise_UGE} that
\[
\textstyle{
\E_{\pi}[\norm{n^{-1}\sum\nolimits_{i=1}^{n}\nabla F(x, Z_i) - \nabla f(x)}^2] \leq \frac{4 \E_{\pi}[\norm{\nabla F(x, Z_1) - \nabla f(x)}^2]}{n \gamma_{ps}} \leq \frac{8\taumix}{n}\left( \cmax^2 + \dmax^2 \| \nabla f(x) \|^2 \right)}\,.
\]
To prove the second part we use the maximal exact coupling construction and follow, e.g., \cite[Theorem~1]{moulines23Rosenthal}. The complete proof is given in \Cref{sec:proof:lem:tech_markov}.
\end{proof}
\vspace{-3mm}
The proof of \Cref{lem:tech_markov} simplifies the arguments in \cite[Lemma~4]{dorfman2022adapting} and allows us to obtain tighter values for the constants when dealing with the randomized batch size. Note that it is especially important to have the result for MSE under arbitrary initial distribution $\xi$, since in the proofs of our main results we will inevitably deal with the conditional expectations w.r.t. the previous iterate. We provide more details on the bias and variance of the Markov SGD gradients in the next section.

\vspace{-3mm}
\subsection{Accelerated method}
\label{sec:acc-method}
\vspace{-2mm}
We begin with a version of Nesterov accelerated SGD with randomized batch size, described in \Cref{alg:AGD_ASGD}.
Due to the unboundedness of the stochastic gradient variance (see \Cref{as:bounded_markov_noise_UGE}), using of the classical Nesterov accelerated method \cite[Section 2.2.]{nesterov2003introductory} does not give the desired result, it is necessary to introduce an additional momentum \cite{doi:10.1137/100802001, vaswani2019fast}. We use our own version, but partially similar to \cite{doi:10.1137/100802001, vaswani2019fast}.
The main feature of \Cref{alg:AGD_ASGD} is that the number of samples used during the $k$-th gradient computation scales as $2^{J_k}$, where $J_k$ comes from a truncated geometric distribution. The truncation parameter needs to be adopted (see \Cref{th:acc}) in order to control the computational complexity of the algorithm.    
{\tiny \begin{algorithm}[h!]
   \caption{\texttt{Randomized Accelerated GD}}
   \label{alg:AGD_ASGD}
\begin{algorithmic}[1]
\State {\bf Parameters:} stepsize $\gamma>0$, momentums $\theta, \eta, \beta, p$, number of iterations $\nbiter$, batchsize limit $\batchbound$
\State {\bf Initialization:} choose  $x^0  = x^0_f$
\For{$k = 0, 1, 2, \dots, \nbiter-1$}
    \State $x^k_g = \theta x^k_f + (1 - \theta) x^k$ \label{line_acc_1}
    \State Sample $\textstyle{J_k \sim \text{Geom}\left(1/2\right)}$
    \State
    \text{\small{ 
    $g^{k} = g^{k}_0 +
        \begin{cases}
        \textstyle{2^{J_k} \left( g^{k}_{J_k}  - g^{k}_{J_k - 1} \right)}, & \text{ if } 2^{J_k} \leq \batchbound \\
        0, & \text{ otherwise}
        \end{cases}
    $ \quad with \quad
    $
    \textstyle{g^k_j = 2^{-j} B^{-1} \sum\nolimits_{i=1}^{2^j B} \nabla f(x^{k}_g, Z_{T^{k} + i})}
    $
    }}
    \State    $\textstyle{x^{k+1}_f = x^k_g - p\gamma g^k}$ \label{line_acc_2}
    \State    $\textstyle{x^{k+1} = \eta x^{k+1}_f + (p - \eta)x^k_f + (1- p)(1 - \beta) x^k +(1 - p) \beta x^k_g}$ \label{line_acc_3}
    \State    $\textstyle{T^{k+1} = T^{k} + 2^{J_{k}} B}$\label{line_counter}
\EndFor
\end{algorithmic}
\end{algorithm}}
\vspace{-1mm}

Randomized batch size allows for efficient \emph{bias} reduction in the stochastic gradient estimates and can be seen as a particular case of the so called multilevel MCMC \cite{glynn2014exact, giles08}. In the optimization context this approach was successfully used by \cite{dorfman2022adapting} for the non-convex problems. Indeed, this bias naturally appears under the Markovian stochastic gradients oracles. It is easy to see that, with the counter $T^{k}$ defined in \Cref{line_counter}, we have
\[
\textstyle{\E_{k}[\nabla F(x^{k},Z_{T^{k}+i})] \neq \nabla f(x^{k})}\,.
\]
Below we show how the bias of the gradient estimate scales with the truncation parameter $M$. The statement of \Cref{lem:expect_bound_grad} yields that the gradient estimates $g_k$ introduced above have the bias, which decreases \emph{quadratically} with $M$.
\begin{lemma}
\label{lem:expect_bound_grad}
Assume \Cref{as:Markov_noise_UGE} and \Cref{as:bounded_markov_noise_UGE}. Then for the gradient estimates $g^k$ from \Cref{alg:AGD_ASGD} it holds that $\E_k[g^k] = \E_k[g^{k}_{\lfloor \log_2 \batchbound \rfloor}]$. Moreover, 
\begin{align*}
&\textstyle{\E_k[\| \nabla f(x^k) - g^k\|^2] \lesssim \bigl(\taumix B^{-1}\log_2 \batchbound + \taumix^2 B^{-2}\bigr)(\cmax^2 + \dmax^2 \| \nabla f(x^k)\|^2) }\,, \\
&\textstyle{\| \nabla f(x^k) - \E_{k}[g^k]\|^2 \lesssim \taumix^2 \batchbound^{-2}B^{-2} (\cmax^2 + \dmax^2 \| \nabla f(x^k)\|^2)}\,.
\end{align*}
\end{lemma}
The proof and the statement with explicit constants are given in \Cref{sec:proof:lem:expect_bound_grad}. Note that the \Cref{lem:expect_bound_grad} is a natural counterpart of the deterministic bound \Cref{lem:tech_markov}. Moreover, it gives the idea of the trade-off between the parameters $B$ and $M$. Namely, the expected number of oracle calls to compute $g_k$ is $\mathcal{O}(B\log_2(M))$ with the bias scaling as $M^{-2}$. Thus the increase of $M$ drastically reduced the bias with only a logarithmic payment in variance. At the same time, gradient variance scales as $(\taumix/B)^2$, but the increase of $B$ is much more expensive for the computational cost of the whole procedure. Taking into account the considerations above, we can prove the following result:
\begin{theorem}
\label{th:acc}
Assume \Cref{as:lipsh_grad} -- \Cref{as:bounded_markov_noise_UGE}. Let problem \eqref{eq:erm}
be solved by \Cref{alg:AGD_ASGD}. Then for any $b \in \nsets$,  $\gamma \in (0; \tfrac{3}{4L}]$, and $\beta, \theta, \eta, p, \batchbound, B$ satisfying 
\begin{align*}
&\textstyle{p \simeq (1 + (  1 +  \gamma L) [\delta^2 \taumix b^{-1} + \dmax^2 \taumix^2 b^{-2}])^{-1}, \quad 
\beta \simeq \sqrt{p^2 \mu \gamma}, \quad\eta \simeq \sqrt{\frac{1}{\mu \gamma}}},
\\
&\textstyle{\theta \simeq \frac{p \eta^{-1} - 1}{\beta p \eta^{-1} - 1}, \quad  M \simeq \max\{2; \sqrt{p^{-1} (1 + p/\beta)}\}, \quad B = \lceil b \log_2 \batchbound \rceil}\,,
\end{align*}
it holds that
\begin{multline}
\label{eq:conv_guarant_accelerated}
\textstyle{
    \EE\left[\|x^{N} - x^*\|^2 + \frac{6}{\mu} (f(x^{N}_f) - f(x^*)) \right]
    \lesssim 
    \exp\bigl( - N\sqrt{\frac{p^2 \mu \gamma}{3}}\bigr) \bigl[\| x^0 - x^*\|^2 + \frac{6}{\mu} (f(x^0) - f(x^*)) \bigr]}
    \\
    \textstyle{+ \frac{p \sqrt{\gamma}}{\mu^{3/2}} \left( \sigma^2 \taumix b^{-1} + \cmax^2  \taumix^2 b^{-2} \right)}\,.
\end{multline}
\end{theorem}
The proof is provided in \Cref{sec:proof:th:strongly_convex_upper_bound}.
The result of \Cref{th:acc} can be rewritten as an upper complexity bound under an appropriate choice of the remaining free parameter $b$:
\begin{corollary}
\label{cor:sample_complexity_accelerated}
Under the conditions of \Cref{th:acc}, choosing $b = \taumix$ and $\gamma$ as 
\begin{align*}
\textstyle{
        \gamma \simeq \min\left\{ \frac{1}{L}; \frac{1}{p^2 \mu \nbiter^2} \ln \left( \max\left\{ 2; \frac{\mu^2 \nbiter [\| x^0 - x^*\|^2 + 6 \mu^{-1} (f(x^0_f) - f(x^*))]}{\cmax^2} \right\}\right) \right\}}\,,
\end{align*}
in order to achieve $\varepsilon$-approximate solution (in terms of $\EE[\|x - x^*\|^2] \lesssim \varepsilon$) it takes 
\begin{equation}
\label{eq:sample_complexity_accelerated}
\textstyle{
\mathcal{\tilde O} \left( \taumix \left[ (1 + \dmax^2) \sqrt{\frac{L}{\mu}} \log \frac{1}{\varepsilon} + \frac{\cmax^2}{\mu^2\varepsilon} \right] \right)
} \quad \text{oracle calls}\,.
\end{equation}
\end{corollary}
\vspace{-2mm}
The results of Corollary \ref{cor:sample_complexity_accelerated} are obtained with fixed parameters of the method. In Corollary \ref{cor:sample_complexity_accelerated} these parameters are selected a bit artificially, e.g., the stepsize $\gamma$ depends on the iteration horizon $N$. In Appendix \ref{sec:decreas} we show how one can similar results, but with a decreasing stepsize.

\textbf{Comparison.} 
Running the procedure above requires to know the mixing time $\taumix$. Estimating the mixing time from a single trajectory of the running Markov chain is known to be computationally hard problem, see e.g. \cite{wolfer2019estimating} and references therein. At the same time, methods, which share the same (optimal) linear scaling of the sample complexity w.r.t. the mixing time also share the same drawback as our method. In particular, it holds true for the EMD algorithm \cite{duchi2012ergodic}, SGD-DD algorithm \cite{nagaraj2020least}, and usual SGD with Markovian data \cite{even2023stochastic}. At the same time, in the non-convex scanario the paper \cite{dorfman2022adapting} is truly oblivious to mixing time, allowing to obtain sample complexity rates for non--convex problems, which are homogeneous w.r.t. $\taumix$ with AdaGrad-type learning rate. An interesting direction for the future work to suggest a procedure that would allow to generalize the results of \cite{dorfman2022adapting} to accelerated SGD setting.

It is possible that the sample complexity bound \eqref{eq:sample_complexity_accelerated} is worse than the respective bounds for non-accelerated SGD with Markov data, provided that $\delta^2$ grows quickly with $L/\mu$. At the same time, this drawback is shared by the classical results on learning under the strong growth condition, see e.g. \cite{vaswani2019fast}. As it is shown in \cite{liu2018accelerating}, the respective rates can be worse than the ones obtained by usual SGD even under the i.i.d. noise setting, see Appendix F.3 in \cite{liu2018accelerating}. Making the analysis of accelerated SGD ‘backward compatible’ w.r.t. the rates of usual SGD requires to perform analysis in terms of additional problem-specific quantities, see \cite{jain2018accelerating, liu2018accelerating}.

The closest equivalent of the result \Cref{cor:sample_complexity_accelerated} is given by \cite[Theorem~3]{doan2020convergence}. However, the corresponding bound of \cite[Theorem~3]{doan2020convergence} is incomplete, since the factor $\taumix^2$ is lost in the proof (see equations $(64-66)$). With this completion, the bound of \cite[Theorem~3]{doan2020convergence} yields a variance term of order $\mathcal{\tilde O} \left( \frac{\cmax^2 \taumix^{2}}{\mu^2\varepsilon} \right)$, which is suboptimal with respect to $\taumix$. Moreover, the corresponding analysis relies heavily on the assumption of a bounded domain. In \cite{even2023stochastic}, the author considers Markovian noise with a finite number of states and manages to obtain a rather interesting result of the form $O \left( \frac{L}{\mu} \log \frac{1}{\varepsilon} + \frac{L \tau \sigma^2_*}{\mu^3 \varepsilon}\right)$. Here the first term does not depend on $\tau$, and the second consists only $\sigma^*$ (stochasticity in $x^*$), but the price for this is an additional factor $L/\mu$ in the second term and more strict assumption that all realizations $F(\cdot, z)$ are smooth and strongly convex.
In the context of overparameterized learning, our results are almost consistent with the bound of \cite[Theorem~1]{vaswani2019fast} under i.i.d. sampling. The difference is that the term $\dmax^2$ in \Cref{as:bounded_markov_noise_UGE} can be more pessimistic than the expectation bound in \cite{vaswani2019fast}. 
\vspace{-3mm}
\subsection{Lower bounds}
\label{sec:lower_bounds_main}
\vspace{-2mm}
We start with a lower bound for the complexity of Markovian stochastic optimization under the assumptions \Cref{as:lipsh_grad} --\Cref{as:bounded_markov_noise_UGE}. Below we provide a result that highlights that the bound of \Cref{th:acc} is tight provided that $\dmax$ does not scale with the instance-dependent quantities, e.g., condition number $L/\mu$. 

\begin{theorem}
\label{th:strongly_convex_lower_bound}
There exists an instance of the optimization problem satisfying assumptions \Cref{as:lipsh_grad} --\Cref{as:bounded_markov_noise_UGE} with $\dmax = 1$ and arbitrary $\sigma \geq 0, L, \mu > 0, \taumix \in \nsets$, such that for any first-order gradient method it takes at least 
\[
\textstyle{
N = \Omega\left(\taumix \sqrt{\frac{L}{\mu}} \log\frac{1}{\varepsilon} + \frac{\taumix \sigma^2}{\mu^2 \varepsilon}\right)
}
\]
oracle calls in order to achieve $\E[\norm{x^N - x^{*}}^2] \leq \varepsilon$. 
\end{theorem}
The proof is provided in \Cref{sec:lower_bounds}. The idea of the constructed lower for deterministic part bound $\Omega\bigl(\taumix \sqrt{\frac{L}{\mu}} \log\frac{1}{\varepsilon}\bigr)$ goes back to \cite[Theorem~2.1.13]{nesterov2003introductory}. The stochastic part lower bound goes back to the classical statistical reasoning, and is well explained for i.i.d. noise in \cite[Chapter~4.1]{lan20}. Our adaptation for Markovian setting is based on Le Cam's theory, see \cite[Theorem~8]{aamari_levrard2019}, and also \cite{Yu1997}. For the case of Markov noise this lower bound is, to the best of our knowledge, original. The closest result to ours is the stochastic term lower bound in \cite[Proposition~1]{even2023stochastic}, but it is valid only for the vanilla stochastic gradient methods. Below we provide another lower bound showing that the dependence of the sample complexity \Cref{cor:sample_complexity_accelerated} on $\dmax$ is not an artefact of the proof. 
\begin{proposition}
\label{prop:regression_lower_bound}
There exists an instance of the optimization problem satisfying assumptions \Cref{as:lipsh_grad} --\Cref{as:bounded_markov_noise_UGE} with arbitrary $L, \mu > 0, \taumix \in \nsets$, $\dmax = \tfrac{L}{\mu}$, and $\cmax = 0$, such that for any first-order gradient method it takes at least 
\[
\textstyle{N = \Omega\left(\taumix \frac{L}{\mu} \log\frac{1}{\varepsilon}\right)}
\]
gradient calls in order to achieve $\E[\norm{x^N - x^{*}}^2] \leq \varepsilon$.
\end{proposition}
\vspace{-2mm}
This lower bound is adapted from the information-theoretic lower bound \cite{nagaraj2020least}. The detailed proof can be found in \Cref{sec:lower_bounds}. Recent studies \cite{kidambi2018insufficiency, nagaraj2020least, chen2020convergence} have revealed the impossibility of accelerating stochastic gradient descent (SGD) for online linear regression problems with specific noise structures. To address this issue, researchers have proposed various solutions, such as the MaSS algorithm \cite{liu2018accelerating} and the approach presented in \cite{JMLR:v18:16-595}. However, these methods rely heavily on the particular structure of the online regression setup. Another question that naturally arises is whether one can get rid of the dependence on $\taumix$ in the deterministic part of \eqref{eq:sample_complexity_accelerated} if $\dmax = 0$. The following counterexample shows that this is not the case in general.
\begin{proposition}
\label{prop:lower_bound_delta_0}
There exists an instance of the optimisation problem satisfying assumptions \Cref{as:lipsh_grad} --\Cref{as:bounded_markov_noise_UGE} with with arbitrary $L, \mu > 0, \taumix \in \nsets$, $\cmax = 1, \dmax = 0$, such that for any first-order gradient method it takes at least 
\[
\textstyle{
N = \Omega\left(\left(\taumix + \sqrt{\frac{L}{\mu}}\right) \log \frac{1}{\varepsilon} \right)
}
\]
oracle calls in order to achieve $\E[\norm{x^N - x^{*}}^2] \leq \varepsilon$.
\end{proposition}
The proof is provided in \Cref{sec:lower_bounds}. 
\vspace{-4mm}
\subsection{Non-convex problems}
\label{sec:non-cvx}
\vspace{-2mm}
Now we proceed with a randomized batch size version of the simple SGD algorithm. It is summarized in \Cref{alg:Rand-GD} and can be shown to achieve optimal rates of convergence for smooth non-convex problems. For the case of non-convex problems with Markov noise similar analysis appeared in \cite[Theorem~4]{dorfman2022adapting}. 

\vspace{-2mm}
\begin{algorithm}[h!]
\caption{\texttt{Randomized GD}}
\label{alg:Rand-GD}
\begin{algorithmic}[1]
\State {\bf Parameters:} stepsize $\gamma>0$, number of iterations $K$, bound on batchsize $B$, mixing time $\taumix$;
\State {\bf Initialization:} choose $x^0 \in \mathcal{X}$
\For{$k = 0, 1, 2, \dots, \nbiter-1$}
    \State Sample $J_k \sim \text{Geom}\left(\tfrac{1}{2}\right)$
    \State
    \text{\small{ 
    $
    g^{k} = g^{k}_0 +
    \begin{cases}
        2^{J_k} \left( g^{k}_{J_k}  - g^{k}_{J_k - 1} \right), & \text{ if } 2^{J_k} \leq \batchbound \\
        0, & \text{ otherwise}
    \end{cases}
    $ \quad with \quad
    $
    \textstyle{g^k_j = 2^{-j} B^{-1} \sum\nolimits_{i=1}^{2^j B} \nabla f(x^{k}, Z_{T^{k} + i})}
    $
    }}
    \label{line_gd_gk}
    \State $x^{k+1} = x^k - \gamma g^{k}$ \label{line_gd_3}
    \State $T^{k+1} = T^{k} +  2^{J_{k}} B$
\EndFor
\end{algorithmic}
\end{algorithm}
\vspace{-2mm}
By balancing the values of $B$ and $M$ with \Cref{lem:expect_bound_grad}, we establish the following result: 
\begin{theorem}
\label{th:non_convex_random_batch}
Assume \Cref{as:lipsh_grad},~\Cref{as:Markov_noise_UGE},~\Cref{as:bounded_markov_noise_UGE}. Let problem \eqref{eq:erm}
be solved by \Cref{alg:Rand-GD}. Let $f^*$ be a global (maybe not unique) minimum of $f$. Then for any $b \in \nsets$, and $\gamma$, $M$ satisfying
\begin{align*}
\gamma \lesssim (L [1 + \delta^2 \taumix b^{-1} + \dmax^2 \taumix^2 b^{-2}])^{-1},
\quad
M \simeq \max\{2;\sqrt{\gamma^{-1} L^{-1}}\}, \quad B = \lceil b \log_2 \batchbound \rceil,
\end{align*}
it holds that
\[
\textstyle{
\EE\left[ \frac{1}{\nbiter}\sum_{k=0}^{\nbiter-1}\| \nabla f(x^k)\|^2 \right] \lesssim \frac{f(x^0) - f^*}{\gamma \nbiter} + L \gamma  \cdot \left[ \sigma^2 \taumix b^{-1} +\cmax^2 \taumix^2 b^{-2} \right]}\,.
\]
\end{theorem}
The proof is provided in \Cref{sec:proof:th:non_convex_upper_bound}. The next corollary immediately follows from the theorem. 
\begin{corollary}
Under the conditions of \Cref{th:non_convex_random_batch}, if we choose $b = \taumix$ and $\gamma$ given by
\begin{align*}
\textstyle{
\gamma \simeq \min\left\{ \frac{1}{L (1 + \dmax^2 )};\, \sqrt{\frac{f(x^0) - f^*}{L N \cmax^2 }} \right\}}\,,
\end{align*}
then to achieve $\varepsilon$-solution (in terms of $\EE[\|\nabla f(x)\|^2] \lesssim \varepsilon^2$) we need 
\begin{align*}
\textstyle{
    \mathcal{\tilde O} \left( \taumix \cdot \left[\frac{(1 + \dmax^2) L (f(x^0) - f^*)}{\varepsilon^2}+ \frac{L(f(x^0) - f^*)\cmax^2}{ \varepsilon^4} \right] \right)
} \quad \text{oracle calls}.
\end{align*}    
\end{corollary}
\vspace{-3mm}
\textbf{Comparison.} The respective bound for the non-convex setting provided in \cite[Theorem~1]{doan2020convergence} yields the sample complexity of order $\mathcal{\tilde O} \left(\frac{ \taumix^2 L(f(x^0) - f(x^*))\cmax^2}{ \varepsilon^4}\right)$.
Also we can note the results of \cite[Theorem 2]{even2023stochastic} with the following estimate
$O\left( \frac{\tau (L(f(x^0) - f(x^*)) + \sigma^2)}{\varepsilon^2} + \frac{\tau (L(f(x^0) - f(x^*)) + \sigma^2) \sigma^2}{\varepsilon^4}\right)$.

To achieve linear convergence rates in the non-convex setting we can use the Polyak-Lojasiewicz (PL) condition \cite{POLYAK1963864}. The respective result is provided in \Cref{sec:polyak_loj}.
\vspace{-3mm}
\subsection{Variational inequalities}
\label{sec:vi}
\vspace{-2mm}
In this section, we are interested in the following problem:
\begin{equation}
\label{eq:VI}
\textstyle{
\text{Find } x^* \in \mathcal{X} \text{ such that } \langle F(x^*), x - x^* \rangle + r(x) - r(x^*) \geq 0 \text{ for all } x \in \mathcal{X}.
}
\end{equation}
Here $F: \R^d \to \R^d$ an operator, $ \mathcal{X}$ a convex set, and $r: \R^d \to \R$ is a regularization term (a suitable lower semicontinuous convex function) which is assumed to have a simple structure. As mentioned earlier, this problem is quite general and covers a wide range of possible problem formulations. For example, if the operator $F$ is the gradient of a convex function $f$, then the problem \eqref{eq:VI} is equivalent to the composite minimization problem \cite{beck2017first}, i.e., minimization of $f(x) + r(x)$.  In the meantime, \eqref{eq:VI} is also a reformulation of the min-max problem
\begin{equation}
\label{eq:minmax}
\textstyle{
\min\limits_{x_1 \in \mathcal{X}_1}\max\limits_{x_2 \in \mathcal{X}_2} r_1(x_1) + g(x_1, x_2) - r_2(x_2),
}
\end{equation}
with convex-concave continuously differentiable $g$, convex sets $\mathcal{X}_1$, $\mathcal{X}_2$ and convex functions $r_1$, $r_2$. Using the first-order optimality conditions, it is easy to verify that \eqref{eq:minmax} is equivalent to \eqref{eq:VI} with $x= (x_1^T, x_2^T)^T$, $F(x) = (\nabla_{x_1}f(x_1,x_2)^T, -\nabla_{x_2}f(x_1,x_2)^T)^T$, and $r(x) = r_1(x_1) + r_2(x_2)$.

\begin{assumption}
\label{as:lipshitz_op}
The operator $F$ is $L$-Lipschitz continuous on $\mathcal{X}$ with $L > 0$, i.e., the following inequality holds for all $x,y \in \mathcal{X}$:
\begin{equation*}
    \label{eq:lipshitz_op}
    \textstyle{
    \norm{F(x) - F(y) } \leq L \norm{x-y } , \qquad \forall x,y\in \mathcal{X}.
    }
\end{equation*}
\end{assumption}

\begin{assumption}
\label{as:strong_monotone_op}
The operator $F$ is $ \mu_F$-strongly monotone on $\mathcal{X}$, i.e., the following inequality holds for all $x,y \in \mathcal{X}$:
\begin{equation}
 \label{eq:strong_monotone_op}
 \textstyle{
 \langle F(x) - F(y), x - y \rangle \geq \mu_F \|x - y\|^2.
 }
\end{equation}
The function $r$ is $\mu_r$-strongly convex on $\mathcal{X}$, i.e. for all $x, y \in \mathcal{X}$ and any $r'(x) \in \partial r(x)$ we have
\begin{equation}
    \label{eq:strong_function_vi}
    \textstyle{
    r(y) \geq r(x) + \langle r'(x), y - x \rangle + (\mu_r/2) \|x - y \|^2.
    }
\end{equation}
\end{assumption}
\vspace{-2mm}
These two assumptions are more than standard for the study of variational inequalities and are found in all the papers from \Cref{tab:comparison1}. We consider two cases: strongly monotone/convex with $\mu_F + \mu_r > 0$ and monotone/convex with $\mu_F + \mu_r = 0$.
\begin{assumption}
\label{as:bounded_markov_noise_UGE_op}
For all $x \in \rset^{d}$ it holds that $\E_{\pi}[F(x, Z)] = F(x)$. Moreover, for all $z \in \Zset$ and $x \in \mathcal{X}$ it holds that
\begin{equation}
\label{eq:growth_condition_vi}
\textstyle{
\norm{F(x, z) - F(x)}^{2} \leq \cmax^{2} + \Dmax^{2} \| x - x^* \|^{2}\,,
}
\end{equation}
where $x^*$ is some point from the solution set.
\end{assumption}
\vspace{-2mm}
\Cref{as:bounded_markov_noise_UGE_op} is found in the literature on variational inequalities \cite{hsieh2020explore, doi:10.1137/15M1031953, gorbunov2022stochastic} and is considered to be analog to \Cref{as:bounded_markov_noise_UGE} on overparametrized learning.

Just as the Nesterov accelerated method is optimal for smooth convex minimization problems, the ExtraGradient method \cite{korpelevich1976extragradient, nemirovski2004prox, juditsky2011solving} is optimal for monotone variational inequalities. Therefore, we take it as a base. On the extrapolation step (\Cref{lin:eg_extr}) of \Cref{alg:EG}, we simply collect a batch of size $B$, but on the main step (\Cref{lin:eg_main}) we use the randomization as in \Cref{alg:AGD_ASGD}. 
The next theorem gives the convergence of our method.
\vspace{-1mm}
\begin{theorem} \label{th:vi_str_mon}
Assume \Cref{as:lipshitz_op},~\Cref{as:strong_monotone_op} with $\mu_F + \mu_r > 0$,~\Cref{as:Markov_noise_UGE},~\Cref{as:bounded_markov_noise_UGE_op}. Let problem \eqref{eq:VI}
be solved by \Cref{alg:EG}. Then for any $b \in \nsets$, and $\gamma$, $M$ satisfying
\begin{eqnarray*}
&\textstyle{\gamma \lesssim \min\left\{ (\mu_F + \mu_r)^{-1} ; L^{-1}; (\mu_F + \mu_r) (\Delta^2 \tau b^{-1} + \Delta^2 \tau^2 b^{-2})^{-1} ; \sqrt{\Delta^{-2} \tau^{-1} b} \right\}}\,,
\\
&
\textstyle{M \simeq \max\{2;\sqrt{\gamma^{-1} (\mu_F + \mu_r)^{-1}}\}, \quad B = \lceil b \log_2 \batchbound \rceil}\,,
\end{eqnarray*}
it holds that
\begin{equation*}
\textstyle{
    \EEE{\|x^{N} - x^* \|^2}
        \lesssim
        \exp\left(- \frac{N (\mu_F + \mu_r) \gamma}{2}\right) \| x^0 - x^*\|^2
        + \frac{\gamma}{\mu}(\sigma^2 \tau b^{-1} + \sigma^2 \tau^2 b^{-2})
}\,.
\end{equation*}
\end{theorem}
The proof is postponed to \Cref{sec:proof:th:strongly_mon_upper_bound}. One can get an estimate on oracle complexity.

\begin{algorithm}[h!]
   \caption{\texttt{Randomized ExtraGradient}}
   \label{alg:EG}
\begin{algorithmic}[1]
\State {\bf Parameters:} stepsize $\gamma>0$, number of iterations $\nbiter$
\State {\bf Initialization:} choose  $x^0 \in \mathcal{X}$
\For{$k = 0, 1, 2, \dots, \nbiter-1$}
    \State    $\textstyle{x^{k+1/2} = \text{prox}_{\gamma r} \bigl(x^k - \gamma B^{-1} \sum\nolimits_{i=1}^{B} F(x^{k}, Z_{T^{k} + i}) \bigr)}$ \label{lin:eg_extr}
    \State $\textstyle{T^{k+1/2} = T^{k} + B}$
    \State Sample $\textstyle{J_k \sim \text{Geom}\bigr(\tfrac{1}{2}\bigr)}$
    \State
    \text{\small{ 
    $
        g^{k} = g^{k}_0 +
        \begin{cases}
        \textstyle{2^{J_k} \left( g^{k}_{J_k}  - g^{k}_{J_k - 1} \right)}, & \text{ if } 2^{J_k} \leq \batchbound \\
        0, & \text{ otherwise}
        \end{cases}
    $ \quad with \quad
    $
    \textstyle{g^k_j = 2^{-j} B^{-1} \sum\nolimits_{i=1}^{2^j \cdot B} F(x^{k+1/2}, Z_{T^{k+1/2} + i})}
    $
    }}
    \State    $\textstyle{x^{k+1} = \text{prox}_{\gamma r} \bigl(x^k - \gamma g^{k}\bigr)}$ \label{lin:eg_main}
    \State    $\textstyle{T^{k+1} = T^{k+1/2} + 2^{J_{k}} B}$
\EndFor
\end{algorithmic}
\end{algorithm}
\vspace{-1mm}
\begin{corollary}
\label{cor:sample_complexity_vi_str_mon}
Under the conditions of \Cref{th:vi_str_mon}, if we choose $b = \taumix$ and $\gamma$ as follows
\begin{align*}
\textstyle{
        \gamma \simeq \min\left\{ \frac{1}{\mu_F + \mu_r}; \frac{1}{L};  \frac{\mu_F + \mu_r}{\Delta^2}; \frac{1}{\Delta};        
        \frac{1}{N(\mu_F + \mu_r)}\ln \left( \max\left\{ 2; \frac{\mu \nbiter \| x^0 - x^*\|^2 }{\cmax^2} \right\}\right) \right\}
}\,,
\end{align*}
then to achieve $\varepsilon$-solution (in terms of $\EE[\|x - x^*\|^2] \lesssim \varepsilon$) we need 
\begin{equation*}
\textstyle{
    \mathcal{\tilde O} \left( \taumix \cdot \left[\left(1 + \frac{L}{\mu_F + \mu_r} + \frac{\Delta}{\mu_F + \mu_r} + \frac{\Delta^2}{(\mu_F + \mu_r)^2}\right)\log \frac{1}{\varepsilon} + \frac{\cmax^2}{\mu^2\varepsilon} \right] \right)
} \quad \text{oracle calls}.
\end{equation*}
\end{corollary}
Note that one provide an (almost) matching lower complexity bounds for variational inequalities via lower bounds for saddle point problems, which are a special case of variational inequalities. The method for obtaining lower bounds for saddle point problems is reduced to obtaining estimates for the strongly convex minimization problem (see \cite{zhang2019lower, han2021lower}  for respective deterministic lower bounds), which we provide in \Cref{sec:lower_bounds_main}. Similarly, the question of constructing a lower bound which is tight w.r.t. $\Delta$ remains open.

For the monotone case, we use the {\em gap function} as a convergence criterion:
\begin{equation}
\label{eq:gap}
\textstyle{
    \text{Gap} (x) = \sup_{y \in \mathcal{X}} \left[ \langle F(y),  x - y  \rangle + r(x) - r(y)\right]\,.
}
\end{equation}
Such a criterion is standard and classical for monotone variational inequalities \cite{nemirovski2004prox, juditsky2011solving}. An important assumption for the gap function is the boundedness of the set $\mathcal{X}$.
\begin{assumption}
\label{as:boundedset}
The set $\mathcal{X}$ is bounded and has a diameter $D$, i.e., for all $x,y \in \mathcal{X}$:
$\| x - y\|^2 \leq D^2$.
\end{assumption}
\vspace{-2mm}
\Cref{as:boundedset} can be slightly relaxed. We need to use a simple trick from \cite{nesterov2007dual}. In particular, we need to consider $\mathcal{C}$ --  a compact subset of $\mathcal{X}$ and change $\mathcal{X}$ to $\mathcal{C}$ in \eqref{eq:gap}. But such a technique is rather technical and does not change the essence.
Finally, the following result holds.
\begin{theorem} \label{th:vi_mon}
Assume \Cref{as:lipshitz_op},~\Cref{as:strong_monotone_op} with $\mu_F + \mu_r = 0$,~\Cref{as:boundedset},~\Cref{as:Markov_noise_UGE},~\Cref{as:bounded_markov_noise_UGE_op}. Let problem \eqref{eq:VI}
be solved by \Cref{alg:EG}. Then for any $B \in \nsets$, and $\gamma$, $M$ satisfying $\gamma \lesssim L^{-1}\,, \,\, M = \sqrt{N}$, it holds that
\begin{equation*}
\textstyle{
    \EEE{\text{Gap}(\bar x^\nbiter)}
        \lesssim
        \frac{D^2}{\gamma N} + \gamma (\taumix B^{-1}\log_2 N + \taumix^2 B^{-2} )(\cmax^2 + \Dmax^2 D^2)
}~~ \text{where}~~ \bar x^\nbiter = \tfrac{1}{\nbiter} \sum_{k=0}^{\nbiter-1} x^{k+1/2}\,.
\end{equation*}
\end{theorem}
The proof is postponed to \Cref{sec:proof:th:mon_upper_bound}. The following corollary holds.
\begin{corollary}
\label{cor:sample_complexity_vi_mon}
Under the conditions of \Cref{th:vi_mon}, if we choose $B = \taumix$ and $\gamma$ as follows
\begin{align*}
\textstyle{\gamma \simeq \min\left\{ \frac{1}{L};  \sqrt{\frac{D^2}{(\sigma^2 + \Delta^2 D^2) N}}\right\}}\,,
\end{align*}
then to achieve $\varepsilon$-solution (in terms of $\EE[\text{Gap} (x)] \lesssim \varepsilon$) we need 
\begin{equation*}
\textstyle{
\mathcal{\tilde O} \left( \taumix \left[ \frac{LD^2}{\varepsilon} +  \frac{\sigma^2 D^2 + \Delta^2 D^4}{\varepsilon^2}\right] \right)
} \quad \text{oracle calls}.
\end{equation*}
\end{corollary}

\textbf{Comparison.} These results is the first for variational inequalities with Markovian stochasticity, either in the strongly monotone or monotone cases. The only close work is \cite{wang2022stability}. The authors work with convex-concave saddle point problems and provide the following estimate on the oracle complexity
$\mathcal{O}\left( \taumix^2 \cdot \frac{G^4}{\varepsilon^2} + \frac{D^2}{\varepsilon^2}\right)$ (with $G$ -- the uniform bound of the operator), which is worse than ours at least in terms of $\taumix$. Moreover, the authors consider the case of a finite Markov chain, which is a special case of our setup.
\vspace{-4mm}
\section{Conclusion}
\vspace{-3mm}
In this paper, we present a unified random batch size framework that achieves optimal finite-time performance for non-convex and strongly convex optimization problems with Markov noise, as well as for variational inequalities. Unlike existing methods, our framework relaxes the assumptions typically imposed on the domain and stochastic gradient oracle. We also provide a variety of lower bounds, which are to the best of our knowledge original in the Markov setting. 

\vspace{-4mm}
\section*{Acknowledgments}
\vspace{-3mm}
This research of A. Beznosikov has been supported by The Analytical Center for the Government of the Russian Federation (Agreement No. 70-2021-00143 dd. 01.11.2021, IGK 000000D730321P5Q0002). E. Moulines received support from the grant ANR-19-CHIA-002 SCAI and parts of his work has been done under the auspices of Lagrange Center for maths and computing.

\bibliographystyle{plain}
\bibliography{refs}
\clearpage
\newpage
\appendix
\section{Notations and definitions.} Let $(\Zset,\metricz)$ be a complete separable metric space endowed with its Borel $\sigma$-field $\Zsigma$. Let $(\Zset^{\nset},\Zsigma^{\otimes \nset})$ be the corresponding canonical process. Consider the Markov kernel $\MKQ$ defined on $\Zset \times \Zsigma$, and denote by $\PP_{\xi}$ and $\E_{\xi}$ the corresponding probability distribution and the expected value with initial distribution $\xi$. Without loss of generality, we assume that $(Z_k)_{k \in \nset}$ is the corresponding canonical process. By construction, for any $A \in \Zsigma$, it holds that $\PP_{\xi}(Z_k \in A | Z_{k-1}) = \MKQ(Z_{k-1},A)$, $\PP_{\xi}$-a.s. If $\xi = \delta_{z}$, $z \in \Zset$, we write $\PP_{z}$ and $\E_{z}$ instead of $\PP_{\delta_{z}}$ and $\E_{\delta_{z}}$, respectively. We denote $\mathcal{F}_{k} = \sigma(x^{j}, j \leq k)$ and write $\E_{k}$ as an alias for $\E[\cdot | \mathcal{F}_{k}]$. For each function $f: \Zset \mapsto \rset$ with $\pi(f) <\infty$ we write $\barf(z) = f(z) - \pi(f)$.

\section{Proofs of \Cref{sec:acc-method}, \Cref{sec:non-cvx}}
\subsection{Proof of \Cref{lem:tech_markov}}
\label{sec:proof:lem:tech_markov}

\begin{lemma}[\Cref{lem:tech_markov}]\label{lem:tech_markov_app}
Assume \Cref{as:Markov_noise_UGE} and \Cref{as:bounded_markov_noise_UGE}. Then, for any $n \geq 1$ and $x \in \rset^{d}$, it holds that
\begin{equation}
\label{eq:var_bound_stationary_app}
\textstyle{
\E_{\pi}[\norm{n^{-1}\sum\nolimits_{i=1}^{n}\nabla F(x, Z_{i}) - \nabla f(x)}^2]
    \leq
    \frac{8 \taumix}{n} \left( \cmax^2 + \dmax^2 \| \nabla f(x) \|^2 \right)
}\,.
\end{equation}
Moreover, for any initial distribution $\xi$ on $(\Zset,\Zsigma)$, that
\begin{equation}
\label{eq:var_bound_any_app}
\textstyle{
\E_{\xi}[\norm{n^{-1}\sum\nolimits_{i=1}^{n}\nabla F(x, Z_i) - \nabla f(x)}^2] \leq \frac{C_{1} \taumix}{n} \left( \cmax^2 + \dmax^2 \| \nabla f(x) \|^2 \right)
}\,,
\end{equation}
where $C_{1} = 16(1 + \frac{1}{\ln^{2}{4}})$.
\end{lemma}
By \cite[Lemma~19.3.6 and Theorem~19.3.9 ]{douc:moulines:priouret:soulier:2018}, for any two probabilities $\xi,\xi'$ on $(\Zset,\Zsigma)$ there is a \emph{maximal exact coupling} $(\Omega,\mathcal{F},\PPcoupling{\xi}{\xi'},Z,Z',T)$ of $\PP ^{\MKQ}_{\xi}$ and $\PP ^{\MKQ}_{\xi'}$, that is,
\begin{equation}
\label{eq:coupling_time_def_markov}
\tvnorm{\xi \MKQ^n- \xi'\MKQ^n} = 2 \PPcoupling{\xi}{\xi'}(T > n)\,.
\end{equation}
We write $\PEcoupling{\xi}{\xi'}$ for the expectation with respect to $\PPcoupling{\xi}{\xi'}$. Using the coupling construction \eqref{eq:coupling_time_def_markov},
\begin{multline*}
\textstyle{\E^{1/2}_{\xi} [\norm{\sum_{i=1}^{n}\{\nabla f(x, Z_i) - \nabla f(x)\}}^2]
\leq \E^{1/2}_{\pi}[\norm{\sum_{i=0}^{n-1}\nabla f(x, Z_i) - \nabla f(x)}^{2}]} + \\
\textstyle{\PEcoupling{\xi}{\pi}^{1/2}[\norm{\sum_{i=0}^{n-1}\{\nabla f(x, Z_i) - \nabla f(x, Z^{\prime}_i)\}}^{2}]}\,.
\end{multline*}
The first term is bounded with \eqref{eq:var_bound_stationary_app}. Moreover, with \eqref{eq:coupling_time_def_markov} and \Cref{as:bounded_markov_noise_UGE}, we get
\begin{align*}
    \norm{\sum_{i=0}^{n-1}\{\nabla f(x, Z_i) - \nabla f(x, Z^{\prime}_i)\}}^{2} 
    &\leq
    8 \left( \cmax^2 + \dmax^2 \| \nabla f(x) \|^2 \right) \bigl(\sum_{i=0}^{n-1}\indiacc{Z_i \neq Z^{'}_{i}}\bigr)^{2}
    \\
    &= 8 \left( \cmax^2 + \dmax^2 \| \nabla f(x) \|^2 \right) \bigl(\sum_{i=0}^{n-1}\indiacc{T > i}\bigr)^{2} 
    \\
    &\leq
    16 \left( \cmax^2 + \dmax^2 \| \nabla f(x) \|^2 \right) \sum_{i=1}^{\infty} i\, \indiacc{T > i}\,.
\end{align*}
Thus, using the assumption \Cref{as:Markov_noise_UGE}, we bound
\[
\PEcoupling{\xi}{\pi}[\sum_{i=1}^{\infty} i\, \indiacc{T > i}] = \sum_{i=1}^{\infty} i \PPcoupling{\xi}{\xi'}(T > i) = \sum_{i=1}^{\infty} i (1/4)^{\lfloor i / \taumix \rfloor} \leq 4 \sum_{i=1}^{\infty} i (1/4)^{ i / \taumix }\,.
\]
Now we set $\rho = (1/4)^{1/ \taumix}$ and use an upper bound
\[
\sum_{k=1}^{\infty}k \rho^{k}
\leq \rho^{-1}\int_{0}^{+\infty}x^{p}\rho^{x}\,d x \leq \rho^{-1}\left(\ln{\rho^{-1}}\right)^{-2} \Gamma(2) = \rho^{-1}\left(\ln{\rho^{-1}}\right)^{-2} = \frac{\taumix^{2}}{(1/4)^{1/\taumix} \ln^{2}{4}}\,.
\]
Combining the bounds above yields
\[
\E_{\xi}[\norm{n^{-1}\sum\nolimits_{i=1}^{n}\nabla f(x, Z_i) - \nabla f(x)}^2] \leq \frac{c_{1} \taumix}{n} \left( \cmax^2 + \dmax^2 \| \nabla f(x) \|^2 \right) + \frac{c_{2} \taumix^{2}}{n^2} \left( \cmax^2 + \dmax^2 \| \nabla f(x) \|^2 \right)\,,
\]
where $c_{1} = 16$, $c_{2} = \frac{128 (1/4)^{-1/\taumix}}{\ln^{2}{4}}$. Now we consider the two cases. If $n < c_{1}\taumix$, we get from Minkowski's inequality that
\[
\E_{\xi}[\norm{n^{-1}\sum\nolimits_{i=1}^{n}\nabla f(x, Z_i) - \nabla f(x)}^2] \leq 2\cmax^2 + 2\dmax^2 \| \nabla f(x) \|^2\,,
\]
and \eqref{eq:var_bound_any_app} holds. If $n > c_{1}\taumix$, it holds that
\[
\frac{c_{2} \taumix^{2}}{n^2} \left( \cmax^2 + \dmax^2 \| \nabla f(x) \|^2 \right) \leq \frac{c_{2} \taumix^{2}}{n c_{1} \taumix}\left( \cmax^2 + \dmax^2 \| \nabla f(x) \|^2 \right)\,,
\]
and we also get \eqref{eq:var_bound_any_app}.

\subsection{Proof of \Cref{lem:expect_bound_grad}} 
\label{sec:proof:lem:expect_bound_grad}
Before we proceed to the proof, we give a statement of \Cref{lem:expect_bound_grad} with exact constants.
\begin{lemma}[\Cref{lem:expect_bound_grad}]
\label{lem:expect_bound_grad_appendix}
Assume \Cref{as:Markov_noise_UGE} and \Cref{as:bounded_markov_noise_UGE}. Then for the gradient estimates $g^k$ from \Cref{alg:AGD_ASGD} it holds that $\E_k[g^k] = \E_k[g^{k}_{\lfloor \log_2 \batchbound \rfloor}]$. Moreover, 
\begin{align}
\label{eq:var_bounds_random_grad}
&\E_k[\| \nabla f(x^k_g) - g^k\|^2] \leq \left(4 C_{1} \taumix B^{-1}\log_2 \batchbound + (4C_{1} + 2) \taumix^2 B^{-2}\right) (\cmax^2 + \dmax^2 \| \nabla f(x^k_g)\|^2)\,, \\
&\| \nabla f(x^k_g) - \E_{k}[g^k]\|^2 \leq C_{2}\taumix^2 \batchbound^{-2}B^{-2} (\cmax^2 + \dmax^2 \| \nabla f(x^k_g)\|^2)\,, \nonumber 
\end{align}
where $C_1$ is defined in \eqref{eq:var_bound_any_app} and $C_2 = 256/3$.
\end{lemma}
\begin{proof}
To show that $\E_k[g^k] = \E_k[g^{k}_{\lfloor \log_2 \batchbound \rfloor}]$ we simply compute conditional expectation w.r.t. $J_k$:
\begin{align*}
\E_k[g^k] &= \E_k\left[\E_{J_k}[g^k]\right] =
\E_k[g^k_0] + \sum\limits_{i=1}^{\lfloor \log_2 \batchbound \rfloor} \Prob\{J_k = i\} \cdot 2^i \E_k[g^{k}_{i}  - g^{k}_{i-1}] \\
&= \E_k[g^k_0] + \sum\limits_{i=1}^{\lfloor \log_2 \batchbound \rfloor} \E_k[g^{k}_{i}  - g^{k}_{i-1}] = \E_k[g^{k}_{\lfloor \log_2 \batchbound \rfloor}]\,.
\end{align*}
We start with the proof of the first statement of \eqref{eq:var_bounds_random_grad} by taking the conditional expectation for $J_k$:
\begin{align*}
    &\E_{k}[\| \nabla f(x^k_g) - g^k\|^2]
    \leq
    2\E_{k}[\| \nabla f(x^k_g) - g^k_0\|^2] + 2\E_{k}[\| g^k - g^k_0\|^2]
    \\
    &\quad=
    2\E_{k}[\| \nabla f(x^k_g) - g^k_0\|^2] + 2 \sum\nolimits_{i=1}^{\lfloor \log_2 \batchbound \rfloor} \Prob\{J_k = i\} \cdot 4^i \E_k[\|g^{k}_{i}  - g^{k}_{i-1}\|^2] \\
    &\quad=
    2\E_{k}[\| \nabla f(x^k_g) - g^k_0\|^2] + 2\sum\nolimits_{i=1}^{\lfloor \log_2 \batchbound \rfloor} 2^i \E_k[\|g^{k}_{i}  - g^{k}_{i-1}\|^2] \\
    &\quad\leq
    2\E_{k}[\| \nabla f(x^k_g) - g^k_0\|^2] + 4\sum\nolimits_{i=1}^{\lfloor \log_2 \batchbound \rfloor} 2^i \left(\E_k[\|\nabla f(x^k_g)  - g^{k}_{i-1}\|^2 + \E_k[\|g^{k}_{i}  - \nabla f(x^k_g)\|^2] \right)\,.
\end{align*}
To bound $\E_{k}[\| \nabla f(x^k_g) - g^k_0\|^2]$, $\E_{k}[\|\nabla f(x^k_g)  - g^{k}_{i-1}\|^2$, $\E_{k}[\|g^{k}_{i}  - \nabla f(x^k_g)\|^2]$, we apply \Cref{lem:tech_markov} and get
\begin{align*}
    \E_{k}[\| \nabla f(x^k_g) - g^k\|^2]
    &\leq 2 \cmax^{2} + 4\sum\nolimits_{i=1}^{\lfloor \log_2 \batchbound \rfloor} 2^i \left(\frac{C_{1} \taumix}{2^{i} B} (\cmax^2 + \dmax^2 \| \nabla f(x^k_g)\|^2) + \frac{C_{1} \taumix^{2}}{2^{2i} B^{2}}(\cmax^2 + \dmax^2 \| \nabla f(x^k_g)\|^2) \right)\\
    &\leq
    \frac{4 C_{1} (\cmax^2 + \dmax^2 \| \nabla f(x^k_g)\|^2) \taumix \log_2 \batchbound}{B} + \frac{(4C_{1} + 2) (\cmax^2 + \dmax^2 \| \nabla f(x^k_g)\|^2) \taumix^2}{B^2}\,.
\end{align*}
To show the second part of the statement, we use \Cref{lem:expect_bound_grad} and get
\[
\| \nabla f(x^k_g) - \E_{k}[g^k]\|^2 = \| \nabla f(x^k_g) - \E_k[g^{k}_{\lfloor \log_2 \batchbound \rfloor}]\|^2\,.
\]
The remaining proof once again uses \Cref{lem:tech_markov} and is omitted. To conclude we use that $2^{\lfloor \log_{2}\batchbound \rfloor} \geq \batchbound/2$.
\end{proof}

\subsection{Proof of \Cref{th:acc}.} 
\label{sec:proof:th:strongly_convex_upper_bound}
We preface the proof by two technical Lemmas.
\begin{lemma}
\label{lem:tech_lemma_acc_gd}
Assume \Cref{as:lipsh_grad} and \Cref{as:strong_conv}. Then for the iterates of \Cref{alg:AGD_ASGD} with $\theta = (p \eta^{-1} - 1) / (\beta p \eta^{-1} - 1)$, $\theta > 0$, $\eta \geq 1$, $p > 0$, it holds that
\begin{align}
\label{eq:acc_temp4}
    \E_k[\|x^{k+1} - x^*\|^2]
     \leq&
    (1 + \alpha p \gamma \eta)( 1 - \beta) \| x^k - x^*\|^2 + (1 + \alpha p \gamma \eta) \beta\|x^k_g - x^*\|^2 
    \notag\\
    &
    + (1 + \alpha p \gamma \eta) (\beta^2 - \beta )\|x^k - x^k_g\|^2
    + p^2 \eta^2 \gamma^2 \EE_{k}[\| g^k \|^2] 
    \notag\\
    &
    - 2 \eta^2 \gamma \langle \nabla f(x^k_g), x^k_g + (p \eta^{-1} - 1)  x^k_f - \eta^{-1} p x^*\rangle 
    \notag\\
    &
    + \frac{p \eta \gamma}{\alpha} \|\EE_k[g^k] - \nabla f(x^k_g)\|^2\,,
\end{align}
where $\alpha > 0$ is any positive constant.
\end{lemma}
\begin{proof}
We start with lines \ref{line_acc_3} and \ref{line_acc_2} of \Cref{alg:AGD_ASGD}:
\begin{align*}
    \|x^{k+1} - x^*\|^2
    =&
    \| \eta x^{k+1}_f + (p - \eta)x^k_f + (1- p)(1 - \beta) x^k +(1 - p) \beta x^k_g - x^*\|^2
    \notag\\
    =&
    \| \eta x^k_g - p\eta\gamma g^k + (p - \eta)x^k_f + (1- p)(1 - \beta) x^k +(1 - p) \beta x^k_g - x^*\|^2
    \notag\\
    =&
    \| \eta x^k_g + (p - \eta)x^k_f + (1- p)(1 - \beta) x^k +(1 - p) \beta x^k_g - x^*\|^2 + p^2 \gamma^2 \eta^2 \| g^k \|^2 
    \notag\\
    &
    - 2 p\gamma \eta \langle  g^k, \eta x^k_g + (p - \eta)x^k_f + (1- p)(1 - \beta) x^k +(1 - p) \beta x^k_g - x^*\rangle.
\end{align*}
Using straightforward algebra, we get
\begin{align*}
    \|x^{k+1} - x^*\|^2
    =&
    \| \eta x^k_g + (p - \eta)x^k_f + (1- p)(1 - \beta) x^k +(1 - p) \beta x^k_g - x^*\|^2 + p^2 \gamma^2 \eta^2 \| g^k \|^2 
    \notag\\
    &
    - 2 p\gamma \eta \langle \nabla f(x^k_g), \eta x^k_g + (p - \eta)x^k_f + (1- p)(1 - \beta) x^k +(1 - p) \beta x^k_g - x^* \rangle
   \notag\\
    &
    - 2 p\gamma \eta \langle \EE_k[g^k] - \nabla f(x^k_g), \eta x^k_g + (p - \eta)x^k_f + (1- p)(1 - \beta) x^k +(1 - p) \beta x^k_g - x^* \rangle
    \notag\\
    &
    - 2 p\gamma \eta \langle g^k - \EE_k[g^k], \eta x^k_g + (p - \eta)x^k_f + (1- p)(1 - \beta) x^k +(1 - p) \beta x^k_g - x^* \rangle
    \notag\\
    \leq&
    (1 + \alpha p \eta \gamma)\| \eta x^k_g + (p - \eta)x^k_f + (1- p)(1 - \beta) x^k +(1 - p) \beta x^k_g - x^* \|^2
    \notag\\
    &
    - 2 p\gamma \eta \langle \nabla f(x^k_g), \eta x^k_g + (p - \eta)x^k_f + (1- p)(1 - \beta) x^k +(1 - p) \beta x^k_g - x^* \rangle
    \notag\\
    &
    - 2 p\gamma \eta \langle g^k - \EE_k[g^k], \eta x^k_g + (p - \eta)x^k_f + (1- p)(1 - \beta) x^k +(1 - p) \beta x^k_g - x^* \rangle
    \notag\\
    &
    + p^2\gamma^2 \eta^2 \| g^k \|^2 + \frac{p\gamma \eta}{\alpha} \|\EE_k[g^k] - \nabla f(x^k_g)\|^2.
\end{align*}
In the last step we also applied Cauchy-Schwartz inequality in the form \eqref{eq:inner_prod} with $\alpha > 0$. Taking the conditional expectation, we get
\begin{align}
\label{eq:acc_temp1}
\E_{k}[\|x^{k+1} - x^*\|^2] \leq& (1 + \alpha p \eta \gamma) \|\eta x^k_g + (p - \eta)x^k_f + (1- p)(1 - \beta) x^k +(1 - p) \beta x^k_g - x^*\|^2 
\notag\\
    &
    - 2 p\gamma \eta \langle \nabla f(x^k_g), \eta x^k_g + (p - \eta)x^k_f + (1- p)(1 - \beta) x^k +(1 - p) \beta x^k_g - x^* \rangle
    \notag\\
    &
    + p^2 \gamma^2 \eta^2 \E_{k}[\| g^k \|^2]  + \frac{p\gamma \eta}{\alpha} \|\EE_k[g^k] - \nabla f(x^k_g)\|^2\,.
\end{align}
Now let us handle expression $\| \eta x^k_g + (p - \eta)x^k_f + (1- p)(1 - \beta) x^k +(1 - p) \beta x^k_g  - x^*\|^2$ for a while. Taking into account line \ref{line_acc_1} and the choice of $\theta$ such that $\theta = (p \eta^{-1} - 1) / (\beta p \eta^{-1} - 1)$ (in particular, $(p \eta^{-1} - 1) = (\beta p \eta^{-1} - 1)\theta$ and $ \eta (1 - \beta p \eta^{-1}) (1 - \theta) = p (1 - \beta)$), we get
\begin{align*}
\eta x^k_g + (p - \eta) & x^k_f + (1- p)(1 - \beta) x^k +(1 - p) \beta x^k_g
\notag\\
&=
(\eta + (1 - p) \beta) x^k_g + (p - \eta)x^k_f + (1- p)(1 - \beta) x^k
\notag\\
&=
(\eta + (1 - p) \beta) x^k_g + \eta (p \eta^{-1} - 1)x^k_f + (1- p)(1 - \beta) x^k
\notag\\
&=
(\eta + (1 - p) \beta) x^k_g + \eta (\beta p \eta^{-1} - 1)\theta x^k_f + (1- p)(1 - \beta) x^k
\notag\\
&=
(\eta + (1 - p) \beta) x^k_g + \eta (\beta p \eta^{-1} - 1)(x^k_g - (1 - \theta) x^k) + (1- p)(1 - \beta) x^k
\notag\\
&=
\beta x^k_g - \eta (\beta p \eta^{-1} - 1)(1 - \theta) x^k + (1- p)(1 - \beta) x^k
\notag\\
&=
\beta x^k_g + p(1 - \beta) x^k + (1- p)(1 - \beta) x^k
\notag\\
&=
\beta x^k_g + (1 - \beta) x^k
\,.
\end{align*}
Substituting into $\|\eta x^k_g + (p - \eta)x^k_f + (1- p)(1 - \beta) x^k +(1 - p) \beta x^k_g - x^*\|^2$, we get
\begin{align}
\label{eq:acc_temp3}
\| \eta x^k_g + (p - \eta)& x^k_f + (1- p)(1 - \beta) x^k +(1 - p) \beta x^k_g - x^*\|^2
\notag\\
&=
\| \beta x^k_g + (1 - \beta) x^k - x^*\|^2
\notag\\
&=
\| x^k - x^* + \beta (x^k_g - x^k)\|^2
\notag\\
&=
\| x^k - x^*\|^2 + 2 \beta \langle x^k - x^*, x^k_g - x^k \rangle  + \beta^2\|x^k - x^k_g\|^2
\notag\\
&=
\| x^k - x^*\|^2 + \beta \left( \|x^k_g - x^*\|^2 - \| x^k - x^*\|^2 - \| x^k_g - x^k\|^2\right)  + \beta^2\|x^k - x^k_g\|^2
\notag\\
&=
( 1 - \beta) \| x^k - x^*\|^2 + \beta\|x^k_g - x^*\|^2  + (\beta^2 - \beta )\|x^k - x^k_g\|^2.
\end{align}
Again with line \ref{line_acc_1} and the choice of $\theta$ such that $\theta = (p \eta^{-1} - 1) / (\beta p \eta^{-1} - 1)$ (in particular, $\eta^{-1} p (1 - \beta) = (1 - \beta p \eta^{-1}) (1 - \theta)$ and $(\beta p \eta^{-1} - 1)\theta = (p \eta^{-1} - 1)$), one can also note
\begin{align}
\label{eq:acc_temp8}
\eta x^k_g &+ (p - \eta)x^k_f + (1- p)(1 - \beta) x^k +(1 - p) \beta x^k_g - x^*
\notag\\
&= (\eta + (1 - p) \beta)x^k_g + (p - \eta)x^k_f + (1- p)(1 - \beta) x^k - x^*
\notag\\
&= \eta p^{-1}\left( (p + (1 - p) \eta^{-1}p \beta)x^k_g + (p \eta^{-1} - 1) p x^k_f + (1- p)(1 - \beta) p \eta^{-1} x^k - \eta^{-1} p x^* \right)
\notag\\
&= \eta p^{-1}\left( (p + (1 - p) \eta^{-1} p  \beta)x^k_g + (p \eta^{-1} - 1) p x^k_f + (1- p)(1 - \beta p \eta^{-1}) (1 - \theta) x^k - \eta^{-1} p x^* \right)
\notag\\
&= \eta p^{-1}\left( (p + (1 - p) \eta^{-1}p \beta)x^k_g + (p \eta^{-1} - 1) p x^k_f + (1- p)(1 - \beta p \eta^{-1}) (x^k_g - \theta x^k_f) - \eta^{-1} p x^* \right)
\notag\\
&= \eta p^{-1}\left( x^k_g + (p \eta^{-1} - 1) p x^k_f - (1- p)(1 - \beta p \eta^{-1}) \theta x^k_f - \eta^{-1} p x^* \right)
\notag\\
&= \eta p^{-1}\left( x^k_g + (p \eta^{-1} - 1) p x^k_f + (1- p)(p \eta^{-1} - 1) x^k_f - \eta^{-1} p x^* \right)
\notag\\
&= \eta p^{-1}\left( x^k_g + (p \eta^{-1} - 1)  x^k_f - \eta^{-1} p x^* \right)
.
\end{align}
Combining  \eqref{eq:acc_temp3} and \eqref{eq:acc_temp8} with \eqref{eq:acc_temp1}, we finish the proof.
\end{proof}

\begin{lemma}
\label{lem:tech_lemma_acc_gd_2}
Assume \Cref{as:lipsh_grad}-\Cref{as:strong_conv}. Let problem \eqref{eq:erm}  be solved by \Cref{alg:AGD_ASGD}. Then for any $u \in \R^d$, we get
\begin{align*}
        \EE_k[f(x^{k+1}_f)]
        \leq&
        f(u) - \langle \nabla f(x^k_g), u - x^k_g \rangle - \frac{\mu}{2} \| u - x^k_g\|^2  - \frac{\gamma}{2} \|\nabla f(x^k_g)\|^2 
        \\
        &
        + \frac{\gamma}{2} \|\EE_k[g^k] - \nabla f(x^k_g) \|^2 + \frac{L \gamma^2 }{2}\EE_k[\| g^k\|^2].
\end{align*}
\end{lemma}
\begin{proof}
Using  \Cref{as:lipsh_grad} in the form \eqref{eq:smothness} with $x= x^{k+1}_f$, $y = x^k_g$ and line \ref{line_acc_2} of \Cref{alg:AGD_ASGD}, we get
\begin{align*}
f(x^{k+1}_f) \leq& f(x^k_g) + \langle \nabla f(x^k_g), x^{k+1}_f - x^k_g \rangle + \frac{L}{2}\| x^{k+1}_f - x^k_g\|^2 \\
=& f(x^k_g) - p \gamma \langle \nabla f(x^k_g), g^k \rangle + \frac{L p^2 \gamma^2 }{2}\| g^k\|^2 \\
        =&
        f(x^k_g) - p \gamma  \langle \nabla f(x^k_g), \nabla f(x^k_g) \rangle - p\gamma \langle \nabla f(x^k_g), \EE_k[g^k] - \nabla f(x^k_g) \rangle 
        \\
        &
        - p\gamma \langle \nabla f(x^k_g), g^k - \EE_k[g^k] \rangle + \frac{L p^2\gamma^2 }{2}\| g^k\|^2
        \\
        \leq&
        f(x^k_g) - p \gamma \|\nabla f(x^k_g)\|^2 + \frac{p\gamma}{2} \| \nabla f(x^k_g) \|^2 + \frac{p\gamma}{2} \|\EE_k[g^k] - \nabla f(x^k_g) \|^2
        \\
        &
        - p \gamma \langle \nabla f(x^k_g), g^k - \EE_k[g^k] \rangle + \frac{L p^2\gamma^2 }{2}\| g^k\|^2.
\end{align*}
Here we also used Cauchy Schwartz inequality \eqref{eq:inner_prod} with $a =  \nabla f(x^k_g)$, $b =  \nabla f(x^k_g) - \EE_k[g^k]$ and $c = 1$. Taking the conditional expectation, we get
\begin{align*}
\EE_k[f(x^{k+1}_f)] \leq& f(x^k_g) - \frac{p \gamma}{2} \|\nabla f(x^k_g)\|^2 + \frac{p \gamma}{2} \|\EE_k[g^k] - \nabla f(x^k_g) \|^2 + \frac{L p^2 \gamma^2 }{2}\EE_k[\| g^k\|^2].
\end{align*}
Using \Cref{as:strong_conv}  with $x= u$ and $y= x^k_g$, one can conclude that for any $u \in \R^d$ it holds
    \begin{align*}
        \EE_k[f(x^{k+1}_f)]
        \leq&
        f(u) - \langle \nabla f(x^k_g), u - x^k_g \rangle - \frac{\mu}{2} \| u - x^k_g\|^2  - \frac{p \gamma}{2} \|\nabla f(x^k_g)\|^2 
        \\
        &+ \frac{p \gamma}{2} \|\EE_k[g^k] - \nabla f(x^k_g) \|^2 + \frac{L p^2 \gamma^2 }{2}\EE_k[\| g^k\|^2].
    \end{align*}
\end{proof}

\begin{theorem}[\Cref{th:acc}]
Assume \Cref{as:lipsh_grad} -- \Cref{as:bounded_markov_noise_UGE}. Let problem \eqref{eq:erm}
be solved by \Cref{alg:AGD_ASGD}. Then for any $b \in \nsets$,  $\gamma \in (0; \tfrac{3}{4L}]$, and $\beta, \theta, \eta, p, \batchbound, B$ satisfying 
\begin{eqnarray*}
&p = \left[1 + 2\left(  1 +  \gamma L \right) \left(1 + 4 \left[C_{1} \taumix b^{-1} + (C_{1} + 1) \taumix^2 b^{-2}\right] \dmax^2 \right)\right]^{-1},
\\
&\beta = \sqrt{\frac{4 p^2 \mu \gamma}{3}}, \quad\eta = \frac{3\beta}{2 p \mu \gamma} = \sqrt{\frac{3}{\mu \gamma}}, \quad  \theta = \frac{p \eta^{-1} - 1}{\beta p \eta^{-1} - 1},
\\
&M = \max\{2; \sqrt{C_2 p^{-1} (1 + 2p/\beta)}\}, \quad B = \lceil b \log_2 \batchbound \rceil .
\,
\end{eqnarray*}
it holds that
\begin{align*}
    \EE\Bigg[&\|x^{N} - x^*\|^2 + \frac{6}{\mu} (f(x^{N}_f) - f(x^*)) \Bigg]
    \\
    &\lesssim 
    \exp\left( - N\sqrt{\frac{p^2 \mu \gamma}{3}}\right) \left[\| x^0 - x^*\|^2 + \frac{6}{\mu} (f(x^0) - f(x^*)) \right]
    + \frac{p \sqrt{\gamma}}{\mu^{3/2}} \left( \sigma^2 \taumix b^{-1} + \cmax^2  \taumix^2 b^{-2} \right)\,.
\end{align*}
\end{theorem}
\begin{proof}
Using \Cref{lem:tech_lemma_acc_gd_2} with $u = x^*$ and $u = x^k_f$, we get
\begin{align*}
\EE_k[f(x^{k+1}_f)]
\leq&
f(x^*) - \langle \nabla f(x^k_g), x^* - x^k_g \rangle - \frac{\mu}{2} \|x^* - x^k_g\|^2  - \frac{p \gamma}{2} \|\nabla f(x^k_g)\|^2 \\
&+ \frac{p \gamma}{2} \|\EE_k[g^k] - \nabla f(x^k_g) \|^2 + \frac{L p^2 \gamma^2 }{2}\EE_k[\| g^k\|^2],
\end{align*}
\begin{align*}
        \EE_k[f(x^{k+1}_f)]
        \leq&
        f(x^k_f) - \langle \nabla f(x^k_g), x^k_f - x^k_g \rangle - \frac{\mu}{2} \| x^k_f - x^k_g\|^2  - \frac{p \gamma}{2} \|\nabla f(x^k_g)\|^2 
        \\
        &+ \frac{p \gamma}{2} \|\EE_k[g^k] - \nabla f(x^k_g) \|^2 + \frac{L p^2 \gamma^2 }{2}\EE_k[\| g^k\|^2].
\end{align*}
Summing the first inequality with coefficient $2 p\gamma \eta  $, the second with coefficient $2 \gamma \eta (\eta - p) $ and \eqref{eq:acc_temp4}, we obtain
\begin{align*}
    \E_k[\|x^{k+1} - x^*\|^2 &+ 2 \gamma \eta^2 f(x^{k+1}_f)]
    \\
    \notag \leq&
    (1 + \alpha p \gamma \eta)( 1 - \beta) \| x^k - x^*\|^2 + (1 + \alpha p \gamma \eta) \beta\|x^k_g - x^*\|^2 
    \\
    &+ (1 + \alpha p \gamma \eta) (\beta^2 - \beta )\|x^k - x^k_g\|^2 - 2 \eta^2 \gamma \langle \nabla f(x^k_g), x^k_g + (p \eta^{-1} - 1)  x^k_f - \eta^{-1} p x^*\rangle 
    \\
    &+ p^2 \eta^2 \gamma^2 \EE_{k}[\| g^k \|^2] + \frac{p \eta \gamma}{\alpha} \|\EE_k[g^k] - \nabla f(x^k_g)\|^2
    \\
    & + 2 p \gamma \eta \Big (f(x^*) - \langle \nabla f(x^k_g), x^* - x^k_g \rangle - \frac{\mu}{2} \|x^* - x^k_g\|^2  - \frac{p\gamma}{2} \|\nabla f(x^k_g)\|^2 
    \\
        &+ \frac{p\gamma}{2} \|\EE_k[g^k] - \nabla f(x^k_g) \|^2 + \frac{L p^2 \gamma^2 }{2}\EE_k[\| g^k\|^2]\Big)
    \\
    & + 2 \gamma \eta (\eta - p) \Big ( f(x^k_f) - \langle \nabla f(x^k_g), x^k_f - x^k_g \rangle - \frac{\mu}{2} \| x^k_f - x^k_g\|^2  - \frac{p \gamma}{2} \|\nabla f(x^k_g)\|^2 
        \\
        &+ \frac{p \gamma}{2} \|\EE_k[g^k] - \nabla f(x^k_g) \|^2 + \frac{L p^2\gamma^2 }{2}\EE_k[\| g^k\|^2]  
    \Big)
    \\
    \notag =&
    (1 + \alpha p \gamma \eta)( 1 - \beta) \| x^k - x^*\|^2 + 2 \gamma \eta \left( \eta - p\right) f(x^{k}_f) + 2 p \gamma \eta f(x^*)
    \\
    &\notag
    + \left((1 + \alpha p \gamma \eta) \beta - p\gamma \eta \mu\right)\|x^k_g - x^*\|^2
    \\
    &\notag
    + (1 + \alpha p \gamma \eta) (\beta^2 - \beta )\|x^k - x^k_g\|^2 
    - p \gamma^2 \eta^2 \|\nabla f(x^k_g)\|^2 
    \\
    &\notag
    + \left( \frac{p \eta \gamma}{\alpha} + p \gamma^2 \eta^2 \right) \|\EE_k[g^k] - \nabla f(x^k_g) \|^2 + \left( p^2 \eta^2 \gamma^2 + p^2 \gamma^3 \eta^2 L \right) \EE_k[\| g^k\|^2]
    \\
    \notag \leq&
    (1 + \alpha p \gamma \eta)( 1 - \beta) \| x^k - x^*\|^2 + 2 \gamma \eta \left( \eta - p\right)  f(x^{k}_f) + 2 p \gamma \eta f(x^*)
    \\
    &\notag
    + \left((1 + \alpha p \gamma \eta) \beta - p\gamma \eta \mu\right)\|x^k_g - x^*\|^2
    \\
    &\notag
    + (1 + \alpha p \gamma \eta) (\beta^2 - \beta )\|x^k - x^k_g\|^2 
    - p \gamma^2 \eta^2 \|\nabla f(x^k_g)\|^2 
    \\
    &\notag
    + p \eta \gamma \left( \frac{1}{\alpha} + \gamma \eta \right) \|\EE_k[g^k] - \nabla f(x^k_g) \|^2 + 2 p^2 \eta^2 \gamma^2 \left(  1 +  \gamma L \right) \EE_k[\| g^k - \nabla f(x^k_g)\|^2] 
    \\
    &\notag
    + 2 p^2 \eta^2 \gamma^2 \left(  1 +  \gamma L \right) \EE_k[\| \nabla f(x^k_g) \|^2]\,.
\end{align*}
In the last step we also used \eqref{eq:inner_prod_and_sqr} with $c = 1$. Since $\gamma \leq \tfrac{3}{4L}$, the choice of $\alpha = \frac{\beta}{2p \eta \gamma}$, $\beta  = \sqrt{4p^2\mu \gamma / 3}$, and $p \mu \gamma \eta = 3\beta /2$ gives
\begin{align*}
    &\beta  = \sqrt{4p^2\mu \gamma / 3} \leq \sqrt{p^2\mu / L} \leq 1,\\
    &(1 + \alpha p \eta \gamma) (1 - \beta) = \left(1 + \frac{\beta}{2}\right) \left( 1 - \beta\right) \leq \left( 1 - \frac{\beta}{2}\right), \\
    &\left((1 + \alpha p \eta \gamma) \beta - p\mu \gamma \eta \right) = \left( \beta + \frac{\beta^2}{2} - p \mu \gamma \eta \right) \leq \left(\frac{3 \beta}{2} - p\mu \gamma \eta \right) \leq 0,
\end{align*}
and, therefore,
\begin{align*}
\E_{k}\bigl[\|x^{k+1} - x^*\|^2 &+ 2\gamma \eta^2 f(x^{k+1}_f)\bigr] 
\\
\leq&
    (1 - \beta / 2) \| x^k - x^*\|^2 + 2 \gamma \eta \left( \eta - p\right)  f(x^{k}_f) + 2 p \gamma \eta f(x^*)
    \\
&\notag
    + p \eta^2 \gamma^2 \left( 1 + 2p /\beta \right) \|\EE_k[g^k] - \nabla f(x^k_g) \|^2
     \\
&\notag
    + 2 p^2 \eta^2 \gamma^2 \left(  1 +  \gamma L \right) \EE_k[\| g^k - \nabla f(x^k_g)\|^2] 
    \\
&\notag
    - p \gamma^2 \eta^2 ( 1 - 2 p (  1 +  \gamma L))  \| \nabla f(x^k_g) \|^2.
\end{align*}
Subtracting $2\gamma \eta^2 f(x^*)$ from both sides, we get
\begin{align*}
    \E_{k}\bigl[\|x^{k+1} - x^*\|^2 &+ 2\gamma \eta^2 (f(x^{k+1}_f) - f(x^*))\bigr]
    \\
    \leq&
    \left( 1 - \beta / 2\right) \| x^k - x^*\|^2 + \left( 1 - p/\eta\right) \cdot 2\gamma \eta^2 (f(x^k_f) - f(x^*))
    \\
    &\notag
    + p \eta^2 \gamma^2 \left( 1 + 2p /\beta \right) \|\EE_k[g^k] - \nabla f(x^k_g) \|^2 
     \\
    &\notag
    + 2 p^2 \eta^2 \gamma^2 \left(  1 +  \gamma L \right) \EE_k[\| g^k - \nabla f(x^k_g)\|^2] 
    \\
    &\notag
    - p \gamma^2 \eta^2 ( 1 - 2 p (  1 +  \gamma L))  \| \nabla f(x^k_g) \|^2.
\end{align*}
Applying \Cref{lem:expect_bound_grad_appendix}, one can obtain
\begin{align*}
    \E_{k}\bigl[\|x^{k+1} - x^*\|^2 &+ 2\gamma \eta^2 (f(x^{k+1}_f) - f(x^*))\bigr]
    \\
    \leq&
    \left( 1 - \beta / 2\right) \| x^k - x^*\|^2 + \left( 1 - p/\eta\right) \cdot 2\gamma \eta^2 (f(x^k_f) - f(x^*))
    \\
    &\notag
    + p \eta^2 \gamma^2 \left( 1 + 2p /\beta \right) \cdot C_{2}\taumix^2 \batchbound^{-2}B^{-2} (\cmax^2 + \dmax^2 \| \nabla f(x^k_g)\|^2) 
    \\
    &\notag
    + 2 p^2 \eta^2 \gamma^2 \left(  1 +  \gamma L \right) \cdot \left(4 C_{1} \taumix B^{-1}\log_2 \batchbound + (4C_{1} + 2) \taumix^2 B^{-2}\right) (\cmax^2 + \dmax^2 \| \nabla f(x^k_g)\|^2) 
    \\
    &\notag
    - p \gamma^2 \eta^2 ( 1 - 2 p (  1 +  \gamma L))  \| \nabla f(x^k_g) \|^2.
\end{align*}
With $M \geq \sqrt{C_2 p^{-1} (1 + 2p/\beta)}$, we have
\begin{align*}
    \E_{k}\bigl[\|x^{k+1} - x^*\|^2 &+ 2\gamma \eta^2 (f(x^{k+1}_f) - f(x^*))\bigr]
    \\
    \leq&
    \left( 1 - \beta / 2\right) \| x^k - x^*\|^2 + \left( 1 - p/\eta\right) \cdot 2\gamma \eta^2 (f(x^k_f) - f(x^*))
    \\
    &\notag
    + p^2 \eta^2 \gamma^2 \taumix^2 B^{-2} (\cmax^2 + \dmax^2 \| \nabla f(x^k_g)\|^2) 
    \\
    &\notag
    + 2 p^2 \eta^2 \gamma^2 \left(  1 +  \gamma L \right) \cdot \left(4 C_{1} \taumix B^{-1}\log_2 \batchbound + (4C_{1} + 2) \taumix^2 B^{-2}\right) (\cmax^2 + \dmax^2 \| \nabla f(x^k_g)\|^2) 
    \\
    &\notag
    - p \gamma^2 \eta^2 ( 1 - 2 p (  1 +  \gamma L))  \| \nabla f(x^k_g) \|^2
    \\
    \leq&
    \left( 1 - \beta / 2\right) \| x^k - x^*\|^2 + \left( 1 - p/\eta\right) \cdot 2\gamma \eta^2 (f(x^k_f) - f(x^*))
    \\
    &\notag
    + 8 p^2 \eta^2 \gamma^2 \left(  1 +  \gamma L \right) \cdot \left(C_{1} \taumix B^{-1}\log_2 \batchbound + (C_{1} + 1) \taumix^2 B^{-2}\right) \cmax^2 
    \\
    &\notag
    - p \gamma^2 \eta^2 \left[ 1 - 2p \left(  1 +  \gamma L \right) \left(1 + 4 \left[C_{1} \taumix B^{-1}\log_2 \batchbound + (C_{1} + 1) \taumix^2 B^{-2}\right] \dmax^2 \right) \right] \| \nabla f(x^k_g) \|^2.
\end{align*}
Since $p = \left[1 + 2\left(  1 +  \gamma L \right) \left(1 + 4 \left[C_{1} \taumix b^{-1} + (C_{1} + 1) \taumix^2 b^{-2}\right] \dmax^2 \right)\right]^{-1}$, $B = \lceil b \log_2 \batchbound \rceil$ and $M \geq 2$, we obtain
\begin{align*}
    p 
    &= \left[1 + 2\left(  1 +  \gamma L \right) \left(1 + 4 \left[C_{1} \taumix b^{-1} + (C_{1} + 1) \taumix^2 b^{-2}\right] \dmax^2 \right)\right]^{-1}
    \\
    &\leq \left[1 + 2\left(  1 +  \gamma L \right) \left(1 + 4 \left[C_{1} \taumix B^{-1} \log_2 \batchbound   + (C_{1} + 1) \taumix^2 B^{-2}\right] \dmax^2 \right)\right]^{-1},
\end{align*}
and then,
\begin{align*}
    \E_{k}\bigl[\|x^{k+1} - x^*\|^2 &+ 2\gamma \eta^2 (f(x^{k+1}_f) - f(x^*))\bigr]
    \\
    \leq&
    \left( 1 - \beta / 2\right) \| x^k - x^*\|^2 + \left( 1 - p/\eta\right) \cdot 2\gamma \eta^2 (f(x^k_f) - f(x^*))
    \\
    &\notag
    + 8 p^2 \eta^2 \gamma^2 \left(  1 +  \gamma L \right) \cdot \left(C_{1} \taumix B^{-1}\log_2 \batchbound + (C_{1} + 1) \taumix^2 B^{-2}\right) \cmax^2 
    \\
    \leq&
    \max\left\{ \left( 1 - \beta / 2\right), \left( 1 - p/\eta\right)\right\} \left[\| x^k - x^*\|^2 + 2\gamma \eta^2 (f(x^k_f) - f(x^*)) \right]
    \\
    &\notag
    + 8 p^2 \eta^2 \gamma^2 \left(  1 +  \gamma L \right) \cdot \left(C_{1} \taumix B^{-1}\log_2 \batchbound + (C_{1} + 1) \taumix^2 B^{-2}\right) \cmax^2 .
\end{align*}
Using that $p \eta \gamma = 3\beta / (2\mu)$, $\beta/2 = p /\eta$, $B = \lceil b \log_2 \batchbound \rceil$ and $\gamma \leq L^{-1}$, we have
\begin{align}
    \label{eq:temp_2345}
    \E_{k}\bigl[\|x^{k+1} - x^*\|^2 &+ 2\gamma \eta^2 (f(x^{k+1}_f) - f(x^*))\bigr]
    \notag\\
    \leq&
    \left( 1 - \beta / 2\right) \left[\| x^k - x^*\|^2 + 2\gamma \eta^2 (f(x^k_f) - f(x^*)) \right]
    \notag\\
    &
    + 36 \beta^2 \mu^{-2} \left(C_{1} \taumix b^{-1} + (C_{1} + 1) \taumix^2 b^{-2}\right) \cmax^2.
\end{align}
Here we also took into account that $M \geq 2$. Finally, we perform the recursion and substitute $\beta = \sqrt{4p^2 \mu \gamma /3}$
\begin{align*}
    \EE\bigl[\|x^{N} - x^*\|^2 &+ 2\gamma \eta^2 (f(x^{N}_f) - f(x^*)) \bigr]
    \\
    \leq&
    \left( 1 - \sqrt{\frac{p^2 \mu \gamma}{3}}\right)^N [\| x^0 - x^*\|^2 + 2\gamma \eta^2 (f(x^0_f) - f(x^*))]
    \\&
    + 72\beta \mu^{-2} \left( C_{1} \taumix b^{-1} + (C_{1} + 1) \taumix^2 b^{-2}\right) \sigma^2 
    \\
    \leq&
    \exp\left( - \sqrt{\frac{p^2 \mu \gamma N^2}{3}}\right) [\| x^0 - x^*\|^2 + 2\gamma \eta^2 (f(x^0_f) - f(x^*))]
    \\&
    + \frac{144p \sqrt{\gamma}}{ \sqrt{3} \mu^{3/2}} \left( C_{1} \sigma^2 \taumix b^{-1} + (C_{1} + 1) \cmax^2  \taumix^2 b^{-2}\right)\,.
\end{align*}
Substituting of $\eta = \sqrt{\tfrac{3}{\mu \gamma}}$ concludes the proof.
\end{proof} 

\subsection{Results of Section \ref{sec:acc-method} with decreasing stepsize} \label{sec:decreas}

The first thing we need to change is to make 
the parameters of Algorithm \ref{alg:AGD_ASGD} depend on the iteration number $k$: $\gamma, p, \beta, \eta, M, B \to \gamma_k, p_k, \beta_k, \eta_k, M_k, B_k$. For this new version of Algorithm \ref{alg:AGD_ASGD} one can reprove Theorem \ref{th:acc}.

\begin{theorem}
Assume \Cref{as:lipsh_grad} -- \Cref{as:bounded_markov_noise_UGE}. Let problem \eqref{eq:erm}
be solved by \Cref{alg:AGD_ASGD}. Then for any $b \in \nsets$,  $\gamma_k \in (0; \tfrac{3}{4L}]$, and $\beta_k, \theta_k, \eta_k, p_k, \batchbound_k, B_k$ satisfying 
\begin{align*}
&\textstyle{p_k \simeq (1 + (  1 +  \gamma_k L) [\delta^2 \taumix b^{-1} + \dmax^2 \taumix^2 b^{-2}])^{-1}, \quad 
\beta_k \simeq \sqrt{p^2_k \mu \gamma_k}, \quad\eta_k \simeq \sqrt{\frac{1}{\mu \gamma_k}}},
\\
&\textstyle{\theta_k \simeq \frac{p_k \eta^{-1}_k - 1}{\beta_k p_k \eta^{-1} - 1}, \quad  M_k \simeq \max\{2; \sqrt{p^{-1}_k (1 + p_k/\beta_k)}\}, \quad B_k = \lceil b \log_2 \batchbound_k \rceil}\,,
\end{align*}
it holds that
\begin{align*}
    \E\bigg[\|x^{k+1} - x^*\|^2 + \frac{6}{\mu} (f(x^{k+1}_f) &- f(x^*))\bigg]
    \\
    \leq&
    \left( 1 - \sqrt{\frac{p_k^2 \mu \gamma_k}{3}}\right) \left[\| x^k - x^*\|^2 + \frac{6}{\mu} (f(x^k_f) - f(x^*)) \right]
    \\
    &\notag
    + \frac{48 p_k^2 \gamma_k }{\mu} \left(C_{1} \taumix b^{-1} + (C_{1} + 1) \taumix^2 b^{-2}\right) \cmax^2.
\end{align*}
\end{theorem}
\begin{proof}
All steps of the proof remain the same with of Theorem \ref{th:acc} and we get \eqref{eq:temp_2345}:
\begin{align*}
    \E\big[\|x^{k+1} - x^*\|^2 + 2\gamma_k \eta^2_k (f(x^{k+1}_f) &- f(x^*))\big]
    \\
    \leq&
    \left( 1 - \beta_k / 2\right) \left[\| x^k - x^*\|^2 + 2\gamma_k \eta^2_k (f(x^k_f) - f(x^*)) \right]
    \\
    &\notag
    + 36 \beta^2_k \mu^{-2} \left(C_{1} \taumix b^{-1} + (C_{1} + 1) \taumix^2 b^{-2}\right) \cmax^2.
\end{align*}
By substituting $\beta_k = \sqrt{4p^2_k \mu \gamma_k /3}$ and $\eta_k = \sqrt{\tfrac{3}{\mu \gamma_k}}$, we finishes the proof.
\end{proof}

Since $p_k = \left[1 + 2\left( 1 + \gamma_k L \right) \left(1 + 4 \left[C_{1} \tau b^{-1} + (C_{1} + 1) \tau^2 b^{-2}\right] \delta^2 \right)\right]^{-1}$ and $\gamma_k \in (0; \frac{3}{4L})$, then $p_k \in [p_l; p_u]$, where $p_l, p_u \sim (1 + (1 + \tau b^{-1} + \tau b^{-2})\delta^2)^{-1}$. It means that we can rewrite the results of the theorem as follows:
\begin{align*}
    \E\bigg[\|x^{k+1} - x^*\|^2 + \frac{6}{\mu} (f(x^{k+1}_f) &- f(x^*))\bigg]
    \\
    \leq&
    \left( 1 - \sqrt{\frac{p_l^2 \mu \gamma_k}{3}}\right) \left[\| x^k - x^*\|^2 + \frac{6}{\mu} (f(x^k_f) - f(x^*)) \right]
    \\
    &\notag
    + \frac{48 p_u^2 \gamma_k }{\mu} \left(C_{1} \taumix b^{-1} + (C_{1} + 1) \taumix^2 b^{-2}\right) \cmax^2.
\end{align*}
With notation $r_{k} = \mathbb{E}\left[\| x^k - x^*\|^2 + \frac{6}{\mu} (f(x^k_f) - f(x^*)) \right]$, $a = \sqrt{p^2_l \mu / 3}$, $\omega_k = \sqrt{\gamma_k}$ and $C = \frac{48 p_u^2}{\mu} \left(C_{1} \tau b^{-1} + (C_{1} + 1) \tau^2 b^{-2}\right) \sigma^2$, one can rewrite the previous estimate:
\begin{align*} 
r_{k+1} \leq \left( 1 - a \omega_k \right) r_{k} + \omega_k^2 C, 
\end{align*} 
where $0 < \omega_k \leq d = \sqrt{3/(4L)}$. For this kind of recursion, we can use the results of Lemma 3 of \cite{stich2019unified}. In particular, we can choose $\gamma_k$ as follows 
\begin{align*}
\label{eq:}
&\text{if } N \leq \frac{d}{a}, && \gamma_k = \frac{1}{d} ,
\\ 
&\text{if } N > \frac{d}{a} \text{ and } k < \left\lceil \frac{N}{2 }\right\rceil, &&\gamma_k = \frac{1}{d} ,
\\ 
& \text{if } N > \frac{d}{a} \text{ and } k \geq \left\lceil \frac{N}{2 }\right\rceil, &&
\gamma_k = \frac{2}{a(k + \frac{2d}{a} + \left\lceil \frac{N}{2 }\right\rceil)}, 
\end{align*} 
and get 
$$r_N = \mathcal{O}\left( \frac{d r_0}{a} \exp\left( - \frac{aN}{2d} \right) + \frac{C}{a^2 N} \right). 
$$ 
But the stepsize still depends on the horizon of iterations $N$. To fix it, we can apply the following restart procedure. We construct a sequence of the iteration number $N_t=2^t$ for $t\geq 0$. For each restart $t$ we set the stepsize $\gamma(N_t)$ according to Lemma 3 of \cite{stich2019unified}, run the algorithm for $N_t$ basic iterations. If we do not achieve the unknown horizon of the total iteration number $N$, then we use the obtained point as a warm-start for the next restart. For simplicity, we can also use the same starting $x^0$ point for all the restarts.
Let us now assume that the algorithm made $N$ iterations. This means that it made at least $T = \lfloor \log_2(N + 1)\rfloor$ finished restarts. Since at the end of the last restart it made $N_T$ basic iterations with the stepsize $\gamma(N_T)$, we can guarantee that 
$$
r_{N_T} = \mathcal{O}\left( \frac{d r_0}{a} \exp\left( - \frac{aN_T}{2d} \right) + \frac{C}{a^2 N_T} \right).
$$
One can note that $N_T \sim N$, then 
$$ 
r_{N_T} = \mathcal{O}\left( \frac{d r_0}{a} \exp\left( - \frac{aN}{2d} \right) + \frac{C}{a^2 N} \right). 
$$ 
This algorithm does not require to fix the number of basic steps $N$ in advance, but if we want to have $\varepsilon$-solution in terms of $r_N$, then we have the following estimate on the number of iterations: 
$$
N = \mathcal{O}\left( \frac{d}{a} \log \frac{1}{\varepsilon} + \frac{C}{a^2 \varepsilon}\right) 
= 
\mathcal{O}\left( \left[1 + (1 + \tau b^{-1} + \tau^2 b^{-2} ) \delta^2\right] \sqrt{\frac{L}{\mu}} \log \frac{1}{\varepsilon} + \frac{\sigma^2}{\mu^2 \varepsilon} \left( \tau b^{-1} + \tau^2 b^{-2}\right) \right). 
$$
To get the close to Corollary \ref{cor:sample_complexity_accelerated} results on the oracle complexity one need to take $b = \tau$ and note that now $B_k = b \log_2 M_k = b \log_2 M_k \sim b \log_2 N \sim b \log_2 \varepsilon^{-1}$. Finally, 
it gives additional logarithmic factor in the estimate for the oracle complexity. But this factor does not really change the bound and it means that we obtain the result of Corollary \ref{cor:sample_complexity_accelerated}.

\subsection{Lower bounds proofs}
\label{sec:lower_bounds}
\textbf{Proof of \Cref{th:strongly_convex_lower_bound}.}
\label{sec:proof:th:strongly_convex_lower_bound}

We begin the proof with two lemmas, showing the lower bounds for deterministic and stochastic components of the error separately. Then we combine the two in \Cref{th:strongly_convex_lower_bound_append} and complete the proof of \Cref{th:strongly_convex_lower_bound}. First, we consider the lower bound for the deterministic part of the error, and construct a problem with $\dmax = 1$ and $\cmax = 0$. 
\begin{lemma}
\label{lem:lower_bound_determ}
There exists an instance of the optimization problem satisfying assumptions \Cref{as:lipsh_grad} --\Cref{as:bounded_markov_noise_UGE} with $\dmax = 1$ and $\cmax = 0$, such that for any first-order gradient method it takes at least 
\[
N = \Omega\bigl(\taumix \sqrt{\frac{L}{\mu}} \log\frac{1}{\varepsilon}\bigr)
\]
oracle calls in order to achieve $\E[\norm{x^N - x^{*}}^2] \leq \varepsilon$. 
\end{lemma}
\begin{proof}
Consider the optimization problem
\begin{equation}
\label{eq:problem_form_hard}
f_1(x) = \frac{\mu(Q-1)}{4} \bigl(\frac{x^{\top} A x}{2} - e_{1}^{\top}x\bigr) + \frac{\mu}{2}\norm{x}^{2} \to \min_{x \in \rset^{d}}\,,
\end{equation}
where $x \in \rset^{d}$, $\mu > 0, Q > 1$, dimension $d$ is even, $d = 2u$, $e_{1} = (1,0,\ldots,0) \in \rset^{d}$ is the first coordinate vector, and $A \in \rset^{d \times d}$ is a symmetric nonnegative-definite matrix   given by
\begin{equation}
\label{eq:matrix_A_def}
A = \begin{pmatrix}
 2 & -1 & 0 & 0 & \ldots & 0 &  0 \\
 -1 & 2 & -1 & 0 & \ldots & 0 & 0 \\
0 &  -1 & 2 & -1 & \ldots & 0 & 0\\
 & & & & \ldots & \\
 0 &  0 & 0 & 0 & \ldots & 2 & -1\\
 0 &  0 & 0 & 0 & \ldots & -1& \alpha\,,
\end{pmatrix}
\end{equation}
where $\alpha = \tfrac{\sqrt{Q}+3}{\sqrt{Q}+1}$. Straightforward calculations (see e.g. \cite[Chapter~5.1.4]{lan20} for more details) yield
\[
0 \preceq A \preceq 4 I\,, \nabla f_1(x) = \frac{\mu(Q-1)}{4}Ax - e_1 + \mu x\,.
\]
Thus the problem \eqref{eq:problem_form_hard} is $L-$smooth with $L = \mu Q$ and $\mu$-strongly convex, i.e., the assumptions \Cref{as:lipsh_grad} and \Cref{as:strong_conv} are satisfied, and the corresponding condition number is equal to $L/\mu = Q$. For $\epsilon \in (0;1/2)$ we now consider the two-state Markov transition matrix (or kernel)
\begin{equation}
\label{eq:Mark_kernel_def}
\MK_1 = \begin{pmatrix}
 1-\epsilon & \epsilon \\
 \epsilon & 1-\epsilon 
\end{pmatrix}
\end{equation}
and denote by $(Z_{i})_{i=1}^{\infty}$ the corresponding Markov Chain, $Z_i \in \{-1,1\}$. It is easy to see that the Markov kernel $\MK$ is uniformly geometrically ergodic and satisfies \Cref{as:bounded_markov_noise_UGE} with $\taumix \leq \epsilon^{-1} \log{4}$.  It is easy to check that the corresponding invariant distribution is $\pi = (1/2,1/2)$. For $Z \in \{-1,1\}$ we now consider the noise matrix
\[
W(Z) = 2 \operatorname{diag}\bigl\{\indiacc{Z = 1}, \indiacc{Z = -1},\indiacc{Z = 1}, \ldots, \indiacc{Z = -1}\bigr\} \in \rset^{d \times d}\,.
\]
Now for $x \in \rset^{d}$ and $Z \in \{-1,1\}$ we define the stochastic gradient oracle as 
\begin{equation}
\label{eq:F_1_oracle_def}
\nabla F_1(x,Z) = W(Z) \nabla f(x)\,.
\end{equation}
It is easy to check that $\E_{\pi}[W(Z)] = \Id$, and the direct calculations imply $\| \nabla F_1(x,Z) - \nabla f_1(x)\| \leq \|\nabla f_1(x)\|$, that is, the assumption \Cref{as:bounded_markov_noise_UGE} holds with $\dmax = 1$ and $\cmax = 0$. Following \cite{nesterov2003introductory}, \cite[Chapter~5.1.4]{lan20}, the solution to the minimization problem \eqref{eq:problem_form_hard} is given by
\begin{equation}
\label{eq:exact_solution_q}
x^{*} = (q^{1},\ldots,q^{d}) \in \rset^{d}\,, q = \tfrac{\sqrt{Q}-1}{\sqrt{Q}+1}\,.
\end{equation}
Suppose that we start from $Z_1 = 1$ and initial point $x_0 = 0 \in \rset^{d}$. Then after $1$ oracle call we observe the $1$-st coordinate of $x$. At the same time, the second component can not be computed until the time moment $T_2 = \inf\{i \in \nset: Z_i = -1\}$. Similarly, the next computation of the $3$-rd component of the solution requires the chain to go back to state $1$ and can not happen earlier then $T_3 = \inf\{i \geq \tau_2: Z_i = 1\}$. Thus, after $k$ iterations of any first-order method, the respective MSE is lower bounded by 
\[
\E_{\pi}[\norm{x^k - x^{*}}^{2}] \geq \E_{\pi}\left[ \indiacc{Z_1 = 1} \sum_{i = N_k}^{d}q^{2i}\right] = \frac{1}{2} \E_{\delta_{1}}\left[\frac{q^{2N_k} - q^{2d}}{1-q^2}\right]\,.
\]
In the formula above we denoted by $N_k$ the number of state changes in the sequence $(Z_i)_{i=1}^{k}$. Using Jensen's inequality and the explicit construction of the Markov kernel $\MK_1$ in \eqref{eq:Mark_kernel_def}, we deduce that
\begin{align*}
\E_{\pi}[\norm{x^k - x^{*}}^{2}] 
&\geq \frac{1}{2} \frac{q^{2\E_{\delta_{1}}[N_k]} - q^{2d}}{1-q^2} = \frac{1}{2} \frac{q^{2 (k-1) \epsilon} - q^{2d}}{1-q^2} \geq (1/2)(1-q^2)^{-1}q^{(2 / \log 4) k / \taumix} = \\
&= (1/2) (1-q^2)^{-1} \left(1 - \frac{2}{\sqrt{Q}+1}\right)^{(2 / \log 4) k / \taumix} \\
&\geq (1/2) (1-q^2)^{-1} \exp\left(-\frac{8k}{(\sqrt{Q}+1) \taumix \log{4}}\right)\,,
\end{align*}
provided that $d$ is large enough. In the last inequality we also used that $1 - x \geq e^{-2x}$ for $x \in [0;1/2]$. Hence, taking into account that $Q$ is the condition number of the problem \eqref{eq:problem_form_hard}, we get the desired lower bound.
\end{proof}

Now we consider an instance of the problem with $\dmax = 0$, arbitrary $\cmax \geq 0$, and construct the respective lower bound for the stochastic part of the error.

\begin{lemma}
\label{lem:lower_bound_stochastic}
There exists an instance of the optimization problem satisfying assumptions \Cref{as:lipsh_grad} --\Cref{as:bounded_markov_noise_UGE} with $\dmax = 0$ and arbitrary $\sigma \geq 0$, such that for any first-order gradient method it takes at least 
\[
N = \Omega\left(\frac{\taumix \sigma^2}{\mu^2 \varepsilon}\right)
\]
oracle calls in order to achieve $\E[\norm{x^N - x^{*}}^2] \leq \varepsilon$.
\end{lemma}
\begin{proof}
Our proof is based on a simple $1$-dimensional optimization problem and Le Cam's lemma \cite[Theorem~8]{aamari_levrard2019}, see also \cite{Yu1997}. Consider the following minimization problem
\begin{equation}
\label{eq:f_2_stochastic}
f_2(x) = \frac{\mu}{2}(x-x^*)^2 \mapsto \min_{x \in \rset}\,.
\end{equation}
Obviously this problem satisfies \Cref{as:strong_conv} with strong convexity constant $\mu$ and \Cref{as:lipsh_grad} with $L = \mu$. Consider the noisy gradient oracle
\begin{equation}
\label{eq:F_2_oracle_def}
\nabla F_2(x,Y) = \mu (x-x^*) + \frac{\sigma}{2}Y\,,
\end{equation}
where $Y$ is a noise variable taking values $Y \in \{-1,1\}$. For now we do not specify the distribution of $Y$, yet we easily note that for any distribution $\pi$ on $\{-1,1\}$ we have
\[
\norm{\nabla F_2(x,Y) - \E_{\pi}\nabla F_2(x,Y)}^2 \leq \sigma^2\,.
\]
Consider the sequence of noise variables $(Y_i)_{i=1}^{n}$ with the joint distribution to the specified later, and any sequence of design points $(x_i)_{i=1}^{n}$, where the resulting gradients are evaluated. At this point the statistician observes the gradients
\[
(\mu(x_i-x^*) + \frac{\sigma}{2}Y_i)\,, i = 1\ldots,n\,,
\]
and, since $x_i$ and $\mu$ are known, this is equivalent to observing
\[
x^* - \frac{\sigma}{2\mu}Y_i\,, i = 1\ldots,n.
\]
Now we aim to construct to "almost indistinguishable" models for the noise variables $Y_i$. Namely, we consider the parametric family of Markov kernels
\begin{equation}
\label{eq:2_sample_MK}
\MK_{\varphi} = \begin{pmatrix}
 1- \epsilon & \epsilon \\
 \epsilon + \varphi & 1 - \epsilon - \varphi
\end{pmatrix}\,, 
\end{equation}
where the parameters $\varphi, \epsilon \in (0;1/4)$, and $\varphi \in [0;\alpha]$, and the parameter $\alpha$ will be set depending on $\epsilon$ and $n$ later. It is easy to check that the invariant distribution of the Markov kernel $\MK_{\varphi}$ is given by
\[
\pi^{\varphi} = \bigl(\frac{\epsilon + \varphi}{2\epsilon + \varphi},\frac{\epsilon}{2\epsilon + \varphi}\bigr)\,.
\]
Now we consider the setting of Le Cam's lemma \cite[Theorem~8]{aamari_levrard2019}. Namely, for a fixed sample size $n$ we consider the family of Markov kernels $(\MK_{\varphi})_{\varphi \in [0;\alpha]}$, and family of corresponding joint $n-$step distributions under stationarity, that is, $\pi^{\varphi}\MK_{\varphi}^{\otimes n}$. The reader not familiar with the respective notation could find it, in particular, in \cite[Chapter~1]{douc:moulines:priouret:soulier:2018}. As a parameter of interest we consider the expectation
\begin{equation}
\label{eq:theta_varphi}
\theta(\varphi) := \theta(\pi^{\varphi}\MK_{\varphi}^{\otimes n}) := \E_{\pi^{\varphi}}[x^* - \frac{\sigma}{2\mu}Z_i] = x^* - \frac{\cmax \varphi}{2\mu(2\epsilon + \varphi)}\,.
\end{equation}
Now we consider the $2$ representatives of the above class, that is, the $n$-step distributions corresponding the parameters $\varphi = 0$ and $\varphi = \alpha$. Then the direct application of Le Cam's lemma \cite[Theorem~8]{aamari_levrard2019} yields
\begin{equation}
\label{eq:le_cam_lemma}
\inf_{\widehat{\theta}}\sup_{\varphi \in [0; \alpha]}\E^{1/2}_{\pi^{\varphi}\MK_{\varphi}^{\otimes n}}[|\widehat{\theta} - \theta(\varphi)|^2] \geq \frac{1}{2}|\theta(0) - \theta(\alpha)|(1 - \tvnorm{\pi^{0}\MK_{0}^{\otimes n} - \pi^{\alpha}\MK_{\alpha}^{\otimes n}})\,,
\end{equation}
where $\widehat{\theta} = \widehat{\theta}(Y_1,\ldots,Y_n)$ is any measurable function. Thus, taking square and using the definition of $\theta(\varphi)$ in \eqref{eq:theta_varphi}, we obtain that
\begin{equation}
\label{eq:le_cam_1}
\inf_{\widehat{\theta}}\sup_{\varphi \in [0; \alpha]}\E_{\pi}[|\theta - \theta(\varphi)|^2] \geq \frac{\cmax^2\alpha^2}{16\mu^2(2\epsilon+\alpha)^2}(1 - \tvnorm{\pi^{0}\MK_{0}^{\otimes n} - \pi^{\alpha}\MK_{\alpha}^{\otimes n}})\,.
\end{equation}
Now we set $\alpha = \sqrt{\tfrac{\epsilon}{n}}$ and apply the statement of \Cref{lem:kl_divergence_estimate} with this choice of $\alpha$. Note that we impose at this point the regularity condition $n \geq \epsilon^{-1}$ in order to have $\alpha \leq \epsilon$. Thus we get
\[
\inf_{\widehat{\theta}}\sup_{\varphi \in [0; \sqrt{\tfrac{\epsilon}{n}}]}\E_{\pi}[|\widehat{\theta} - \theta(\varphi)|^2] \geq \frac{\cmax^2 \epsilon}{32 \mu^2 n (2\epsilon + \sqrt{\tfrac{\epsilon}{n}})^2} \geq \frac{\cmax^2}{288 \mu^2 n \epsilon}\,,
\]
and the statement follows by noticing that the corresponding mixing time $\taumix \leq c \epsilon^{-1}$ for some $c > 0$ (see e.g. \cite[Proposition~1]{nagaraj2020least}).
\end{proof}

\begin{lemma}
\label{lem:kl_divergence_estimate}
Consider the family of Markov kernels 
\[
\MK_{\varphi} = \begin{pmatrix}
 1-\epsilon & \epsilon \\
 \epsilon + \varphi & 1 - \epsilon - \varphi
\end{pmatrix}
\]
and the corresponding invariant distributions $\pi^{\varphi} = \bigl(\frac{\epsilon + \varphi}{2\epsilon + \varphi},\frac{\epsilon}{2\epsilon + \varphi}\bigr)$ for $\varphi \in \{0,\alpha\}$. Then it holds that 
\[
\tvnorm{\pi^{0}\MK_{0}^{\otimes n} - \pi^{\alpha}\MK_{\alpha}^{\otimes n}} \leq \frac{1}{2}\sqrt{\frac{n \alpha^2}{\epsilon}}\,.
\]
\end{lemma}
\begin{proof}
Note first that an application of Pinsker's inequality yields
\[
\tvnorm{\pi^{0}\MK_{0}^{\otimes n} - \pi^{\alpha}\MK_{\alpha}^{\otimes n}} \leq \sqrt{(1/2)\KL(\pi^{0}\MK_{0}^{\otimes n} || \pi^{\alpha}\MK_{\alpha}^{\otimes n})}\,.
\]
Using the chain rule for KL-divergence, we get 
\begin{equation}
\label{eq:cond_kl_chain}
\KL(\pi^{0}\MK_{0}^{\otimes n} || \pi^{\alpha}\MK_{\alpha}^{\otimes n}) = \KL(\pi^{0} || \pi^{\alpha}) + \sum_{i=1}^{n-1}\sum_{y \in \{-1,1\}}\MK_{\pi^{0}\MK_{0}^{\otimes n}}(Y_i = y)\KL(\MK_0(\cdot|y)||\MK_{\alpha}(\cdot|y))\,.
\end{equation}
In the notation above for $y \in \{-1,1\}$ we have set $\KL(\MK_0(\cdot|y)||\MK_{\alpha}(\cdot|y))$ for the $1-$step conditional distribution
\[
\KL(\MK_0(\cdot|y)||\MK_{\alpha}(\cdot|y)) = \sum_{x \in \{-1,1\}}\MK_0(x|y)\log\frac{\MK_0(x|y)}{\MK_{\alpha}(x|y)}\,.
\]
Now an application of reversed Pinsker's inequality together with $\alpha \leq \epsilon$ yields that 
\[
\KL(\MK_0(\cdot|y)||\MK_{\alpha}(\cdot|y)) \leq \frac{\alpha^2}{2\epsilon}\,,
\]
and the bound \eqref{eq:cond_kl_chain} implies that
\[
\KL(\pi^{0}\MK_{0}^{\otimes n} || \pi^{\alpha}\MK_{\alpha}^{\otimes n}) \leq \frac{n\alpha^2}{2\epsilon}\,.
\]
Combining the bounds above yields the statement.
\end{proof}

Now we are ready to combine the bounds above and prove \Cref{th:strongly_convex_lower_bound}.

\begin{theorem}[\Cref{th:strongly_convex_lower_bound}]
\label{th:strongly_convex_lower_bound_append}
There exists an instance of the optimization problem satisfying assumptions \Cref{as:lipsh_grad} --\Cref{as:bounded_markov_noise_UGE} with $\dmax = 1$ and arbitrary $\sigma \geq 0, L, \mu > 0, \taumix \in \nsets$, such that for any first-order gradient method it takes at least 
\[
\textstyle{
N = \Omega\bigl(\taumix \sqrt{\frac{L}{\mu}} \log\frac{1}{\varepsilon} + \frac{\taumix \sigma^2}{\mu^2 \varepsilon}\bigr)
}
\]
oracle calls in order to achieve $\E[\norm{x^N - x^{*}}^2] \leq \varepsilon$. 
\end{theorem}

\begin{proof}
We split the original problem into two parts. Indeed, for any $d \in \nsets$ we consider $x = (x_{det},x_{stoch}) \in \rset^{d+1}$, where $x_{det} \in \rset^{d}$ and $x_{stoch} \in \rset$. Now we consider the minimization problem
\begin{equation}
\label{eq:combined_eq}
f(x) = f(x_{det},x_{stoch}) = f_1(x_{det}) + f_2(x_{stoch}) \to \min_{x \in \rset^{d+1}}\,,
\end{equation}
where the functions $f_1: \rset^{d} \to \rset$ and $f_2: \rset \to \rset$ are defined in \eqref{eq:problem_form_hard} and \eqref{eq:f_2_stochastic}, respectively. We fix the respective parameters $\mu, Q,$ and $\cmax$. Applying \Cref{lem:lower_bound_determ} and \Cref{lem:lower_bound_stochastic}, we get that the respective problem \eqref{eq:combined_eq} is $L$-smooth and $\mu$-strongly convex with $L = \mu Q$ and parameter $Q > 1$ defined in \eqref{eq:problem_form_hard}. For $Z,Y \in \{-1,1\}$ we define the stochastic gradient oracle as 
\[
\nabla F(x,Z,Y) = (\nabla F_1(x_{det},Z), \nabla F_2(x_{stoch},Y)) \in \rset^{d+1}\,.
\]
The oracles $\nabla F_1(x_{det},Z)$ and $\nabla F_2(x_{stoch},Y)$ are defined in \eqref{eq:F_1_oracle_def} and \eqref{eq:F_2_oracle_def}, respectively. \Cref{lem:lower_bound_determ} and \Cref{lem:lower_bound_stochastic} imply that \Cref{as:bounded_markov_noise_UGE} holds with $\dmax = 1$ and $\cmax > 0$ defined in \eqref{eq:F_2_oracle_def}. Consider now the Markov chains $(Z_{i})_{i=1}^{\infty}$ with the transition kernel $\MK_1$ defined in \eqref{eq:Mark_kernel_def} and $(Y_{i})_{i=1}^{\infty}$ with the transition kernel $\MK_\varphi$ of the form \eqref{eq:2_sample_MK}. As in the proof of \Cref{lem:lower_bound_stochastic}, we take $\varphi \in [0; \sqrt{\epsilon/n}]$ and assume that $n \geq \epsilon^{-1}$. Consider the joint process $(X_i,Y_i)_{i=1}^{\infty}$ of independently evolving Markov chains $(Z_{i})_{i=1}^{\infty}$ and $(Y_{i})_{i=1}^{\infty}$. It is easy to see that such a process is a Markov chain on $\{-1,1\}^{2}$ with the transition kernel 
\[
\MK = \MK_1 \otimes \MK_{\varphi}\,,
\]
where $\otimes$ stands for the Kronecker's product. In is clear that $\MK$ is irreducible and aperiodic, hence the assumption \Cref{as:Markov_noise_UGE} holds. Note that both $\MK_1$ and $\MK_{\varphi}$ are reversible (see e.g. \cite{paulin_spectral}[Section~3.1] for the respective definitions). Thus their Kronecker's product is also reversible, with the spectrum given by the pairwise products of eigenvalues of $\MK_1$ and $\MK_{\varphi}$. Hence, with the direct calculations, we compute the eigenvalues of $\MK$: $\{1,1-2\epsilon-\varphi,1-2\epsilon,(1-2\epsilon)(1-2\epsilon-\varphi)\}$. Thus the corresponding spectral gap $\gamma = 2\epsilon$, and the mixing time $\taumix$ of $\MK$ is bounded by 
\[
\frac{1}{2\epsilon (1 + 1/\log{2})} \leq \taumix \leq \frac{2 \log{2} + \log{6}}{4\epsilon}\,,
\]
see \cite{paulin_spectral}[Proposition~3.3]. Hence, the mixing time of the corresponding joint chain scales as $\epsilon^{-1}$, as for $(Z_{i})_{i=1}^{\infty}$ and $(Y_{i})_{i=1}^{\infty}$ separately. On the $k$-th step of the stochastic gradient computations we rely on the stochastic gradient
\[
\nabla F(x_k,Z_k,Y_k)\,,
\]
computed using the pair $(Z_k,Y_k)$. To complete the proof it remains to apply the complexity results of \Cref{lem:lower_bound_determ} and \Cref{lem:lower_bound_stochastic} to the parts $x_{det}$ and $x_{stoch}$, respectively.
\end{proof}

\begin{proposition}[\Cref{prop:regression_lower_bound}]
There exists an instance of the optimization problem satisfying assumptions \Cref{as:lipsh_grad} --\Cref{as:bounded_markov_noise_UGE} with arbitrary $L, \mu > 0, \taumix \in \nsets$, $\dmax = \tfrac{L}{\mu}$, and $\cmax = 0$, such that for any first-order gradient method it takes at least 
\[
\textstyle{N = \Omega\bigl(\taumix \frac{L}{\mu} \log\frac{1}{\varepsilon}\bigr)}
\]
gradient calls in order to achieve $\E[\norm{x^N - x^{*}}^2] \leq \varepsilon$.
\end{proposition}
\begin{proof}
In this part we closely follow the setting of \cite{nagaraj2020least}. We consider the setting of linear regression:
\begin{equation}
\label{eq:least_squares_opt}
f(x) = \frac{1}{2}\E_{(\varphi,Y) \sim \mathcal{D}}\bigl[|Y - \varphi^{\top}x|^{2}\bigr] \to \min_{x}\,,
\end{equation}
where $\varphi \in \rset^{d}$ is a (random) feature vector, $Y \in \rset$ is a (random) regressor, with the joint distribution $(\varphi,Y) \sim \mathcal{D}$, and $x \in \rset^{d}$ is the optimized parameter. We consider the so-called realizable case, that is, we assume that 
\[
Y = \varphi^{\top}x^{*}
\]
for some vector $x^{*} \in \rset^{d}$. In this scenario the problem \eqref{eq:least_squares_opt} reduces to 
\[
f(x) = \frac{1}{2}(x - x^{*})^{\top} \Sigma^{2} (x - x^{*}) \to \min_{x}\,,
\]
where we have denoted $\Sigma^2 = \E_{\pi}[\varphi \varphi^{\top}]$. This means that the exact gradient is given by $\nabla f(x) = \Sigma^{2} (x - x^{*})$. Now we consider the stochastic setting of the online regression with sequentially observed data points $(\varphi_i,Y_i)_{i=1}^{N}$ with $Y_i = \varphi_i^{\top}x^{*}$. In this case the $i$-th realization of stochastic gradient at point $x \in \rset^{d}$ is given by 
\[
\nabla F(x, \varphi_i, Y_i) = \varphi_i(\varphi_i^{\top}x - Y_i) = \varphi_i \varphi_i^{\top} (x - x^{*})\,.
\]
Hence, with a simple algebra we get 
\[
\| \nabla F(x, \varphi_i, Y_i) - \nabla f(x)\| = \| (\Sigma^2 - \varphi_i \varphi_i^{\top})(x-x^{*})\| = \| (\Id - \varphi_i \varphi_i^{\top}\Sigma^{-2})\nabla f(x)\|\,, 
\]
where we have used the fact that $x - x^{*} = (\Sigma^{2})^{-1}\nabla f(x)$ and used additional notation $\Sigma^{-2} := (\Sigma^{2})^{-1}$. Fix now the condition number $Q > 1$, parameter $\epsilon \in (0;1/4)$ and consider the Markov kernel
\[
\MK = \begin{pmatrix}
 1- \frac{\epsilon}{Q-1} & \frac{\epsilon}{Q-1} \\
 \epsilon & 1- \epsilon
\end{pmatrix}
\]
and the corresponding canonical chain $(Z_i)_{i=1}^{N}$. The invariant distribution of $\MK$ is given by $\pi = (1 - 1/Q, 1/Q)$, and the corresponding mixing time $\taumix$ is bounded by
\[
\taumix \leq \frac{(Q-1)\log{4}}{Q\epsilon}\,,
\]
see e.g. \cite[Proposition~1]{nagaraj2020least}. We let $\varphi = \varphi(Z)$, and w.l.o.g. we can assume that $Z \in \{-1,1\}$. Consider 
\[
\varphi(1) = (1,0)\,, \quad \varphi(-1) = (0,1)\,.
\]
The design matrix $\Sigma^2$ is given by 
\[
\Sigma^2 = \E_{\pi}[\varphi(Z_i) \varphi(Z_i)^{\top}] = \begin{pmatrix}
 1-1/Q & 0 \\
 0 & 1/Q
\end{pmatrix}\,,
\]
which implies that \Cref{as:lipsh_grad} and \Cref{as:strong_conv} are satisfied with $\mu = 1/Q$ and $L = 1 - 1/Q$. Then the direct calculations yield 
\[
\| \nabla F(x, \varphi(Z_i), Y_i) - \nabla f(x)\| \leq (Q - 1) \|\nabla f(x)\|\,,
\]
and the assumption \Cref{as:bounded_markov_noise_UGE} is satisfied with $\dmax = Q - 1$. Then the direct application of lower bound \cite{nagaraj2020least} implies the lower bound
\[
\E_{\pi}[\|x^k - x^*\|^2] \geq \exp\left(-\frac{c k}{Q \taumix}\right)
\]
after $k$ iterations of any first-order method with Markovian sampling oracle defined above. Here $c > 0$ is some absolute positive constant not dependeing upon $\taumix$ and $Q$. This means that the instance-dependent increase of $\dmax$ yields to inevitably slower convergence rates.
\end{proof}

\begin{proposition}[\Cref{prop:lower_bound_delta_0}]
There exists an instance of the optimization problem satisfying assumptions \Cref{as:lipsh_grad} --\Cref{as:bounded_markov_noise_UGE} with with arbitrary $L, \mu > 0, \taumix \in \nsets$, $\cmax = 1, \dmax = 0$, such that for any first-order gradient method it takes at least 
\[
\textstyle{
N = \Omega\left(\left(\taumix + \sqrt{\frac{L}{\mu}}\right) \log\{\frac{1}{\varepsilon}\}\right)
}
\]
oracle calls in order to achieve $\E[\norm{x^N - x^{*}}^2] \leq \varepsilon$.
\end{proposition}
\begin{proof}
Let us consider the same minimization problem \eqref{eq:problem_form_hard} as in the proof of \Cref{th:strongly_convex_lower_bound}. Recall that the true gradient in this setting is given by
\[
\nabla f(x) = \frac{\mu(Q-1)}{4}Ax - e_1 + \mu x\,.
\]
Hence the problem \eqref{eq:problem_form_hard} is $L-$smooth with $L = \mu Q$ and $\mu$-strongly convex, that is, assumptions \Cref{as:lipsh_grad} and \Cref{as:strong_conv} are satisfied, and the corresponding condition number equals $L/\mu = Q$. Now for $\epsilon \in (0;1/2)$ we consider the discrete-state space Markov kernel 
\begin{equation}
\label{eq:Mark_kernel_def_same}
\MK = \begin{pmatrix}
 1-\epsilon & \epsilon \\
 \epsilon & 1-\epsilon 
\end{pmatrix}
\end{equation}
and the corresponding Markov Chain $(Z_{i})_{i=1}^{\infty}$. It is easy to see that the Markov kernel $\MK$ is uniformly geometrically ergodic and satisfies \Cref{as:bounded_markov_noise_UGE} with $\taumix \leq \epsilon^{-1} \log{4}$. Each $Z_i$ takes $2$ different values, and w.l.o.g. we can assume that $Z_i \in \{-1,1\}$. It is easy to check that the corresponding invariant distribution is $\pi = (1/2,1/2)$. For $Z \in \{-1,1\}$ we now consider the noisy oracle
\[
\nabla F(x,Z) = \frac{\mu(Q-1)}{4}Ax - (1 + \indiacc{Z = -1} - \indiacc{Z = 1}) e_1 + \mu x\,.
\]
It is easy to check that $\E_{\pi}[\nabla F(x,Z)] = \nabla f(x)$, and the direct calculations imply $\| \nabla F(x,Z) - \nabla f(x)\| \leq 1$, that is, the assumption \Cref{as:bounded_markov_noise_UGE} holds with $\dmax = 0$ and $\cmax = 1$. Suppose that we start from $Z_1 = 1$ and initial point $x_0 = 0 \in \rset^{d}$. Then we observe $\nabla F(x,Z) = 0 \in \rset^{d}$ unless the time moment $T_2 = \inf\{i \in \nset: Z_i = -1\}$. Thus, after $k$ iterations of any first-order method, the respective MSE is lower bounded by 
\begin{align*}
\E_{\pi}[\norm{x^k - x^{*}}^{2}] &\geq \E_{\pi}\left[ \indiacc{Z_1 = 1} \sum_{i = k}^{d}q^{2i}\right] + \E_{\pi}\left[ \indiacc{Z_1 = 1, T_2 \geq k}(1-q^{2d})\right] \\
&\geq \frac{1}{2} (1-q^2)^{-1} (q^{2k} - q^{2d}) + \frac{1}{2} (1-q^2)^{-1} \PP_{\delta_{1}}(T_2 \geq k) \\
&= \frac{1}{2} (1-q^2)^{-1} (q^{2k} - q^{2d}) + \frac{1}{2} (1-q^2)^{-1} (1-\epsilon)^{k-1}\,. 
\end{align*} 
Hence, with the defition of $q$ in \eqref{eq:exact_solution_q}, we get from the previous bound that 
\[
\E_{\pi}[\norm{x^k - x^{*}}^{2}] \geq \frac{1}{2} (1-q^2)^{-1} \left[\exp\left(-\frac{4k}{(\sqrt{Q}+1)}\right) - q^{2d}\right] + \frac{1}{2} (1-q^2)^{-1} \exp\left(-\frac{2k}{\taumix \log{4}}\right)\,,
\]
where in the last inequality we also used that $1 - x \geq e^{-2x}$ for $x \in [0;1/2]$. Now the statement follows from the definition of $Q = L/\mu$.
\end{proof}

\subsection{Proof of \Cref{th:non_convex_random_batch}} \label{sec:proof:th:non_convex_upper_bound}

\begin{theorem}[\Cref{th:non_convex_random_batch}]
Assume \Cref{as:lipsh_grad},~\Cref{as:Markov_noise_UGE},~\Cref{as:bounded_markov_noise_UGE}. Let problem \eqref{eq:erm}
be solved by \Cref{alg:Rand-GD}. Let $f^*$ be a global (maybe not unique) minimum of $f$. Then for any $b \in \nsets$, and $\gamma$, $M$ satisfying
\begin{eqnarray*}
&\gamma \leq \left[4L \left( 1 +  4\left[C_{1} \taumix b^{-1} + (C_{1} + 1) \taumix^2 b^{-2}\right] \dmax^2\right)\right]^{-1},
\\
&M = \max\{2;\sqrt{C_2 \gamma^{-1} L^{-1}}\}, \quad B = \lceil b \log_2 \batchbound \rceil,
\end{eqnarray*}
it holds that
\[
\EE\left[ \frac{1}{\nbiter}\sum_{k=0}^{\nbiter-1}\| \nabla f(x^k)\|^2 \right] \lesssim \frac{f(x^0) - f^*}{\gamma \nbiter} + L \gamma  \cdot \left[ \sigma^2 \taumix b^{-1} +\cmax^2 \taumix^2 b^{-2} \right]\,.
\]
\end{theorem}
\begin{proof}
We start from \Cref{as:lipsh_grad} (in the form \eqref{eq:smothness} with $x = x^{k+1}$ and $y = x^k$) and line \ref{line_gd_3} of \Cref{alg:Rand-GD}:
\begin{align*}
f(x^{k+1}) &\leq f(x^k) + \langle \nabla f(x^k), x^{k+1} - x^k \rangle + \frac{L}{2} \| x^{k+1} - x^k\|^2 \\
&\leq f(x^k) - \gamma \langle \nabla f(x^k), g^k \rangle + \frac{\gamma^2 L}{2} \| g^k\|^2
\\
&=
        f(x^k) - \gamma  \langle \nabla f(x^k), \nabla f(x^k) \rangle - \gamma \langle \nabla f(x^k), \EE_k[g^k] - \nabla f(x^k) \rangle 
        \\
        &
        - \gamma \langle \nabla f(x^k), g^k - \EE_k[g^k] \rangle + \frac{L \gamma^2 }{2}\EE_k [\| g^k\|^2]
\end{align*}
Subtracting $f^*$ from both sides, using Cauchy Schwartz inequality \eqref{eq:inner_prod} and taking the conditional expectation, we get
\begin{align*}
\EE_k[f(x^{k+1}) - f^*]
        \leq&
        f(x^k) - f^* - \gamma \|\nabla f(x^k)\|^2 + \frac{\gamma}{2} \| \nabla f(x^k) \|^2 
        \\
        &+ \frac{\gamma}{2} \|\EE_k[g^k] - \nabla f(x^k_g) \|^2
         + \frac{L \gamma^2 }{2} \EE_k[\| g^k\|^2]
         \\
         =&
        f(x^k) - f^* - \frac{\gamma}{2} \| \nabla f(x^k) \|^2 
        + \frac{\gamma}{2} \|\EE_k[g^k] - \nabla f(x^k) \|^2
         + \frac{L \gamma^2 }{2} \EE_k[\| g^k\|^2].
\end{align*}
Reapplying Cauchy Schwartz inequality \eqref{eq:inner_prod_and_sqr} one more time, we have
\begin{align*}
\EE_k[f(x^{k+1}) - f^*] 
\leq&
    f(x^k) - f^* - \frac{\gamma}{2} (1 - 2\gamma L) \| \nabla f(x^k) \|^2
    \\
         &+ \frac{\gamma}{2} \|\EE_k[g^k] - \nabla f(x^k) \|^2
         + L \gamma^2 \EE_k[\| g^k - \nabla f(x^k)\|^2]\,.
\end{align*}
\Cref{lem:expect_bound_grad_appendix} with $x^k_g$ replaced by $x^k$ gives
\begin{align*}
\EE_k[f(x^{k+1}) - f^*] 
\leq&
    f(x^k) - f^* - \frac{\gamma}{2} (1 - 2\gamma L) \| \nabla f(x^k) \|^2
    \\
         &+ \frac{\gamma}{2}  \cdot C_{2}\taumix^2 \batchbound^{-2}B^{-2} (\cmax^2 + \dmax^2 \| \nabla f(x^k)\|^2)
         \\
         &
         + L \gamma^2  \cdot \left(4 C_{1} \taumix B^{-1}\log_2 \batchbound + (4C_{1} + 2) \taumix^2 B^{-2}\right) (\cmax^2 + \dmax^2 \| \nabla f(x^k)\|^2)\,.
\end{align*}
With $M \geq \sqrt{C_2 \gamma^{-1} L^{-1}} $, we have
\begin{align*}
\EE_k[f(x^{k+1}) - f^*] 
\leq&
    f(x^k) - f^* - \frac{\gamma}{2} (1 - 2\gamma L) \| \nabla f(x^k) \|^2
    \\
         &+ \frac{L \gamma^2}{2}  \cdot \taumix^2 B^{-2} (\cmax^2 + \dmax^2 \| \nabla f(x^k)\|^2)
         \\
         &
         + L \gamma^2  \cdot \left(4 C_{1} \taumix B^{-1}\log_2 \batchbound + (4C_{1} + 2) \taumix^2 B^{-2}\right) (\cmax^2 + \dmax^2 \| \nabla f(x^k)\|^2)
         \\
\leq&
f(x^k) - f^* 
\\
         &
- \frac{\gamma}{2} \left[1 - 2\gamma L \left( 1 +  4\left[C_{1} \taumix B^{-1}\log_2 \batchbound + (C_{1} + 1) \taumix^2 B^{-2}\right] \dmax^2\right) \right] \| \nabla f(x^k) \|^2
    \\
         &
         + 4L \gamma^2  \cdot \left(C_{1} \taumix B^{-1}\log_2 \batchbound + (C_{1} + 1) \taumix^2 B^{-2}\right) \cmax^2\,.
\end{align*}
Since $\gamma \leq \left[4L \left( 1 +  4\left[C_{1} \taumix b^{-1} + (C_{1} + 1) \taumix^2 b^{-2}\right] \dmax^2\right)\right]^{-1}$, $B = \lceil b \log_2 \batchbound \rceil$ and $M \geq 2$, one can obtain
\begin{align*}
    \gamma 
    &\leq 
    \left[4L \left( 1 +  4\left[C_{1} \taumix b^{-1} + (C_{1} + 1) \taumix^2 b^{-2}\right] \dmax^2\right)\right]^{-1}
    \\
    &\leq
    \left[4L \left( 1 +  4\left[C_{1} \taumix B^{-1} \log_2 \batchbound + (C_{1} + 1) \taumix^2 B^{-2}\right] \dmax^2\right)\right]^{-1},
\end{align*}
and then,
\begin{align}
\label{eq:non_conv_temp40}
\EE_k[f(x^{k+1}) - f^*] 
\leq&
f(x^k) - f^* - \frac{\gamma}{4} \| \nabla f(x^k) \|^2
        \notag\\
         &+ 4L \gamma^2  \cdot \left(C_{1} \taumix B^{-1}\log_2 \batchbound + (C_{1} + 1) \taumix^2 B^{-2}\right) \cmax^2\,.
\end{align}
By doing a small rearrangements, summing over all $k$ from $0$ to $\nbiter-1$, averaging over $\nbiter$ iterations, taking the full expectation of both sides, we  get
\[
\EE\left[ \frac{1}{\nbiter}\sum\limits_{k=0}^{\nbiter-1}\| \nabla f(x^k)\|^2 \right] \leq \frac{4(f(x^0) - f^*)}{\gamma \nbiter} + 16L \gamma  \cdot \left[ C_{1} \sigma^2 \taumix B^{-1}\log_2 \batchbound + (C_{2} + 1) \cmax^2 \taumix^2 B^{-2} \right]\,.
\]
Substituting $B = \lceil b \log_2 \batchbound \rceil$ and using $M \geq 2$ finish the proof.
\end{proof}

\subsection{Result for Polyak-Loiasyewitch condition}
\label{sec:polyak_loj}
\begin{assumption}
\label{as:pl}
The function $f$ satisfies PL condition on $\R^d$ with $\mu > 0$, i.e. the following inequality holds for all $x\in \R^d$:
\begin{equation*}
    \norm{\nabla f(x) } \geq 2\mu (f(x) - f^*),
\end{equation*}
where  $f^*$ is a global (potentially not unique) minimum of $f$.
\end{assumption}

\begin{corollary}
Under the conditions of \Cref{th:non_convex_random_batch} and \Cref{as:pl}, if we choose $b = \taumix$ and $\gamma$ given by
\begin{align}
\label{eq:PL_gamma}
        \gamma \simeq \min\left\{ \frac{1}{(1+\delta^2)L}; \frac{1}{\mu \nbiter} \ln \left( \max\left\{ 2; \frac{\mu^2 \nbiter (f(x^0) - f^*)}{L\cmax^2} \right\}\right) \right\}\,,
\end{align}
then to achieve $\varepsilon$-solution (in terms of $\EE[f(x) - f^*] \lesssim \varepsilon$) we need 
\begin{equation*}
        \mathcal{\tilde O} \left( \taumix \cdot \left[ (1 + \dmax^2) \frac{L}{\mu} \log \frac{1}{\varepsilon} + \frac{L\cmax^2}{\mu^2\varepsilon} \right] \right) \quad \text{oracle calls}.
\end{equation*}  
\end{corollary}
\begin{proof}
We start from \eqref{eq:non_conv_temp40} and apply \Cref{as:pl}.
\begin{align*}
\EE_k[f(x^{k+1}) - f^*] 
\leq&
(1 - \mu \gamma / 2)(f(x^k) - f^*)
        \notag\\
         &+ 4L \gamma^2  \cdot \left(C_{1} \taumix B^{-1}\log_2 \batchbound + (C_{1} + 1) \taumix^2 B^{-2}\right) \cmax^2\,.
\end{align*}
Next, we perform the recursion
\begin{align*}
\EE[f(x^{N}) - f^*] 
\leq&
(1 - \mu \gamma / 2)^N (f(x^k) - f^*)
        \notag\\
         &+ 8L \mu^{-1} \gamma  \cdot \left(C_{1} \taumix B^{-1}\log_2 \batchbound + (C_{1} + 1) \taumix^2 B^{-2}\right) \cmax^2
\\
\leq&
\exp(- \mu \gamma N / 2) (f(x^0) - f^*)
        \notag\\
         &+ 8L \mu^{-1} \gamma  \cdot \left(C_{1} \taumix B^{-1}\log_2 \batchbound + (C_{1} + 1) \taumix^2 B^{-2}\right) \cmax^2\,.
\end{align*}
It remains to substitute $\gamma$ from \eqref{eq:PL_gamma}, $B = \lceil b \log_2 \batchbound \rceil$ and $b = \tau$.
\end{proof}

\subsection{Proof of \Cref{th:vi_str_mon}} \label{sec:proof:th:strongly_mon_upper_bound}
We preface the proof by technical Lemma.
\begin{lemma}
    \label{lem:tech1}
    Let $r$ be $\mu_r$-strongly convex and $x^+ = \text{prox}_{\gamma r}(x)$. Then for all $u \in \mathcal{X}$ the following iniqulity hold:
    $$
    \langle x^+ - x, u - x^+ \rangle \geq \gamma \left( r(x^+) - r(u) + \frac{\mu_r}{2} \| x^+ - u\|^2\right).
    $$
\end{lemma}
\begin{proof}
The optimality condition for $x^+ = \text{prox}_{\gamma r}(x) = \arg\min_{y \in \mathcal{X}} (\gamma r (y) + \tfrac{1}{2} \| x^+ - y\|^2)$ gives that $(x - x^+) \in \partial r(x^+)$. Therefore, using strong convexity (see \Cref{as:strong_monotone_op}) for $r'(x^+) = (x - x^+) \in \partial r(x^+)$, we get
$$
\gamma (r (u) - r (x^+)) \geq \langle x - x^+, u - x^+ \rangle + \frac{\gamma \mu_r}{2} \| x^+ - u\|^2.
$$
After small rearrangements we have what we need to prove.
\end{proof}
\begin{theorem}[\Cref{th:vi_str_mon}]
Assume \Cref{as:lipshitz_op},~\Cref{as:strong_monotone_op} with $\mu_F + \mu_r > 0$,~\Cref{as:Markov_noise_UGE},~\Cref{as:bounded_markov_noise_UGE_op}. Let problem \eqref{eq:VI}
be solved by \Cref{alg:EG}. Then for any $b \in \nsets$, and $\gamma$, $M$ satisfying
\begin{eqnarray*}
&\text{\small{$\gamma \leq \min\{ (3\mu_F + 3\mu_r)^{-1} ; (3L)^{-1}; (6\mu_F + \mu_r) \cdot  [120(C_{1} \taumix b^{-1} + (C_{1} + 1) \taumix^2 b^{-2})\Dmax^2 ]^{-1} ; \sqrt{(18C_{1})^{-1} \Delta^{-2} \tau^{-1} b} \}$}}\,,
\\
&
M = \max\{2;\sqrt{C_2 \gamma^{-1} (\mu_F + \mu_r)^{-1}}\}, \quad B = \lceil b \log_2 \batchbound \rceil\,,
\end{eqnarray*}
it holds that
\begin{equation*}
    \EEE{\|x^{N} - x^* \|^2}
        \lesssim
        \exp\left(- \frac{(\mu_F + \mu_r) \gamma N}{16}\right) \| x^0 - x^*\|^2
        + \frac{\gamma}{\mu}(\sigma^2 \tau b^{-1} + \sigma^2 \tau^2 b^{-2})
\,.
\end{equation*}
\end{theorem}
\begin{proof}
We start from \Cref{lem:tech1} for $x^{k+1} = \text{prox}_{\gamma r} \left(x^k - \gamma g^{k}\right)$ with $x^+ = x^{k+1}$, $x = x^k - \gamma g^{k}$, $u = x^*$ and for $x^{k+1/2} = \text{prox}_{\gamma r} \left(x^k - \gamma B^{-1} \sum_{i=1}^{B} F(x^{k}, z^{k}_i)\right)$ with $x^+ = x^{k+1/2}$, $x = x^k - \gamma B^{-1} \sum_{i=1}^{B} F(x^{k}, z^{k}_i)$, $u = x^{k+1}$:
    \begin{align*}
         \langle x^{k+1} - x^k + \gamma g^{k}, x^* - x^{k+1} \rangle \geq \gamma \left( r(x^{k+1}) - r(x^*) + \frac{\mu_r}{2} \| x^{k+1} - x^*\|^2\right),
    \end{align*}
and
    \begin{align*}
         \langle x^{k+1/2} - x^k + \gamma B^{-1} \sum\limits_{i=1}^{B} &F(x^{k}, z^{k}_i), x^{k+1} - x^{k+1/2} \rangle 
         \\
         &\geq \gamma \left( r(x^{k+1/2}) - r(x^{k+1}) + \frac{\mu_r}{2} \| x^{k+1} - x^{k+1/2}\|^2\right).
    \end{align*}
Summing up these two inequalities, we get
    \begin{align*}
         \langle x^{k+1} - x^k &+ \gamma g^{k}, x^* - x^{k+1} \rangle + \langle x^{k+1/2} - x^k + \gamma F(x^k, z^k), x^{k+1} - x^{k+1/2} \rangle
         \\
         &\geq
         \gamma \left( r(x^{k+1/2}) - r(x^*) + \frac{\mu_r}{2} \| x^{k+1} - x^*\|^2 + \frac{\mu_r}{2} \| x^{k+1} - x^{k+1/2}\|^2\right).
    \end{align*}
After some rearrangements, we have
    \begin{align*}
         &\langle x^{k+1} - x^k, x^* - x^{k+1} \rangle + \langle x^{k+1/2} - x^k, x^{k+1} -  x^{k+1/2}\rangle
         \\
         &+ \gamma \langle g^k - B^{-1} \sum\limits_{i=1}^{B} F(x^{k}, z^{k}_i), x^{k+1/2} - x^{k+1} \rangle  + \gamma \langle g^k, x^* - x^{k+1/2} \rangle
         \\
         &\hspace{2cm}\geq
         \gamma \left( r(x^{k+1/2}) - r(x^*) + \frac{\mu_r}{2} \| x^{k+1} - x^*\|^2 + \frac{\mu_r}{2} \| x^{k+1} - x^{k+1/2}\|^2\right).
    \end{align*}
With $2\langle a, b\rangle = \| a + b\|^2 - \| a\|^2 - \| b\|^2$, we deduce
    \begin{align*}
        \|x^k & - x^* \|^2 - \| x^{k+1} - x^*\|^2 - \| x^{k+1} - x^k \|^2
        \\
        &+ \| x^{k+1} - x^k\|^2 - \| x^{k+1/2} - x^k \|^2 - \| x^{k+1} -  x^{k+1/2} \|^2
        \\
        &+ 2\gamma \langle g^k - B^{-1} \sum\limits_{i=1}^{B} F(x^{k}, z^{k}_i), x^{k+1/2} - x^{k+1} \rangle  + 2\gamma \langle g^k, x^* - x^{k+1/2} \rangle
         \\
         &\hspace{2cm}\geq
         2\gamma \left( r(x^{k+1/2}) - r(x^*) + \frac{\mu_r}{2} \| x^{k+1} - x^*\|^2 + \frac{\mu_r}{2} \| x^{k+1} - x^{k+1/2}\|^2\right).
    \end{align*}
After rewriting in a slightly different way,
    \begin{align*}
        \|x^{k+1} - x^* \|^2 +\|x^{k+1/2} - x^{k+1}\|^2
        \leq&
        \| x^k - x^*\|^2
        - 2\gamma \langle g^{k}, x^{k+1/2} - x^* \rangle 
        \\
        &
        - 2\gamma \langle B^{-1} \sum\limits_{i=1}^{B} F(x^{k}, z^{k}_i) - g^{k}, x^{k+1/2} - x^{k+1} \rangle
        \\
        &
        - \| x^{k+1/2} - x^k\|^2 - 2\gamma (r(x^{k+1/2}) - r(x^*)) 
        \\
        &
        - \mu_r \gamma \| x^{k+1} - x^*\|^2 - \mu_r\gamma \| x^{k+1} - x^{k+1/2}\|^2
        \\
        \leq&
        \| x^k - x^*\|^2
        - 2\gamma \langle g^{k}, x^{k+1/2} - x^* \rangle
        \\
        &
        + \gamma^2 \left\| B^{-1} \sum\limits_{i=1}^{B} F(x^{k}, z^{k}_i) - g^{k}\right\|^2 + \| x^{k+1/2} - x^{k+1} \|^2
        \\
        &
        - \| x^{k+1/2} - x^k\|^2 - 2\gamma (r(x^{k+1/2}) - r(x^*)) 
        \\
        &
        - \mu_r \gamma \| x^{k+1} - x^*\|^2 - \mu_r\gamma \| x^{k+1} - x^{k+1/2}\|^2.
    \end{align*}
In the last step, we used Cauchy-Schwartz inequality \eqref{eq:inner_prod}. Subtracting $\| x^{k+1} - x^{k+1/2}\|^2$ from both parts, we get
    \begin{align*}
        \|x^{k+1} - x^* \|^2
        \leq&
        \| x^k - x^*\|^2
        - 2\gamma \langle g^{k}, x^{k+1/2} - x^* \rangle + \gamma^2 \| B^{-1} \sum\limits_{i=1}^{B} F(x^{k}, z^{k}_i) - g^{k}\|^2
        \\
        &
        - \| x^{k+1/2} - x^k\|^2 - 2\gamma (r(x^{k+1/2}) - r(x^*)) 
        \\
        &
        - \mu_r \gamma \| x^{k+1} - x^*\|^2 - \mu_r\gamma \| x^{k+1} - x^{k+1/2}\|^2
        \\
        =&
        \| x^k - x^*\|^2
        - 2\gamma \langle F(x^{k+1/2}), x^{k+1/2} - x^* \rangle
        \\
        &
        - 2\gamma \langle \E_{k+1/2}[g^{k}] - F(x^{k+1/2}), x^{k+1/2} - x^* \rangle
        \\
        &
        - 2\gamma \langle g^{k} - \E_{k+1/2}[g^{k}], x^{k+1/2} - x^* \rangle
        \\
        &
        + \gamma^2 \| F(x^k) - F(x^{k+1/2}) + F(x^k) - B^{-1} \sum\limits_{i=1}^{B} F(x^{k}, z^{k}_i) + F(x^{k+1/2}) - g^{k}\|^2
        \\
        &
        - \| x^{k+1/2} - x^k\|^2 - 2\gamma (r(x^{k+1/2}) - r(x^*)) 
        \\
        &
        - \mu_r \gamma \| x^{k+1} - x^*\|^2 - \mu_r\gamma \| x^{k+1} - x^{k+1/2}\|^2.
    \end{align*}
Again with Cauchy-Schwartz inequality \eqref{eq:sum_sqr}, we conduct
    \begin{align}
        \label{eq:vi_temp1}
        \|x^{k+1} - x^* \|^2
        \leq&
        \| x^k - x^*\|^2
        - 2\gamma \langle F(x^{k+1/2}), x^{k+1/2} - x^* \rangle
        \notag\\
        &
        - 2\gamma \langle \E_{k+1/2}[g^{k}] - F(x^{k+1/2}), x^{k+1/2} - x^* \rangle
        \notag\\
        &
        - 2\gamma \langle g^{k} - \E_{k+1/2}[g^{k}], x^{k+1/2} - x^* \rangle
        + 3\gamma^2 \| B^{-1} \sum\limits_{i=1}^{B} F(x^{k}, z^{k}_i) - F(x^k) \|^2
        \notag\\
        &
        + 3\gamma^2 \| F(x^{k+1/2}) - g^{k}\|^2
        + 3\gamma^2 \| F(x^{k+1/2}) - F(x^k)\|^2
        - \| x^{k+1/2} - x^k\|^2
        \notag\\
        &
        - 2\gamma (r(x^{k+1/2}) - r(x^*)) - \mu_r \gamma \| x^{k+1} - x^*\|^2 - \mu_r\gamma \| x^{k+1} - x^{k+1/2}\|^2.
    \end{align}
\Cref{as:lipshitz_op} and the property of the solution \eqref{eq:VI}: $- (r(x^{k+1/2}) - r(x^*)) \leq \langle F(x^*), x^{k+1/2} - x^*\rangle$, together give
    \begin{align*}
        \|x^{k+1} - x^* \|^2
        \leq&
        \| x^k - x^*\|^2
        - 2\gamma \langle F(x^{k+1/2}) - F(x^*), x^{k+1/2} - x^* \rangle
        \\
        &
        - 2\gamma \langle \E_{k+1/2}[g^{k}] - F(x^{k+1/2}), x^{k+1/2} - x^* \rangle
        \\
        &
        - 2\gamma \langle g^{k} - \E_{k+1/2}[g^{k}], x^{k+1/2} - x^* \rangle
        \\
        &
        + 3\gamma^2 \| B^{-1} \sum\limits_{i=1}^{B} F(x^{k}, z^{k}_i) - F(x^k) \|^2
        + 3\gamma^2 \| F(x^{k+1/2}) - g^{k}\|^2
         \\
        &
        + 3\gamma^2 L^2 \| x^{k+1/2} - x^k\|^2
        - \| x^{k+1/2} - x^k\|^2
        \\
        &
        - \mu_r \gamma \| x^{k+1} - x^*\|^2 - \mu_r\gamma \| x^{k+1} - x^{k+1/2}\|^2.
    \end{align*}
Next, one can apply \Cref{as:strong_monotone_op} and have
    \begin{align*}
        \|x^{k+1} - x^* \|^2
        \leq&
        \| x^k - x^*\|^2
        - 2\mu_F \gamma \| x^{k+1/2} - x^*\|^2
        \\
        &
        - 2\gamma \langle \E_{k+1/2}[g^{k}] - F(x^{k+1/2}), x^{k+1/2} - x^* \rangle
        \\
        &
        - 2\gamma \langle g^{k} - \E_{k+1/2}[g^{k}], x^{k+1/2} - x^* \rangle
        + 3\gamma^2 \| B^{-1} \sum\limits_{i=1}^{B} F(x^{k}, z^{k}_i) - F(x^k) \|^2
        \\
        &
        + 3\gamma^2 \| F(x^{k+1/2}) - g^{k}\|^2
        + 3\gamma^2 L^2 \| x^{k+1/2} - x^k\|^2
        - \| x^{k+1/2} - x^k\|^2
        \\
        &
        - \mu_r \gamma \| x^{k+1} - x^*\|^2 - \mu_r\gamma \| x^{k+1} - x^{k+1/2}\|^2.
    \end{align*}
Using Cauchy-Schwartz inequality \eqref{eq:inner_prod} one more time, we get
    \begin{align*}
        \|x^{k+1} - x^* \|^2
        \leq&
        \| x^k - x^*\|^2
        - 2\mu_F \gamma \| x^{k+1/2} - x^*\|^2 
        \\
        &
        - \mu_r \gamma \| x^{k+1} - x^*\|^2 - \mu_r\gamma \| x^{k+1} - x^{k+1/2}\|^2
        \\
        &
        +\frac{4\gamma}{\mu_F + \mu_r} \| \E_{k+1/2}[g^{k}] - F(x^{k+1/2})\|^2 + \frac{(\mu_F + \mu_r)\gamma}{4} \| x^{k+1/2} - x^* \|^2
        \\
        &
        - 2\gamma \langle g^{k} - \E_{k+1/2}[g^{k}], x^{k+1/2} - x^* \rangle
        + 3\gamma^2 \| B^{-1} \sum\limits_{i=1}^{B} F(x^{k}, z^{k}_i) - F(x^k)\|^2
        \\
        &
        + 3\gamma^2 \| F(x^{k+1/2}) - g^{k}\|^2
        + 3\gamma^2 L^2 \| x^{k+1/2} - x^k\|^2
        - \| x^{k+1/2} - x^k\|^2
        \\
        \leq&
        \| x^k - x^*\|^2
        - \frac{7\mu_F \gamma}{4} \| x^{k+1/2} - x^*\|^2 
        \\
        &
        - \mu_r \gamma \| x^{k+1} - x^*\|^2 - \mu_r\gamma \| x^{k+1} - x^{k+1/2}\|^2
        \\
        &
        + \frac{\mu_r\gamma}{4} \| x^{k+1/2} - x^* \|^2 +\frac{4\gamma}{\mu_F + \mu_r} \| \E_{k+1/2}[g^{k}] - F(x^{k+1/2})\|^2
        \\
        &
        - 2\gamma \langle g^{k} - \E_{k+1/2}[g^{k}], x^{k+1/2} - x^* \rangle
        + 3\gamma^2 \| B^{-1} \sum\limits_{i=1}^{B} F(x^{k}, z^{k}_i) - F(x^k)\|^2
        \\
        &
        + 3\gamma^2 \| F(x^{k+1/2}) - g^{k}\|^2
        + 3\gamma^2 L^2 \| x^{k+1/2} - x^k\|^2
        - \| x^{k+1/2} - x^k\|^2.
    \end{align*}
With Cauchy-Schwartz inequality in the form: $-  \mu_r \gamma \| x^{k+1} - x^*\|^2 \leq - \tfrac{ \mu_r \gamma}{2} \| x^{k+1/2} - x^*\|^2 +  \mu_r \gamma \| x^{k+1} - x^{k+1/2}\|^2$, one can deduce
    \begin{align*}
        \|x^{k+1} - x^* \|^2
        \leq&
        \| x^k - x^*\|^2
        - \frac{(7\mu_F + \mu_r) \gamma}{4} \| x^{k+1/2} - x^*\|^2
        \\
        &
        +\frac{4\gamma}{\mu_F + \mu_r} \| \E_{k+1/2}[g^{k}] - F(x^{k+1/2})\|^2
        \\
        &
        - 2\gamma \langle g^{k} - \E_{k+1/2}[g^{k}], x^{k+1/2} - x^* \rangle
        + 3\gamma^2 \| B^{-1} \sum\limits_{i=1}^{B} F(x^{k}, z^{k}_i) - F(x^k)\|^2
        \\
        &
        + 3\gamma^2 \| F(x^{k+1/2}) - g^{k}\|^2
        + 3\gamma^2 L^2 \| x^{k+1/2} - x^k\|^2
        - \| x^{k+1/2} - x^k\|^2.
    \end{align*}
Taking the expectation and using \Cref{lem:tech_markov_app}, \Cref{lem:expect_bound_grad_appendix} (with $\Delta^2 \|x - x^* \|^2$ instead of $\delta^2 \| \nabla f(x)\|^2$), we have
    \begin{align*}
        \EEE{\|x^{k+1} - x^* \|^2}
        \leq&
        \EEE{\| x^k - x^*\|^2}
        - \frac{(7\mu_F + \mu_r) \gamma}{4} \EEE{\| x^{k+1/2} - x^*\|^2}
        \\
        &
        +\frac{4\gamma}{\mu_F + \mu_r} \EEE{\| \E_{k+1/2}[g^{k}] - F(x^{k+1/2})\|^2}
        \\
        &
        + 3\gamma^2 \EEE{\EE_k\left[\| B^{-1} \sum\limits_{i=1}^{B} F(x^{k}, z^{k}_i) - F(x^k)\|^2\right]}
        \\
        &
        + 3\gamma^2 \EEE{\EE_{k+1/2}\left[\| F(x^{k+1/2}) - g^{k}\|^2\right]}
        \\
        &
        + 3\gamma^2 L^2 \EEE{\| x^{k+1/2} - x^k\|^2}
        - \EEE{\| x^{k+1/2} - x^k\|^2}
        \\
        \leq&
        \EEE{\| x^k - x^*\|^2}
        - \frac{(7\mu_F + \mu_r) \gamma}{4} \EEE{\| x^{k+1/2} - x^*\|^2}
        \\
        &
        +\frac{4\gamma}{\mu_F + \mu_r} \cdot C_{2}\taumix^2\batchbound^{-2}B^{-2} \left(\cmax^2 + \Dmax^2 \EEE{\| x^{k+1/2} - x^*\|^2} \right)
        \\
        &
        + 3\gamma^2 \cdot C_{1}  \taumix B^{-1} \left(\cmax^2 + \Dmax^2 \EEE{\| x^{k} - x^*\|^2}\right)
        \\
        &
        + 3\gamma^2 \cdot \left(4 C_{1} \taumix B^{-1}\log_2 \batchbound + (4C_{1} + 2) \taumix^2 B^{-2}\right) (\cmax^2 + \Dmax^2 \| x^{k+1/2} - x^*\|^2)
        \\
        &
        + 3\gamma^2 L^2 \EEE{\| x^{k+1/2} - x^k\|^2}
        - \EEE{\| x^{k+1/2} - x^k\|^2}.
    \end{align*}
With $M \geq \sqrt{C_2 \gamma^{-1} (\mu_F + \mu_r)^{-1}}$, we have
   \begin{align*}
        \EEE{\|x^{k+1} - x^* \|^2}
        \leq&
        \EEE{\| x^k - x^*\|^2}
        - \frac{(7\mu_F + \mu_r) \gamma}{4} \EEE{\| x^{k+1/2} - x^*\|^2}
        \\
        &
        +4\gamma^2 \cdot \taumix^2 B^{-2} \left(\cmax^2 + \Dmax^2 \EEE{\| x^{k+1/2} - x^*\|^2} \right)
        \\
        &
        + 3\gamma^2 \cdot C_{1}  \taumix B^{-1} \left(\cmax^2 + \Dmax^2 \EEE{\| x^{k} - x^*\|^2}\right)
        \\
        &
        + 3\gamma^2 \cdot \left(4 C_{1} \taumix B^{-1}\log_2 \batchbound + (4C_{1} + 2) \taumix^2 B^{-2}\right) \left(\cmax^2 + \Dmax^2 \EEE{\| x^{k+1/2} - x^*\|^2}\right)
        \\
        &
        + 3\gamma^2 L^2 \EEE{\| x^{k+1/2} - x^k\|^2}
        - \EEE{\| x^{k+1/2} - x^k\|^2}
        \\
        \leq&
        \EEE{\| x^k - x^*\|^2}
        - \frac{(7\mu_F + \mu_r) \gamma}{4} \EEE{\| x^{k+1/2} - x^*\|^2}
        \\
        &
        + 3\gamma^2 \cdot C_{1}  \taumix B^{-1} \Dmax^2 \EEE{\| x^{k} - x^*\|^2}
        \\
        &
        + 12\gamma^2 \cdot \left(C_{1} \taumix B^{-1}\log_2 \batchbound + (C_{1} + 1) \taumix^2 B^{-2}\right)\Dmax^2 \EEE{\| x^{k+1/2} - x^*\|^2}
        \\
        &
        + 15\gamma^2 \cdot \left(C_{1} \taumix B^{-1}\log_2 \batchbound + (C_{1} + 1) \taumix^2 B^{-2}\right) \cmax^2
        \\
        &
        + 3\gamma^2 L^2 \EEE{\| x^{k+1/2} - x^k\|^2}
        - \EEE{\| x^{k+1/2} - x^k\|^2}.
    \end{align*}
Cauchy-Schwartz inequality \eqref{eq:inner_prod_and_sqr} gives
   \begin{align*}
        \EEE{\|x^{k+1} - x^* \|^2}
        \leq&
        \EEE{\| x^k - x^*\|^2}
        - \frac{(7\mu_F + \mu_r) \gamma}{4} \EEE{\| x^{k+1/2} - x^*\|^2}
        \\
        &
        + 15\gamma^2 \cdot \left(C_{1} \taumix B^{-1}\log_2 \batchbound + (C_{1} + 1) \taumix^2 B^{-2}\right)\Dmax^2 \EEE{\| x^{k+1/2} - x^*\|^2}
        \\
        &
        + 15\gamma^2 \cdot \left(C_{1} \taumix B^{-1}\log_2 \batchbound + (C_{1} + 1) \taumix^2 B^{-2}\right) \cmax^2
        \\
        &
        + 6\gamma^2 \cdot C_{1}  \taumix B^{-1} \Dmax^2 \EEE{\| x^{k+1/2} - x^k\|^2}
        \\
        &
        + 3\gamma^2 L^2 \EEE{\| x^{k+1/2} - x^k\|^2}
        - \EEE{\| x^{k+1/2} - x^k\|^2}.
    \end{align*}
Since $\gamma \leq (7\mu_F + \mu_r) \cdot  \left[120\left(C_{1} \taumix b^{-1} + (C_{1} + 1) \taumix^2 b^{-2}\right)\Dmax^2 \right]^{-1}$, $B = \lceil b \log_2 \batchbound \rceil$ and $M \geq 2$, one can obtain
\begin{align*}
    \gamma 
    &\leq 
    (7\mu_F + \mu_r) \cdot  \left[120\left(C_{1} \taumix b^{-1} + (C_{1} + 1) \taumix^2 b^{-2}\right)\Dmax^2 \right]^{-1}
    \\
    &\leq
    (7\mu_F + \mu_r) \cdot  \left[120\left(C_{1} \taumix B^{-1}\log_2 \batchbound + (C_{1} + 1) \taumix^2 B^{-2}\right)\Dmax^2 \right]^{-1},
\end{align*}
and then,
   \begin{align*}
        \EEE{\|x^{k+1} - x^* \|^2}
        \leq&
        \EEE{\| x^k - x^*\|^2}
        - \frac{(7\mu_F + \mu_r) \gamma}{8} \EEE{\| x^{k+1/2} - x^*\|^2}
        \\
        &
        + 15\gamma^2 \cdot \left(C_{1} \taumix B^{-1}\log_2 \batchbound + (C_{1} + 1) \taumix^2 B^{-2}\right) \cmax^2
        \\
        &
        + 6\gamma^2 \cdot C_{1}  \taumix B^{-1} \Dmax^2 \EEE{\| x^{k+1/2} - x^k\|^2}
        \\
        &
        + 3\gamma^2 L^2 \EEE{\| x^{k+1/2} - x^k\|^2}
        - \EEE{\| x^{k+1/2} - x^k\|^2}.
    \end{align*}
With Cauchy-Schwartz inequality in the form: $- \| x^{k+1/2} - x^*\|^2 \leq - \tfrac{ 1}{2} \| x^{k} - x^*\|^2 +  \| x^{k} - x^{k+1/2}\|^2$, we have
   \begin{align*}
        \EEE{\|x^{k+1} - x^* \|^2}
        \leq&
        \left( 1 - \frac{(7\mu_F + \mu_r) \gamma}{16} \right)\EEE{\| x^k - x^*\|^2}
        \\
        &
        - \left(1 - (\mu_F + \mu_r) \gamma - 3\gamma^2 L^2 - 6\gamma^2 \cdot C_{1}  \taumix B^{-1} \Dmax^2\right) \EEE{\| x^{k+1/2} - x^k\|^2}
        \\
        &
        + 15\gamma^2 \cdot \left(C_{1} \taumix B^{-1}\log_2 \batchbound + (C_{1} + 1) \taumix^2 B^{-2}\right) \cmax^2.
    \end{align*}
Since $\gamma \leq \min\left\{ (3\mu_F + 3\mu_r)^{-1}; (3L)^{-1};  \sqrt{(18C_{1})^{-1}  \taumix^{-1} b \Dmax^{-2}} \right\}$, we get   
   \begin{align*}
        \EEE{\|x^{k+1} - x^* \|^2}
        \leq&
        \left( 1 - \frac{(7\mu_F + \mu_r) \gamma}{16} \right)\EEE{\| x^k - x^*\|^2}
        \\
        &
        + 15\gamma^2 \cdot \left(C_{1} \taumix B^{-1}\log_2 \batchbound + (C_{1} + 1) \taumix^2 B^{-2}\right) \cmax^2.
    \end{align*}
Next, we perform the recursion
   \begin{align*}
        \EEE{\|x^{N} - x^* \|^2}
        \leq&
        \left( 1 - \frac{(7\mu_F + \mu_r) \gamma}{16} \right)^N \| x^0 - x^*\|^2
        \\
        &
        + \frac{240\gamma}{(\mu_F + \mu_r)} \cdot \left(C_{1} \taumix B^{-1}\log_2 \batchbound + (C_{1} + 1) \taumix^2 B^{-2}\right) \cmax^2
        \\
        \leq&
        \exp\left( - \frac{(\mu_F + \mu_r) \gamma N}{16} \right) \| x^0 - x^*\|^2
        \\
        &
        + \frac{240\gamma}{(\mu_F + \mu_r)} \cdot \left(C_{1} \taumix B^{-1}\log_2 \batchbound + (C_{1} + 1) \taumix^2 B^{-2}\right) \cmax^2.
    \end{align*}
Substituting $B = \lceil b \log_2 \batchbound \rceil$ and using $M \geq 2$ finish the proof.
\end{proof}    

\subsection{Proof of \Cref{th:vi_mon}} \label{sec:proof:th:mon_upper_bound}
\begin{theorem}[\Cref{th:vi_mon}]
Assume \Cref{as:lipshitz_op},~\Cref{as:strong_monotone_op} with $\mu_F + \mu_r = 0$,~\Cref{as:boundedset},~\Cref{as:Markov_noise_UGE},~\Cref{as:bounded_markov_noise_UGE_op}. Let problem \eqref{eq:VI}
be solved by \Cref{alg:EG}. Then for any $B \in \nsets$, and $\gamma$, $M$ satisfying $\gamma \lesssim L^{-1}\,, \,\, M = \sqrt{N}$, it holds that
\begin{equation*}
    \EEE{\text{Gap}(\bar x^\nbiter)}
        \lesssim
        \frac{D^2}{\gamma N} + \gamma (\taumix B^{-1}\log_2 N + \taumix^2 B^{-2} )(\cmax^2 + \Dmax^2 D^2)
\,,
\end{equation*}
where $\bar x^\nbiter = \tfrac{1}{\nbiter} \sum\limits_{k=0}^{\nbiter-1} x^{k+1/2}$.
\end{theorem}
\begin{proof}
We start from \eqref{eq:vi_temp1} with arbitrary $x \in \mathcal{X}$ instead of $x^*$ and $\mu_r = 0$:
    \begin{align*}
        \|x^{k+1} - x \|^2
        \leq&
        \| x^k - x\|^2
        - 2\gamma \langle F(x^{k+1/2}), x^{k+1/2} - x \rangle
        \notag\\
        &
        - 2\gamma \langle \E_{k+1/2}[g^{k}] - F(x^{k+1/2}), x^{k+1/2} - x \rangle
        \notag\\
        &
        - 2\gamma \langle g^{k} - \E_{k+1/2}[g^{k}], x^{k+1/2} - x \rangle
        + 3\gamma^2 \| B^{-1} \sum\limits_{i=1}^{B} F(x^{k}, z^{k}_i) - F(x^k)\|^2
        \notag\\
        &
        + 3\gamma^2 \| F(x^{k+1/2}) - g^{k}\|^2
        + 3\gamma^2 \| F(x^{k+1/2}) - F(x^k)\|^2
        - \| x^{k+1/2} - x^k\|^2
        \notag\\
        &
        - 2\gamma (r(x^{k+1/2}) - r(x)).
    \end{align*}
After small rearrangements, we get
    \begin{align*}
    2\gamma \big(\langle F(x^{k+1/2}) &, x^{k+1/2} - x \rangle +  r(x^{k+1/2}) - r(x) \big)
        \\\leq&
        \| x^k - x\|^2 - \|x^{k+1} - x \|^2
        - 2\gamma \langle \E_{k+1/2}[g^{k}] - F(x^{k+1/2}), x^{k+1/2} - x \rangle
        \notag\\
        &
        - 2\gamma \langle g^{k} - \E_{k+1/2}[g^{k}], x^{k+1/2} - x \rangle
        + 3\gamma^2 \| B^{-1} \sum\limits_{i=1}^{B} F(x^{k}, z^{k}_i) - F(x^k)\|^2
        \notag\\
        &
        + 3\gamma^2 \| F(x^{k+1/2}) - g^{k}\|^2
        + 3\gamma^2 \| F(x^{k+1/2}) - F(x^k)\|^2
        - \| x^{k+1/2} - x^k\|^2.
    \end{align*}
Applying Cauchy-Schwartz inequality and making more rearrangements, we get
    \begin{align*}
    2\gamma \big(\langle F(x^{k+1/2})&, x^{k+1/2} - x \rangle +  r(x^{k+1/2}) - r(x) \big)
        \\\leq&
        \| x^k - x\|^2 - \|x^{k+1} - x \|^2
        + \gamma^2 N \| \E_{k+1/2}[g^{k}] - F(x^{k+1/2}) \|^2  + \frac{1}{N}\| x^{k+1/2} - x \|^2
        \notag\\
        &
        - 2\gamma \langle g^{k} - \E_{k+1/2}[g^{k}], x^{k+1/2} - x^0 \rangle
        - 2\gamma \langle g^{k} - \E_{k+1/2}[g^{k}], x^{0} - x \rangle
        \notag\\
        &
        + 3\gamma^2 \| B^{-1} \sum\limits_{i=1}^{B} F(x^{k}, z^{k}_i) - F(x^k)\|^2 + 3\gamma^2 \| F(x^{k+1/2}) - g^{k}\|^2
        \notag\\
        &
        + 3\gamma^2 \| F(x^{k+1/2}) - F(x^k)\|^2
        - \| x^{k+1/2} - x^k\|^2.
    \end{align*}
Summing over all $k$ from $0$ to $\nbiter-1$ and dividing by $\nbiter$, we have 
    \begin{align*}
    2\gamma \cdot& \frac{1}{\nbiter} \sum\limits_{k=0}^{\nbiter-1}\big(\langle F(x^{k+1/2}),  x^{k+1/2} - x \rangle +  r(x^{k+1/2}) - r(x) \big)
        \\\leq&
        \frac{\| x^0 - x\|^2 - \|x^{K} - x \|^2}{\nbiter}
        + \gamma^2 \sum\limits_{k=0}^{\nbiter-1}\| \E_{k+1/2}[g^{k}] - F(x^{k+1/2}) \|^2  + \frac{1}{\nbiter^2} \sum\limits_{k=0}^{\nbiter-1}\| x^{k+1/2} - x \|^2
        \notag\\
        &
        - 2\gamma \cdot \frac{1}{\nbiter} \sum\limits_{k=0}^{\nbiter-1} \langle g^{k} - \E_{k+1/2}[g^{k}], x^{k+1/2} - x^0 \rangle
        - 2\gamma \langle \nbiter^{-1} \sum\limits_{k=0}^{\nbiter-1} \left[g^{k} - \E_{k+1/2}[g^{k}] \right], x^{0} - x \rangle
        \notag\\
        &
        + 3\gamma^2 \cdot \frac{1}{\nbiter} \sum\limits_{k=0}^{\nbiter-1} \| B^{-1} \sum\limits_{i=1}^{B} F(x^{k}, z^{k}_i) - F(x^k)\|^2 + 3\gamma^2\cdot \frac{1}{\nbiter} \sum\limits_{k=0}^{\nbiter-1}  \| F(x^{k+1/2}) - g^{k}\|^2
        \notag\\
        &
        + 3\gamma^2 \cdot \frac{1}{\nbiter} \sum\limits_{k=0}^{\nbiter-1} \| F(x^{k+1/2}) - F(x^k)\|^2
        - \frac{1}{\nbiter} \sum\limits_{k=0}^{\nbiter-1} \| x^{k+1/2} - x^k\|^2.
    \end{align*}
Using monotonicity and Jensen's inequality \eqref{eq:jensen} for convex function $r$, we get (with notation $\bar x^\nbiter = \frac{1}{\nbiter} \sum\limits_{k=0}^{\nbiter-1} x^{k+1/2}$)
    \begin{align*}
    2\gamma \big(\langle& F(x),  \bar x^\nbiter - x \rangle +  r(\bar x^\nbiter) - r(x) \big)
        \\\leq&
        \frac{\| x^0 - x\|^2 - \|x^{\nbiter} - x \|^2}{\nbiter}
        + \gamma^2 \sum\limits_{k=0}^{\nbiter-1}\| \E_{k+1/2}[g^{k}] - F(x^{k+1/2}) \|^2  + \frac{1}{\nbiter^2} \sum\limits_{k=0}^{\nbiter-1}\| x^{k+1/2} - x \|^2
        \notag\\
        &
        - 2\gamma \cdot \frac{1}{\nbiter} \sum\limits_{k=0}^{\nbiter-1} \langle g^{k} - \E_{k+1/2}[g^{k}], x^{k+1/2} - x^0 \rangle
        - 2\gamma \langle \nbiter^{-1} \sum\limits_{k=0}^{\nbiter-1} \left[g^{k} - \E_{k+1/2}[g^{k}] \right], x^{0} - x \rangle
        \notag\\
        &
        + 3\gamma^2 \cdot \frac{1}{\nbiter} \sum\limits_{k=0}^{\nbiter-1} \| B^{-1} \sum\limits_{i=1}^{B} F(x^{k}, z^{k}_i) - F(x^k)\|^2 + 3\gamma^2\cdot \frac{1}{\nbiter} \sum\limits_{k=0}^{\nbiter-1}  \| F(x^{k+1/2}) - g^{k}\|^2
        \notag\\
        &
        + 3\gamma^2 \cdot \frac{1}{\nbiter} \sum\limits_{k=0}^{\nbiter-1} \| F(x^{k+1/2}) - F(x^k)\|^2
        - \frac{1}{\nbiter} \sum\limits_{k=0}^{\nbiter-1} \| x^{k+1/2} - x^k\|^2.
    \end{align*}
Applying Cauchy-Schwartz inequality \eqref{eq:inner_prod} one more time,
    \begin{align*}
    2\gamma \big(\langle& F(x),  \bar x^\nbiter - x \rangle +  r(\bar x^\nbiter) - r(x) \big)
        \\\leq&
        \frac{2\| x^0 - x\|^2}{\nbiter}
        + \gamma^2 \sum\limits_{k=0}^{\nbiter-1}\| \E_{k+1/2}[g^{k}] - F(x^{k+1/2}) \|^2  + \frac{1}{\nbiter^2} \sum\limits_{k=0}^{\nbiter-1}\| x^{k+1/2} - x \|^2
        \notag\\
        &
        - 2\gamma \cdot \frac{1}{\nbiter} \sum\limits_{k=0}^{\nbiter-1} \langle g^{k} - \E_{k+1/2}[g^{k}], x^{k+1/2} - x^0 \rangle
        + \gamma^2 \| \nbiter^{-1} \sum\limits_{k=0}^{\nbiter-1} \left[g^{k} - \E_{k+1/2}[g^{k}] \right] \|^2        \notag\\
        &
        + 3\gamma^2 \cdot \frac{1}{\nbiter} \sum\limits_{k=0}^{\nbiter-1} \| B^{-1} \sum\limits_{i=1}^{B} F(x^{k}, z^{k}_i) - F(x^k)\|^2 + 3\gamma^2\cdot \frac{1}{\nbiter} \sum\limits_{k=0}^{\nbiter-1}  \| F(x^{k+1/2}) - g^{k}\|^2
        \notag\\
        &
        + 3\gamma^2 \cdot \frac{1}{\nbiter} \sum\limits_{k=0}^{\nbiter-1} \| F(x^{k+1/2}) - F(x^k)\|^2
        - \frac{1}{\nbiter} \sum\limits_{k=0}^{\nbiter-1} \| x^{k+1/2} - x^k\|^2.
    \end{align*}
Taking supermom on $x$ from $\mathcal{X}$ and then the full expectation, we get
    \begin{align*}
    2\gamma \EEE{\text{Gap}(\bar x^\nbiter)}
        \leq&
        \frac{2\max_{x \in \mathcal{X}}\| x^0 - x\|^2}{\nbiter} + \frac{1}{\nbiter^2} \sum\limits_{k=0}^{\nbiter-1} \EEE{\max_{x \in \mathcal{X}}\| x^{k+1/2} - x \|^2}
        \notag\\
        &
        - 2\gamma \cdot \frac{1}{\nbiter} \sum\limits_{k=0}^{\nbiter-1} \EEE{\langle g^{k} - \E_{k+1/2}[g^{k}], x^{k+1/2} - x^0 \rangle}
        \notag\\
        &
        + \gamma^2 \sum\limits_{k=0}^{\nbiter-1} \EEE{\| \E_{k+1/2}[g^{k}] - F(x^{k+1/2}) \|^2}  
        \notag\\
        &
        + \gamma^2 \EEE{\| \nbiter^{-1} \sum\limits_{k=0}^{\nbiter-1} \left[g^{k} - \E_{k+1/2}[g^{k}] \right] \|^2}
        \notag\\
        &
        + 3\gamma^2 \cdot \frac{1}{\nbiter} \sum\limits_{k=0}^{\nbiter-1} \EEE{\| B^{-1} \sum\limits_{i=1}^{B} F(x^{k}, z^{k}_i) - F(x^k)\|^2}
        \notag\\
        &
        + 3\gamma^2\cdot \frac{1}{\nbiter} \sum\limits_{k=0}^{\nbiter-1}  \EEE{\| F(x^{k+1/2}) - g^{k}\|^2}
        \notag\\
        &
        + 3\gamma^2 \cdot \frac{1}{\nbiter} \sum\limits_{k=0}^{\nbiter-1} \EEE{\| F(x^{k+1/2}) - F(x^k)\|^2}
        - \frac{1}{\nbiter} \sum\limits_{k=0}^{\nbiter-1} \EEE{\| x^{k+1/2} - x^k\|^2}.
    \end{align*}
One can note that 
\begin{align*}
\EEE{\langle g^{k} - \E_{k+1/2}[g^{k}], x^{k+1/2} - x^0 \rangle} 
&= 
\EEE{\EE_{k+1/2}[\langle g^{k} - \E_{k+1/2}[g^{k}], x^{k+1/2} - x^0 \rangle]} 
\\
&= \EEE{\langle \EE_{k+1/2}[g^{k} - \E_{k+1/2}[g^{k}]], x^{k+1/2} - x^0 \rangle} 
\\
&= 0,
\end{align*}
and (here we also need Cauchy-Schwartz inequality \eqref{eq:inner_prod_and_sqr})
\begin{align*}
\EEE{\|  \nbiter^{-1} \sum\limits_{k=0}^{\nbiter-1} \left[g^{k} - \E_{k+1/2}[g^{k}] \right] \|^2} 
=& 
\frac{1}{\nbiter^2} \sum\limits_{k=0}^{\nbiter-1} \EEE{\|   g^{k} - \E_{k+1/2}[g^{k}] \|^2} 
\\
&+ \frac{1}{\nbiter^2} \sum\limits_{k \neq j} \EEE{\langle g^{k} - \E_{k+1/2}[g^{k}], g^{j} - \E_{j+1/2}[g^{j}] \rangle}
\\
=& 
\frac{1}{\nbiter^2} \sum\limits_{k=0}^{\nbiter-1} \EEE{\|   g^{k} - \E_{k+1/2}[g^{k}] \|^2} 
\\
&+ \frac{2}{\nbiter^2} \sum\limits_{k > j} \EEE{\langle \EE_{k+1/2}[g^{k} - \E_{k+1/2}[g^{k}]], g^{j} - \E_{j+1/2}[g^{j}] \rangle}
\\
=& 
\frac{1}{\nbiter^2} \sum\limits_{k=0}^{\nbiter-1} \EEE{\|   g^{k} - \E_{k+1/2}[g^{k}] \|^2}
\\
\leq& 
\frac{2}{\nbiter^2} \sum\limits_{k=0}^{\nbiter-1} \EEE{\|   g^{k} - F(x^{k+1/2}) \|^2}
\\
&
+\frac{2}{\nbiter^2} \sum\limits_{k=0}^{\nbiter-1} \EEE{\|   F(x^{k+1/2}) - \E_{k+1/2}[g^{k}] \|^2}.
\end{align*}
Then, we have
    \begin{align*}
    2\gamma \EEE{\text{Gap}(\bar x^\nbiter)}
        \leq&
        \frac{2\max_{x \in \mathcal{X}}\| x^0 - x\|^2}{\nbiter}
        + \frac{1}{\nbiter^2} \sum\limits_{k=0}^{\nbiter-1} \EEE{\max_{x \in \mathcal{X}}\| x^{k+1/2} - x \|^2}
        \\
        &
        + 2\gamma^2 \sum\limits_{k=0}^{\nbiter-1} \EEE{\| \E_{k+1/2}[g^{k}] - F(x^{k+1/2}) \|^2}
        \notag\\
        &
        + 3\gamma^2 \cdot \frac{1}{\nbiter} \sum\limits_{k=0}^{\nbiter-1} \EEE{\| B^{-1} \sum\limits_{i=1}^{B} F(x^{k}, z^{k}_i) - F(x^k)\|^2} 
        \\
        &
        + 5\gamma^2\cdot \frac{1}{\nbiter} \sum\limits_{k=0}^{\nbiter-1}  \EEE{\| F(x^{k+1/2}) - g^{k}\|^2}
        \notag\\
        &
        + 3\gamma^2 \cdot \frac{1}{\nbiter} \sum\limits_{k=0}^{\nbiter-1} \EEE{\| F(x^{k+1/2}) - F(x^k)\|^2}
        - \frac{1}{\nbiter} \sum\limits_{k=0}^{\nbiter-1} \EEE{\| x^{k+1/2} - x^k\|^2}.
    \end{align*}
With \Cref{as:strong_monotone_op} and \Cref{as:boundedset}, we obtain
    \begin{align*}
    2\gamma \EEE{\text{Gap}(\bar x^\nbiter)}
        \leq&
        \frac{3 D^2}{\nbiter}
        + 2\gamma^2 \sum\limits_{k=0}^{\nbiter-1} \EEE{\| \E_{k+1/2}[g^{k}] - F(x^{k+1/2}) \|^2} 
        \notag\\
        &
        + 3\gamma^2 \cdot \frac{1}{\nbiter} \sum\limits_{k=0}^{\nbiter-1} \EEE{\| B^{-1} \sum\limits_{i=1}^{B} F(x^{k}, z^{k}_i) - F(x^k)\|^2}
        \notag\\
        &
        + 5\gamma^2\cdot \frac{1}{\nbiter} \sum\limits_{k=0}^{\nbiter-1}  \EEE{\| F(x^{k+1/2}) - g^{k}\|^2}
        \notag\\
        &
        - (1 - 3 \gamma^2 L^2)\frac{1}{\nbiter} \sum\limits_{k=0}^{\nbiter-1} \EEE{\| x^{k+1/2} - x^k\|^2}.
    \end{align*}
Using \Cref{lem:tech_markov_app} and \Cref{lem:expect_bound_grad_appendix}, we have
    \begin{align*}
    2\gamma \EEE{\text{Gap}(\bar x^\nbiter)}
        \leq&
        \frac{3 D^2}{\nbiter}
        + 2\gamma^2 C_2 \taumix^2\batchbound^{-2}B^{-2}  \sum\limits_{k=0}^{\nbiter-1} \left(\cmax^2 + \Dmax^2 \EEE{\| x^{k+1/2} - x^*\|^2} \right)
        \notag\\
        &
        + 3\gamma^2 C_{1}  \taumix B^{-1} \cdot \frac{1}{\nbiter} \sum\limits_{k=0}^{\nbiter-1} \left(\cmax^2 + \Dmax^2 \EEE{\| x^{k} - x^*\|^2}\right) 
         \notag\\
        &
        + 20\gamma^2 (C_1 \taumix B^{-1}\log_2 \batchbound + (C_1 + 1) \taumix^2 B^{-2} )\cdot \frac{1}{\nbiter} \sum\limits_{k=0}^{\nbiter-1}  \left(\cmax^2 + \Dmax^2 \EEE{\| x^{k+1/2} - x^*\|^2} \right) 
        \notag\\
        &
        - (1 - 3 \gamma^2 L^2)\frac{1}{\nbiter} \sum\limits_{k=0}^{\nbiter-1} \EEE{\| x^{k+1/2} - x^k\|^2}.
    \end{align*}
Again with \Cref{as:boundedset}, we get
    \begin{align*}
    2\gamma \EEE{\text{Gap}(\bar x^\nbiter)}
        \leq&
        \frac{3 D^2}{\nbiter}
        + 2\gamma^2 C_2 \taumix^2\batchbound^{-2}B^{-2} \nbiter  \left(\cmax^2 + \Dmax^2 D^2 \right)
        \notag\\
        &
        + 3\gamma^2 C_{1}  \taumix B^{-1} \left(\cmax^2 + \Dmax^2 D^2\right) 
         \notag\\
        &
        + 20\gamma^2 (C_1 \taumix B^{-1}\log_2 \batchbound + (C_1 + 1) \taumix^2 B^{-2} )\cdot \left(\cmax^2 + \Dmax^2 D^2 \right) 
        \notag\\
        &
        - (1 - 3 \gamma^2 L^2)\frac{1}{\nbiter} \sum\limits_{k=0}^{\nbiter-1} \EEE{\| x^{k+1/2} - x^k\|^2}.
    \end{align*}
With $M = \sqrt{\nbiter}$ and $\gamma \leq (3L)^{-1}$, one can deduce
    \begin{align*}
    2\gamma \EEE{\text{Gap}(\bar x^\nbiter)}
        \leq&
        \frac{3 D^2}{\nbiter}
        + 25\gamma^2 (C_1 \taumix B^{-1}\log_2 \batchbound + (C_1 + C_2 + 1) \taumix^2 B^{-2} )\cdot \left(\cmax^2 + \Dmax^2 D^2 \right) .
    \end{align*}
Substituting $M = \sqrt{\nbiter}$ finishes the proof.
\end{proof}

\section{Basic Facts}

\begin{lemma}[see Lemma 1.2.3 and Theorem 2.1.5 from \cite{nesterov2003introductory}]
\label{lem:smoth}
If $f$ is $L$-smooth in $\R^d$, then for any $x,y\in \R^d$
\begin{equation}
    \label{eq:smothness}
    f(x) - f(y) -  \langle \nabla f(y), x-y \rangle \leq \frac{L}{2} \| x - y\|^2.
\end{equation}
\end{lemma}

\begin{lemma}[Cauchy Schwartz inequality]
\label{lem:CS}
For any $a,b,x_1, \ldots, x_n \in \R^d$ and $c > 0$ the following inequalities hold:
\begin{equation}
  \label{eq:inner_prod}
  2\langle a,b \rangle  \leq \frac{\norm{a}^2}{c} + c
  \norm{b}^2,
\end{equation}
\begin{equation}
  \label{eq:inner_prod_and_sqr}
  \norm{a+b}^2 \leq \left(1 + \frac{1}{c}\right)\norm{a}^2 + (1+c)
  \norm{b}^2,
\end{equation}
\begin{equation}
  \label{eq:sum_sqr}
  \left\|\sum\limits_{i=1}^n x_i\right\|^2 \leq n \cdot \sum\limits_{i=1}^n \|x_i\|^2.
\end{equation}
\end{lemma}

\begin{lemma}[Jensen's inequality]
\label{lem:Jensen}
If $f$ is a convex function, then for any $n \in \mathbb{N}^*$ and $x_1, \ldots, x_n \in \R^d$ the following inequality holds:
\begin{equation}
  \label{eq:jensen}
  f\left(\frac{1}{n} \sum\limits_{i=1}^n x_i\right) \leq \frac{1}{n} \sum\limits_{i=1}^n f(x_i).
\end{equation}
\end{lemma}

\end{document}

%% file: heading.tex
\usepackage[utf8]{inputenc} 
\usepackage[T1]{fontenc}    
\usepackage{hyperref}       
\usepackage{url}            
\usepackage{booktabs}       
\usepackage{amsfonts}       
\usepackage{nicefrac}       
\usepackage{microtype}      
\usepackage{xcolor}         
\usepackage{amsmath,amsfonts,amssymb,amsthm,array}
\usepackage{comment}
\usepackage{algorithm}
\usepackage[font=small]{caption}
\usepackage{subcaption}
\usepackage{algpseudocode}
\usepackage{bm}
\usepackage{color}
\usepackage{graphicx}
\usepackage{enumitem}
\usepackage{xargs}
\usepackage{multirow}
\usepackage{makecell}
\usepackage{wrapfig}
\usepackage{lscape}
\usepackage{bbm, dsfont}
\usepackage[colorinlistoftodos,bordercolor=orange,backgroundcolor=orange!20,linecolor=orange,textsize=scriptsize]{todonotes}

\usepackage{xspace}

\usepackage{cleveref}

\usepackage{makecell}
\usepackage{caption}
\usepackage{multirow}

\usepackage{colortbl}
\definecolor{bgcolor}{rgb}{0.8,1,1}
\definecolor{bgcolor2}{rgb}{0.8,1,0.8}
\usepackage{threeparttable}

\usepackage{tcolorbox}
\usepackage{pifont}
\definecolor{mydarkgreen}{RGB}{39,130,67}
\definecolor{mydarkred}{RGB}{192,47,25}
\newcommand{\green}{\color{mydarkgreen}}
\newcommand{\red}{\color{mydarkred}}
\newcommand{\cmark}{{\green\ding{51}}}
\newcommand{\xmark}{{\red\ding{55}}}
\usepackage{subcaption}
\usepackage{wrapfig}

\allowdisplaybreaks

\newtheorem{theorem}{Theorem}
\newtheorem{proposition}{Proposition}

\newtheorem{lemma}{Lemma}
\newtheorem{corollary}{Corollary}

\newtheorem{assumption}{A}

%% file: def.tex
\newcommand{\R}{\ensuremath{\mathbb{R}}}

\newcommand{\Prob}{\ensuremath{\mathbb{P}}}
\newcommand{\EE}{\mathbb{E}}
\newcommand{\E}{\mathbb{E}}

\newcommand{\EEE}[1]{\mathbb{E}\left[#1\right]}
\def\<#1,#2>{\left\langle #1,#2 \right\rangle}

\def\KL{\operatorname{KL}}

\def\batchbound{M}
\newcommand{\indiacc}[1]{\mathbbm{1}_{\{#1\}}}
\newcommand{\tvnorm}[1]{\| #1 \|_{\operatorname{TV}}}
\newcommand{\PEcoupling}[2]{\tilde{\mathbb{E}}_{#1,#2}}
\newcommand{\PPcoupling}[2]{\tilde{\mathbb{P}}_{#1,#2}}
\newcommand{\PP}{\mathbb{P}}
\def\metricz{\mathsf{d}_{\Zset}}
\def\rset{\mathbb{R}}
\def\barf{\bar{f}}
\def\cmax{\sigma}
\def\dmax{\delta}
\def\Dmax{\Delta}
\def\nset{\ensuremath{\mathbb{N}}}
\def\nsets{\ensuremath{\mathbb{N}^*}}
\def\Zset{\mathsf{Z}}
\def\Zsigma{\mathcal{Z}}
\def\MKQ{\mathrm{Q}}
\def\MK{\mathrm{P}}
\def\Id{\mathrm{I}}
\newcommandx{\norm}[2][2=]{\Vert#1 \Vert_{{#2}}}
\def\taumix{\tau}
\newcommand{\dobrush}{\mathsf{\Delta}}
\newcommandx{\dobru}[3][1=,3=]{\dobrush_{#1}^{#3}( #2)}  
\def\nbiter{N}

%% file: neurips_2023.bbl
\begin{thebibliography}{100}

\bibitem{aamari_levrard2019}
Eddie Aamari and Cl{\'e}ment Levrard.
\newblock {Nonasymptotic rates for manifold, tangent space and curvature
  estimation}.
\newblock {\em The Annals of Statistics}, 47(1):177 -- 204, 2019.

\bibitem{alacaoglu2022stochastic}
Ahmet Alacaoglu and Yura Malitsky.
\newblock Stochastic variance reduction for variational inequality methods.
\newblock In {\em Conference on Learning Theory}, pages 778--816. PMLR, 2022.

\bibitem{aybat2019universally}
Necdet~Serhat Aybat, Alireza Fallah, Mert Gurbuzbalaban, and Asuman Ozdaglar.
\newblock A universally optimal multistage accelerated stochastic gradient
  method.
\newblock {\em Advances in neural information processing systems}, 32, 2019.

\bibitem{bach2011optimization}
Francis Bach, Rodolphe Jenatton, Julien Mairal, and Guillaume Obozinski.
\newblock Optimization with sparsity-inducing penalties.
\newblock {\em arXiv preprint arXiv:1108.0775}, 2011.

\bibitem{bach2008convex}
Francis Bach, Julien Mairal, and Jean Ponce.
\newblock Convex sparse matrix factorizations.
\newblock {\em arXiv preprint arXiv:0812.1869}, 2008.

\bibitem{beck2017first}
Amir Beck.
\newblock {\em First-order methods in optimization}.
\newblock Society for Industrial and Applied Mathematics (SIAM), 2017.

\bibitem{BenTal2009:book}
Aharon Ben-Tal, Laurent~El Ghaoui, and Arkadi Nemirovski.
\newblock {\em Robust Optimization}.
\newblock Princeton University Press, 2009.

\bibitem{beznosikov2023stochastic}
Aleksandr Beznosikov, Eduard Gorbunov, Hugo Berard, and Nicolas Loizou.
\newblock Stochastic gradient descent-ascent: Unified theory and new efficient
  methods.
\newblock In {\em International Conference on Artificial Intelligence and
  Statistics}, pages 172--235. PMLR, 2023.

\bibitem{beznosikov2020distributed}
Aleksandr Beznosikov, Valentin Samokhin, and Alexander Gasnikov.
\newblock Distributed saddle-point problems: Lower bounds, optimal and robust
  algorithms.
\newblock {\em arXiv preprint arXiv:2010.13112}, 2020.

\bibitem{bhandari2018finite}
Jalaj Bhandari, Daniel Russo, and Raghav Singal.
\newblock A finite time analysis of temporal difference learning with linear
  function approximation.
\newblock In {\em Conference on learning theory}, pages 1691--1692. PMLR, 2018.

\bibitem{chambolle2011first}
Antonin Chambolle and Thomas Pock.
\newblock A first-order primal-dual algorithm for convex problems with
  applications to imaging.
\newblock {\em Journal of mathematical imaging and vision}, 40(1):120--145,
  2011.

\bibitem{chavdarova2019reducing}
Tatjana Chavdarova, Gauthier Gidel, Fran{\c{c}}ois Fleuret, and Simon
  Lacoste-Julien.
\newblock Reducing noise in gan training with variance reduced extragradient.
\newblock {\em Advances in Neural Information Processing Systems}, 32, 2019.

\bibitem{chen2020convergence}
You-Lin Chen, Sen Na, and Mladen Kolar.
\newblock Convergence analysis of accelerated stochastic gradient descent under
  the growth condition.
\newblock {\em arXiv preprint arXiv:2006.06782}, 2020.

\bibitem{cotter2011better}
Andrew Cotter, Ohad Shamir, Nati Srebro, and Karthik Sridharan.
\newblock Better mini-batch algorithms via accelerated gradient methods.
\newblock {\em Advances in neural information processing systems}, 24, 2011.

\bibitem{daskalakis2017training}
Constantinos Daskalakis, Andrew Ilyas, Vasilis Syrgkanis, and Haoyang Zeng.
\newblock Training gans with optimism.
\newblock {\em arXiv preprint arXiv:1711.00141}, 2017.

\bibitem{devolder2011stochastic}
Olivier Devolder et~al.
\newblock Stochastic first order methods in smooth convex optimization.
\newblock Technical report, CORE, 2011.

\bibitem{dieuleveut2017harder}
Aymeric Dieuleveut, Nicolas Flammarion, and Francis Bach.
\newblock Harder, better, faster, stronger convergence rates for least-squares
  regression.
\newblock {\em The Journal of Machine Learning Research}, 18(1):3520--3570,
  2017.

\bibitem{dimakis07}
Alexandros~G. Dimakis, Soummya Kar, José M.~F. Moura, Michael~G. Rabbat, and
  Anna Scaglione.
\newblock Gossip algorithms for distributed signal processing.
\newblock {\em Proceedings of the IEEE}, 98(11):1847--1864, 2010.

\bibitem{doan23}
Thinh~T. Doan.
\newblock Finite-time analysis of markov gradient descent.
\newblock {\em IEEE Transactions on Automatic Control}, 68(4):2140--2153, 2023.

\bibitem{doan2020convergence}
Thinh~T Doan, Lam~M Nguyen, Nhan~H Pham, and Justin Romberg.
\newblock Convergence rates of accelerated markov gradient descent with
  applications in reinforcement learning.
\newblock {\em arXiv preprint arXiv:2002.02873}, 2020.

\bibitem{dorfman2022adapting}
Ron Dorfman and Kfir~Yehuda Levy.
\newblock Adapting to mixing time in stochastic optimization with markovian
  data.
\newblock In {\em International Conference on Machine Learning}, pages
  5429--5446. PMLR, 2022.

\bibitem{douc:moulines:priouret:soulier:2018}
R.~Douc, E.~Moulines, P.~Priouret, and P.~Soulier.
\newblock {\em Markov chains}.
\newblock Springer Series in Operations Research and Financial Engineering.
  Springer, 2018.

\bibitem{duchi2012ergodic}
John~C Duchi, Alekh Agarwal, Mikael Johansson, and Michael~I Jordan.
\newblock Ergodic mirror descent.
\newblock {\em SIAM Journal on Optimization}, 22(4):1549--1578, 2012.

\bibitem{moulines23Rosenthal}
Alain Durmus, Eric Moulines, Alexey Naumov, Sergey Samsonov, and Marina
  Sheshukova.
\newblock Rosenthal-type inequalities for linear statistics of markov chains.
\newblock {\em arXiv preprint arXiv:2303.05838}, 2023.

\bibitem{durmus2021stability}
Alain Durmus, Eric Moulines, Alexey Naumov, Sergey Samsonov, and Hoi-To Wai.
\newblock On the stability of random matrix product with markovian noise:
  Application to linear stochastic approximation and td learning.
\newblock In {\em Conference on Learning Theory}, pages 1711--1752. PMLR, 2021.

\bibitem{dvurechensky2016stochastic}
Pavel Dvurechensky and Alexander Gasnikov.
\newblock Stochastic intermediate gradient method for convex problems with
  stochastic inexact oracle.
\newblock {\em Journal of Optimization Theory and Applications}, 171:121--145,
  2016.

\bibitem{esser2010general}
Ernie Esser, Xiaoqun Zhang, and Tony~F Chan.
\newblock A general framework for a class of first order primal-dual algorithms
  for convex optimization in imaging science.
\newblock {\em SIAM Journal on Imaging Sciences}, 3(4):1015--1046, 2010.

\bibitem{even2023stochastic}
Mathieu Even.
\newblock Stochastic gradient descent under {M}arkovian sampling schemes.
\newblock {\em arXiv preprint arXiv:2302.14428}, 2023.

\bibitem{facchinei2007finite}
F.~Facchinei and J.S. Pang.
\newblock {\em Finite-Dimensional Variational Inequalities and Complementarity
  Problems}.
\newblock Springer Series in Operations Research and Financial Engineering.
  Springer New York, 2007.

\bibitem{gasnikov2018universal}
Alexander~Vladimirovich Gasnikov and Yu~E Nesterov.
\newblock Universal method for stochastic composite optimization problems.
\newblock {\em Computational Mathematics and Mathematical Physics}, 58:48--64,
  2018.

\bibitem{ghadimi2013stochastic}
Saeed Ghadimi and Guanghui Lan.
\newblock Stochastic first-and zeroth-order methods for nonconvex stochastic
  programming.
\newblock {\em SIAM Journal on Optimization}, 23(4):2341--2368, 2013.

\bibitem{ghadimi2016mini}
Saeed Ghadimi, Guanghui Lan, and Hongchao Zhang.
\newblock Mini-batch stochastic approximation methods for nonconvex stochastic
  composite optimization.
\newblock {\em Mathematical Programming}, 155(1-2):267--305, 2016.

\bibitem{gidel2018variational}
Gauthier Gidel, Hugo Berard, Ga{\"e}tan Vignoud, Pascal Vincent, and Simon
  Lacoste-Julien.
\newblock A variational inequality perspective on generative adversarial
  networks.
\newblock {\em arXiv preprint arXiv:1802.10551}, 2018.

\bibitem{giles08}
Michael~B. Giles.
\newblock Multilevel monte carlo path simulation.
\newblock {\em Operations Research}, 56(3):607--617, 2008.

\bibitem{glynn2014exact}
Peter~W Glynn and Chang-han Rhee.
\newblock Exact estimation for markov chain equilibrium expectations.
\newblock {\em Journal of Applied Probability}, 51(A):377--389, 2014.

\bibitem{goodfellow2014generative}
Ian Goodfellow, Jean Pouget-Abadie, Mehdi Mirza, Bing Xu, David Warde-Farley,
  Sherjil Ozair, Aaron Courville, and Yoshua Bengio.
\newblock Generative adversarial nets.
\newblock In {\em Advances in Neural Information Processing Systems (NIPS)},
  2014.

\bibitem{GoodBengCour16}
Ian~J. Goodfellow, Yoshua Bengio, and Aaron Courville.
\newblock {\em Deep Learning}.
\newblock MIT Press, Cambridge, MA, USA, 2016.
\newblock \url{http://www.deeplearningbook.org}.

\bibitem{gorbunov2022stochastic}
Eduard Gorbunov, Hugo Berard, Gauthier Gidel, and Nicolas Loizou.
\newblock Stochastic extragradient: General analysis and improved rates.
\newblock In {\em International Conference on Artificial Intelligence and
  Statistics}, pages 7865--7901. PMLR, 2022.

\bibitem{gorbunov2021near}
Eduard Gorbunov, Marina Danilova, Innokentiy Shibaev, Pavel Dvurechensky, and
  Alexander Gasnikov.
\newblock Near-optimal high probability complexity bounds for non-smooth
  stochastic optimization with heavy-tailed noise.
\newblock {\em arXiv preprint arXiv:2106.05958}, 2021.

\bibitem{han2021lower}
Yuze Han, Guangzeng Xie, and Zhihua Zhang.
\newblock Lower complexity bounds of finite-sum optimization problems: The
  results and construction.
\newblock {\em arXiv preprint arXiv:2103.08280}, 2021.

\bibitem{HarkerVIsurvey1990}
P.~T. Harker and J.-S. Pang.
\newblock Finite-dimensional variational inequality and nonlinear
  complementarity problems: a survey of theory, algorithms and applications.
\newblock {\em Mathematical programming}, 1990.

\bibitem{hsieh2019convergence}
Yu-Guan Hsieh, Franck Iutzeler, J{\'e}r{\^o}me Malick, and Panayotis
  Mertikopoulos.
\newblock On the convergence of single-call stochastic extra-gradient methods.
\newblock {\em Advances in Neural Information Processing Systems}, 32, 2019.

\bibitem{hsieh2020explore}
Yu-Guan Hsieh, Franck Iutzeler, J{\'e}r{\^o}me Malick, and Panayotis
  Mertikopoulos.
\newblock Explore aggressively, update conservatively: Stochastic extragradient
  methods with variable stepsize scaling.
\newblock {\em Advances in Neural Information Processing Systems},
  33:16223--16234, 2020.

\bibitem{hu2009accelerated}
Chonghai Hu, Weike Pan, and James Kwok.
\newblock Accelerated gradient methods for stochastic optimization and online
  learning.
\newblock {\em Advances in Neural Information Processing Systems}, 22, 2009.

\bibitem{doi:10.1137/15M1031953}
A.~N. Iusem, A.~Jofr\'{e}, R.~I. Oliveira, and P.~Thompson.
\newblock Extragradient method with variance reduction for stochastic
  variational inequalities.
\newblock {\em SIAM Journal on Optimization}, 27(2):686--724, 2017.

\bibitem{iusem2019incremental}
Alfredo~N Iusem, Alejandro Jofr{\'e}, and Philip Thompson.
\newblock Incremental constraint projection methods for monotone stochastic
  variational inequalities.
\newblock {\em Mathematics of Operations Research}, 44(1):236--263, 2019.

\bibitem{jain2018accelerating}
Prateek Jain, Sham~M Kakade, Rahul Kidambi, Praneeth Netrapalli, and Aaron
  Sidford.
\newblock Accelerating stochastic gradient descent for least squares
  regression.
\newblock In {\em Conference On Learning Theory}, pages 545--604. PMLR, 2018.

\bibitem{JMLR:v18:16-595}
Prateek Jain, Sham~M. Kakade, Rahul Kidambi, Praneeth Netrapalli, and Aaron
  Sidford.
\newblock Parallelizing stochastic gradient descent for least squares
  regression: Mini-batching, averaging, and model misspecification.
\newblock {\em Journal of Machine Learning Research}, 18(223):1--42, 2018.

\bibitem{4610024}
Houyuan Jiang and Huifu Xu.
\newblock Stochastic approximation approaches to the stochastic variational
  inequality problem.
\newblock {\em IEEE Transactions on Automatic Control}, 53(6):1462--1475, 2008.

\bibitem{Jin2020:mdp}
Yujia Jin and Aaron Sidford.
\newblock Efficiently solving {MDP}s with stochastic mirror descent.
\newblock In {\em Proceedings of the 37th International Conference on Machine
  Learning (ICML)}, volume 119, pages 4890--4900. PMLR, 2020.

\bibitem{Thorsten}
Thorsten Joachims.
\newblock A support vector method for multivariate performance measures.
\newblock pages 377--384, 01 2005.

\bibitem{juditsky2011solving}
Anatoli Juditsky, Arkadi Nemirovski, and Claire Tauvel.
\newblock Solving variational inequalities with stochastic mirror-prox
  algorithm.
\newblock {\em Stochastic Systems}, 1(1):17--58, 2011.

\bibitem{kannan2019optimal}
Aswin Kannan and Uday~V Shanbhag.
\newblock Optimal stochastic extragradient schemes for pseudomonotone
  stochastic variational inequality problems and their variants.
\newblock {\em Computational Optimization and Applications}, 74(3):779--820,
  2019.

\bibitem{kidambi2018insufficiency}
Rahul Kidambi, Praneeth Netrapalli, Prateek Jain, and Sham Kakade.
\newblock On the insufficiency of existing momentum schemes for stochastic
  optimization.
\newblock In {\em 2018 Information Theory and Applications Workshop (ITA)},
  pages 1--9. IEEE, 2018.

\bibitem{kingma2014adam}
Diederik~P Kingma and Jimmy Ba.
\newblock Adam: A method for stochastic optimization.
\newblock {\em arXiv preprint arXiv:1412.6980}, 2014.

\bibitem{koloskova2023shuffle}
Anastasia Koloskova, Nikita Doikov, Sebastian~U Stich, and Martin Jaggi.
\newblock Shuffle {SGD} is always better than {SGD}: Improved analysis of {SGD}
  with arbitrary data orders.
\newblock {\em arXiv preprint arXiv:2305.19259}, 2023.

\bibitem{korpelevich1976extragradient}
G.~M. Korpelevich.
\newblock The extragradient method for finding saddle points and other
  problems.
\newblock {\em Matecon}, 12:35--49, 1977.

\bibitem{lan2012optimal}
Guanghui Lan.
\newblock An optimal method for stochastic composite optimization.
\newblock {\em Mathematical Programming}, 133(1-2):365--397, 2012.

\bibitem{lan20}
Guanghui Lan.
\newblock {\em First-order and Stochastic Optimization Methods for Machine
  Learning}.
\newblock 01 2020.

\bibitem{pmlr-v89-liang19b}
Tengyuan Liang and James Stokes.
\newblock Interaction matters: A note on non-asymptotic local convergence of
  generative adversarial networks.
\newblock In Kamalika Chaudhuri and Masashi Sugiyama, editors, {\em Proceedings
  of the Twenty-Second International Conference on Artificial Intelligence and
  Statistics}, volume~89 of {\em Proceedings of Machine Learning Research},
  pages 907--915. PMLR, 16--18 Apr 2019.

\bibitem{lin2014smoothing}
Qihang Lin, Xi~Chen, and Javier Pena.
\newblock A smoothing stochastic gradient method for composite optimization.
\newblock {\em Optimization Methods and Software}, 29(6):1281--1301, 2014.

\bibitem{liu2018accelerating}
Chaoyue Liu and Mikhail Belkin.
\newblock Accelerating sgd with momentum for over-parameterized learning.
\newblock {\em arXiv preprint arXiv:1810.13395}, 2018.

\bibitem{lopes_sayed07}
Cassio~G. Lopes and Ali~H. Sayed.
\newblock Incremental adaptive strategies over distributed networks.
\newblock {\em IEEE Transactions on Signal Processing}, 55(8):4064--4077, 2007.

\bibitem{Madry2017:adv}
Aleksander Madry, Aleksandar Makelov, Ludwig Schmidt, Dimitris Tsipras, and
  Adrian Vladu.
\newblock Towards deep learning models resistant to adversarial attacks.
\newblock In {\em International Conference on Learning Representations (ICLR)},
  2018.

\bibitem{mao20}
Xianghui Mao, Kun Yuan, Yubin Hu, Yuantao Gu, Ali~H. Sayed, and Wotao Yin.
\newblock Walkman: A communication-efficient random-walk algorithm for
  decentralized optimization.
\newblock {\em IEEE Transactions on Signal Processing}, 68:2513--2528, 2020.

\bibitem{mertikopoulos2018optimistic}
Panayotis Mertikopoulos, Bruno Lecouat, Houssam Zenati, Chuan-Sheng Foo, Vijay
  Chandrasekhar, and Georgios Piliouras.
\newblock Optimistic mirror descent in saddle-point problems: Going the extra
  (gradient) mile.
\newblock {\em arXiv preprint arXiv:1807.02629}, 2018.

\bibitem{mishchenko2020random}
Konstantin Mishchenko, Ahmed Khaled, and Peter Richt{\'a}rik.
\newblock Random reshuffling: Simple analysis with vast improvements.
\newblock {\em Advances in Neural Information Processing Systems},
  33:17309--17320, 2020.

\bibitem{mishchenko2020revisiting}
Konstantin Mishchenko, Dmitry Kovalev, Egor Shulgin, Peter Richt{\'a}rik, and
  Yura Malitsky.
\newblock Revisiting stochastic extragradient.
\newblock In {\em International Conference on Artificial Intelligence and
  Statistics}, pages 4573--4582. PMLR, 2020.

\bibitem{deeprl}
Volodymyr Mnih, Koray Kavukcuoglu, David Silver, Andrei~A. Rusu, Joel Veness,
  Marc~G. Bellemare, Alex Graves, Martin Riedmiller, Andreas~K. Fidjeland,
  Georg Ostrovski, Stig Petersen, Charles Beattie, Amir Sadik, Ioannis
  Antonoglou, Helen King, Dharshan Kumaran, Daan Wierstra, Shane Legg, and
  Demis Hassabis.
\newblock Human-level control through deep reinforcement learning.
\newblock {\em Nature}, 518(7540):529--533, 2015.

\bibitem{bachmoulines2011}
Eric Moulines and Francis Bach.
\newblock Non-asymptotic analysis of stochastic approximation algorithms for
  machine learning.
\newblock In J.~Shawe-Taylor, R.~Zemel, P.~Bartlett, F.~Pereira, and K.Q.
  Weinberger, editors, {\em Advances in Neural Information Processing Systems},
  volume~24. Curran Associates, Inc., 2011.

\bibitem{nagaraj2020least}
Dheeraj Nagaraj, Xian Wu, Guy Bresler, Prateek Jain, and Praneeth Netrapalli.
\newblock Least squares regression with markovian data: Fundamental limits and
  algorithms.
\newblock {\em Advances in neural information processing systems},
  33:16666--16676, 2020.

\bibitem{nemirovski2004prox}
Arkadi Nemirovski.
\newblock Prox-method with rate of convergence {O(1/t)} for variational
  inequalities with lipschitz continuous monotone operators and smooth
  convex-concave saddle point problems.
\newblock {\em SIAM Journal on Optimization}, 15(1):229--251, 2004.

\bibitem{nesterov2005smooth}
Yu~Nesterov.
\newblock Smooth minimization of non-smooth functions.
\newblock {\em Mathematical programming}, 103(1):127--152, 2005.

\bibitem{doi:10.1137/100802001}
Yu. Nesterov.
\newblock Efficiency of coordinate descent methods on huge-scale optimization
  problems.
\newblock {\em SIAM Journal on Optimization}, 22(2):341--362, 2012.

\bibitem{nesterov_accelerated}
Yu.~E. Nesterov.
\newblock A method for solving the convex programming problem with convergence
  rate {$O(1/k^{2})$}.
\newblock {\em Dokl. Akad. Nauk SSSR}, 269(3):543--547, 1983.

\bibitem{nesterov2003introductory}
Yurii Nesterov.
\newblock {\em Introductory lectures on convex optimization: A basic course},
  volume~87.
\newblock Springer Science \& Business Media, 2003.

\bibitem{nesterov2007dual}
Yurii Nesterov.
\newblock Dual extrapolation and its applications to solving variational
  inequalities and related problems.
\newblock {\em Mathematical Programming}, 109(2):319--344, 2007.

\bibitem{NeumannGameTheory1944}
J.~Von Neumann and O.~Morgenstern.
\newblock {\em Theory of games and economic behavior}.
\newblock Princeton University Press, 1944.

\bibitem{Omidshafiei2017:rl}
Shayegan Omidshafiei, Jason Pazis, Christopher Amato, Jonathan~P. How, and John
  Vian.
\newblock Deep decentralized multi-task multi-agent reinforcement learning
  under partial observability.
\newblock In {\em Proceedings of the 34th International Conference on Machine
  Learning (ICML)}, volume~70, pages 2681--2690. PMLR, 2017.

\bibitem{palaniappan2016stochastic}
Balamurugan Palaniappan and Francis Bach.
\newblock Stochastic variance reduction methods for saddle-point problems.
\newblock {\em Advances in Neural Information Processing Systems}, 29, 2016.

\bibitem{paulin_spectral}
Daniel Paulin.
\newblock {Concentration inequalities for Markov chains by Marton couplings and
  spectral methods}.
\newblock {\em Electronic Journal of Probability}, 20(none):1 -- 32, 2015.

\bibitem{peng2020training}
Wei Peng, Yu-Hong Dai, Hui Zhang, and Lizhi Cheng.
\newblock Training gans with centripetal acceleration.
\newblock {\em Optimization Methods and Software}, 35(5):955--973, 2020.

\bibitem{POLYAK1963864}
B.T. Polyak.
\newblock Gradient methods for the minimisation of functionals.
\newblock {\em USSR Computational Mathematics and Mathematical Physics},
  3(4):864--878, 1963.

\bibitem{10.1214/aoms/1177729586}
Herbert Robbins and Sutton Monro.
\newblock {A Stochastic Approximation Method}.
\newblock {\em The Annals of Mathematical Statistics}, 22(3):400 -- 407, 1951.

\bibitem{schulman2015trust}
John Schulman, Sergey Levine, Pieter Abbeel, Michael Jordan, and Philipp
  Moritz.
\newblock Trust region policy optimization.
\newblock In {\em International conference on machine learning}, pages
  1889--1897. PMLR, 2015.

\bibitem{scutari2010vi}
Gesualdo Scutari, Daniel Palomar, Francisco Facchinei, and Jong-shi Pang.
\newblock Convex optimization, game theory, and variational inequality theory.
\newblock {\em Signal Processing Magazine, IEEE}, 27:35 -- 49, 06 2010.

\bibitem{srikant2019finite}
Rayadurgam Srikant and Lei Ying.
\newblock Finite-time error bounds for linear stochastic approximation andtd
  learning.
\newblock In {\em Conference on Learning Theory}, pages 2803--2830. PMLR, 2019.

\bibitem{stich2019unified}
Sebastian~U Stich.
\newblock Unified optimal analysis of the (stochastic) gradient method.
\newblock {\em arXiv preprint arXiv:1907.04232}, 2019.

\bibitem{pmlr-v162-sun22b}
Tao Sun, Dongsheng Li, and Bao Wang.
\newblock Adaptive {R}andom {W}alk {G}radient {D}escent for {D}ecentralized
  {O}ptimization.
\newblock In Kamalika Chaudhuri, Stefanie Jegelka, Le~Song, Csaba Szepesvari,
  Gang Niu, and Sivan Sabato, editors, {\em Proceedings of the 39th
  International Conference on Machine Learning}, volume 162 of {\em Proceedings
  of Machine Learning Research}, pages 20790--20809. PMLR, 17--23 Jul 2022.

\bibitem{sun2018markov}
Tao Sun, Yuejiao Sun, and Wotao Yin.
\newblock On {M}arkov chain gradient descent.
\newblock {\em Advances in neural information processing systems}, 31, 2018.

\bibitem{pmlr-v28-sutskever13}
Ilya Sutskever, James Martens, George Dahl, and Geoffrey Hinton.
\newblock On the importance of initialization and momentum in deep learning.
\newblock In Sanjoy Dasgupta and David McAllester, editors, {\em Proceedings of
  the 30th International Conference on Machine Learning}, volume~28 of {\em
  Proceedings of Machine Learning Research}, pages 1139--1147, Atlanta,
  Georgia, USA, 17--19 Jun 2013. PMLR.

\bibitem{sutton1988learning}
Richard~S Sutton.
\newblock Learning to predict by the methods of temporal differences.
\newblock {\em Machine learning}, 3:9--44, 1988.

\bibitem{Sutton1998}
Richard~S. Sutton and Andrew~G. Barto.
\newblock {\em Reinforcement Learning: An Introduction}.
\newblock The MIT Press, second edition, 2018.

\bibitem{taylor2019stochastic}
Adrien Taylor and Francis Bach.
\newblock Stochastic first-order methods: non-asymptotic and computer-aided
  analyses via potential functions.
\newblock In {\em Conference on Learning Theory}, pages 2934--2992. PMLR, 2019.

\bibitem{doi:10.1137/S0363012998338806}
Paul Tseng.
\newblock A modified forward-backward splitting method for maximal monotone
  mappings.
\newblock {\em SIAM Journal on Control and Optimization}, 38(2):431--446, 2000.

\bibitem{van2000asymptotic}
Aad~W Van~der Vaart.
\newblock {\em Asymptotic statistics}, volume~3.
\newblock Cambridge university press, 2000.

\bibitem{vaswani2019fast}
Sharan Vaswani, Francis Bach, and Mark Schmidt.
\newblock Fast and faster convergence of {SGD} for over-parameterized models
  and an accelerated perceptron.
\newblock In {\em The 22nd international conference on artificial intelligence
  and statistics}, pages 1195--1204. PMLR, 2019.

\bibitem{DBLP:journals/corr/abs-1905-09997}
Sharan Vaswani, Aaron Mishkin, Issam~H. Laradji, Mark Schmidt, Gauthier Gidel,
  and Simon Lacoste{-}Julien.
\newblock Painless stochastic gradient: Interpolation, line-search, and
  convergence rates.
\newblock {\em CoRR}, abs/1905.09997, 2019.

\bibitem{wang2022stability}
Puyu Wang, Yunwen Lei, Yiming Ying, and Ding-Xuan Zhou.
\newblock Stability and generalization for markov chain stochastic gradient
  methods.
\newblock {\em arXiv preprint arXiv:2209.08005}, 2022.

\bibitem{williams1992simple}
Ronald~J Williams.
\newblock Simple statistical gradient-following algorithms for connectionist
  reinforcement learning.
\newblock {\em Machine learning}, 8:229--256, 1992.

\bibitem{wolfer2019estimating}
Geoffrey Wolfer and Aryeh Kontorovich.
\newblock Estimating the mixing time of ergodic markov chains.
\newblock In {\em Conference on Learning Theory}, pages 3120--3159. PMLR, 2019.

\bibitem{woodworth2021even}
Blake~E Woodworth and Nathan Srebro.
\newblock An even more optimal stochastic optimization algorithm: minibatching
  and interpolation learning.
\newblock {\em Advances in Neural Information Processing Systems},
  34:7333--7345, 2021.

\bibitem{NIPS2004_64036755}
Linli Xu, James Neufeld, Bryce Larson, and Dale Schuurmans.
\newblock Maximum margin clustering.
\newblock In L.~Saul, Y.~Weiss, and L.~Bottou, editors, {\em Advances in Neural
  Information Processing Systems}, volume~17. MIT Press, 2005.

\bibitem{NEURIPS2020_0cc6928e}
Junchi Yang, Negar Kiyavash, and Niao He.
\newblock Global convergence and variance reduction for a class of
  nonconvex-nonconcave minimax problems.
\newblock In H.~Larochelle, M.~Ranzato, R.~Hadsell, M.F. Balcan, and H.~Lin,
  editors, {\em Advances in Neural Information Processing Systems}, volume~33,
  pages 1153--1165. Curran Associates, Inc., 2020.

\bibitem{Yu1997}
Bin Yu.
\newblock {\em Assouad, Fano, and Le Cam}, pages 423--435.
\newblock Springer New York, New York, NY, 1997.

\bibitem{zhang2019lower}
Junyu Zhang, Mingyi Hong, and Shuzhong Zhang.
\newblock On lower iteration complexity bounds for the saddle point problems.
\newblock {\em arXiv preprint arXiv:1912.07481}, 2019.

\end{thebibliography}
